\documentclass[10pt]{article}

\usepackage[dvipdf]{graphicx}
\usepackage[latin1]{inputenc}
\usepackage{latexsym}
\usepackage{amsmath,amssymb,amsfonts,amsthm}
\usepackage{xcolor}
\usepackage{mathabx}
\usepackage[colorlinks=true,citecolor=red,linkcolor=blue,urlcolor=blue,pdfstartview=FitH]{hyperref}
\usepackage{pdfsync}
\numberwithin{equation}{section}

\usepackage{stmaryrd}

\definecolor{rp}{rgb}{0.25, 0, 0.75}
\definecolor{dg}{rgb}{0, 0.5, 0}

\newcommand{\gw}{g_{_W}}

\newcommand{\der}{\delta}

\newcommand{\id}{\text{Id}}

\newcommand{\wt}{\widetilde}

\newcommand{\1}{\mathbf{1}}
\newcommand{\2}{\mathbf{2}}
\newcommand{\xd}{{\bf X^{2}}}

\newcommand{\Uti}{\widetilde{U}}
\newcommand{\Kti}{\widetilde{K}}
\newcommand{\Qti}{\widetilde{Q}}
\newcommand{\Rti}{\widetilde{R}}
\newcommand{\yti}{\widetilde{y}}
\newcommand{\wti}{\widetilde{w}}

\newcommand{\cb}{{\mathcal B}}
\newcommand{\cac}{{\mathcal C}}
\newcommand{\cd}{{\mathcal D}}
\newcommand{\ce}{{\mathcal E}}

\newcommand{\ch}{{\mathcal H}}

\newcommand{\ck}{{\mathcal K}}
\newcommand{\cl}{{\mathcal L}}
\newcommand{\cm}{{\mathcal M}}

\newcommand{\cn}{{\mathcal N}}

\newcommand{\cp}{{\mathcal P}}

\newcommand{\ct}{{\mathcal T}}

\newcommand{\al}{\alpha}

\newcommand{\ep}{\varepsilon}

\newcommand{\ga}{\gamma}

\newcommand{\ka}{\kappa}
\newcommand{\la}{\lambda}
\newcommand{\laa}{\Lambda}
\newcommand{\om}{\omega}
\newcommand{\oom}{\Omega}

\newcommand{\si}{\sigma}

\newcommand{\vp}{\varphi}

\newcommand{\E}{{\mathbb E}}

\newcommand{\R}{{\mathbb R}}

\newcommand{\be}{\mathbf{E}}

\newcommand{\lcl}{\left\{}
\newcommand{\rcl}{\right\}}
\newcommand{\lp}{\left(}
\newcommand{\rp}{\right)}
\newcommand{\lc}{\left[}
\newcommand{\rc}{\right]}
\newcommand{\lln}{\left|}
\newcommand{\rrn}{\right|}

\newtheorem{theorem}{Theorem}[section]

\newtheorem{prop}[theorem]{Proposition}

\newtheorem{corollary}[theorem]{Corollary}

\newtheorem{definition}[theorem]{Definition}

\newtheorem{lemma}[theorem]{Lemma}
\newtheorem{notation}[theorem]{Notation}

\newtheorem{proposition}[theorem]{Proposition}
\theoremstyle{remark}
\newtheorem{remark}[theorem]{Remark}

\newcommand{\ER}{\mathbb {R}}
\newcommand{\EN}{\mathbb {N}}

\newcommand{\PE}{\mathbb {P}}
\newcommand{\ES}{\mathbb{E}}

\newcommand{\bqn}{\begin{equation}}
\newcommand{\bqne}{\begin{equation*}}
\newcommand{\eqn}{\end{equation}}
\newcommand{\eqne}{\end{equation*}}

\newcommand{\FA}{F}
\newcommand{\AD}{A}
\usepackage[left=1in,right=1in,top=1in,bottom=1in]{geometry}

\theoremstyle{definition}
\theoremstyle{remark}

\author{Aurélien Deya\footnote{Institut Elie Cartan, Universit\'e de Lorraine, B.P. 239, 54506 Vandoeuvre-l\`es-Nancy, Cedex, France. E-mail: \texttt{aurelien.deya@univ-lorraine.fr}} ,
Fabien Panloup\footnote{LAREMA, Universit\'e d'Angers, 2, Bd Lavoisier, 49045 Angers Cedex 01, France. E-mail: \texttt{fabien.panloup@univ-angers.fr}} , Samy Tindel\footnote{Department of Mathematics, Purdue University,
150 N. University Street,  West Lafayette, IN 47907, United States. E-mail: \texttt{stindel@purdue.edu}}}

\title{Rate of convergence to equilibrium of fractional driven stochastic differential equations with rough multiplicative noise}
\begin{document}
\maketitle

\begin{abstract}
We investigate the problem of the rate of convergence to equilibrium for ergodic stochastic differential equations driven by fractional Brownian motion with Hurst parameter $H\in(1/3,1)$ and multiplicative noise component $\sigma$. When $\sigma$ is constant and for every $H\in(0,1)$, it was proved in \cite{hairer} that, under some mean-reverting assumptions, such a process converges to its equilibrium at  a rate of order $t^{-\alpha}$ where $\alpha\in(0,1)$ (depending on $H$). In \cite{fontbona-panloup}, this result has been extended to the multiplicative case when $H>1/2$. In this paper, we  obtain these types of  results in the rough setting $H\in(1/3,1/2)$. Once again, we retrieve the rate orders of the additive setting. Our methods also extend the multiplicative results of \cite{fontbona-panloup}  by deleting  the gradient assumption on the noise coefficient $\sigma$. The main theorems include some existence and uniqueness results for the invariant distribution.
\end{abstract}

\medskip
\noindent \textit{AMS classification (2010)}: 60G22, 37A25.

\section{Introduction}

Convergence to an equilibrium distribution is one of the most natural and most studied problems concerning Markov processes. This holds true in particular for diffusions processes, seen as solutions to stochastic differential equations (SDEs in the sequel) driven by a Brownian motion. More specifically, consider the $\ER^d$-valued process $(Y_t)_{t\ge0}$ solving the following SDE:
\begin{equation}\label{wiener-SDE0}
dY_t=b(Y_t)dt+ \sigma(Y_t)\, dW_t
\end{equation}
where $b:\ER^{d}\rightarrow\ER^d$, $\sigma:\ER^{d}\rightarrow \mathbb{M}_{d,d}$  are smooth enough functions, where $ \mathbb{M}_{d,d}$ is the  set of $d\times d$ real matrices, and where $W$ is a $d$-dimensional Wiener process. Assume for simplicity that $\si(x)$ is invertible for every $x\in\ER^{d}$ and that $\si^{-1}$ is a bounded function. 

In the context of equation \eqref{wiener-SDE0}, {a simple assumption which} ensures ergodicity of the process $Y$ is the following reinforcing condition on the drift $b$ (see Hypothesis (H2) below for further details): 
There exist $C_1,C_2 >0$ such that for every $v\in \R^d$, one has
\begin{equation}\label{eq:hyp-drift-intro}
\langle v,b(v)\rangle \leq C_1-C_2 \lVert v\rVert^2 \quad .
\end{equation}
Under condition \eqref{eq:hyp-drift-intro} ({and the non-degeneracy of $\sigma$}), exponential convergence of the probability law $\cl(Y_{t})$ to a unique invariant measure $\mu$ in total variation is a classical fact, and can be {mainly} obtained via two different methods:

\noindent
\emph{(i) Functional inequalities.} Starting from Poincaré type inequalities (or further refinements) for the solution of \eqref{wiener-SDE0}, and invoking Dirichlet form  techniques, exponential and sub-exponential rates of convergence are obtained e.g in \cite{bakry,rockner}.

\noindent
{\emph{(ii) Lyapunov/Coupling techniques.} In these methods (see $e.g.$ \cite{DownMeynTweedie}), the idea is to try to  stick some solutions of \eqref{wiener-SDE0} (in an exponential time) with the following strategy:   taking advantage of the Lyapunov assumption~\eqref{eq:hyp-drift-intro} (which can be strongly alleviated in the context of \eqref{wiener-SDE0})  leads to some exponential bounds on the return-time of the (coupled) process into compact subsets of $\ER^d$ (or more generally \textit{petite sets}). Then, classical coupling techniques (involving the non-degeneracy of $\sigma$)  allow to attempt the sticking of the paths when  being in the compact subset.}

\noindent
Notice that in the setting of equation \eqref{wiener-SDE0}, the convergence analysis relies heavily on the Markov property for $Y$, or equivalently on the semi-group property for the transition probability. It also hinges on the irreducibility of $Y$, which can be seen as a non-degeneracy condition on the noisy part of the equation. Finally, observe that the first approach generally leads to sharper exponents but may require stronger assumptions.

Convergence to equilibrium being a relatively well understood phenomenon for equations li\-ke~\eqref{wiener-SDE0}, recent developments in ergodic theory for stochastic equations have focused on deviations from the irreducible Markov setting. The reference \cite{hairer-mattingly-scheutzow} handles for instance infinite dimensional situations where only asymptotic couplings of the process (starting from different initial conditions) are available. Let us also mention \cite{hairer-mattingly}, about a situation where the strong Feller property is fulfilled as $t\to\infty$, due to the degeneracy of the noise.

The current contribution is more directly related to another line of investigation, which aims at handling cases deviating from the fundamental Markov assumption. A general setting for this kind of situation is provided in the landmark of random dynamical systems~\cite{Arnold98,crauel,GKN09}. However, the type of information one can retrieve with these techniques seldom include rates of convergence to an equilibrium measure. Alternatively, one can also consider differential systems driven by a fractional Brownian motion (fBm) as a canonical example on which  non standard Markovian approaches to convergence can be elaborated. This point of view is ours, and is justified by the fact that fBm is widely used in applications (see e.g~\cite{Gua06,Odde-al96,Jeon-al11,Kou08}), and also by the fact that fBm can be seen as one of the simplest processes exhibiting long range dependence.

In this paper, we are thus concerned by the long time behavior of an equation which is similar to~\eqref{wiener-SDE0}, except for the fact that the noisy input is a fractional Brownian motion. Specifically, we consider the following SDE:
\begin{equation}\label{fractionalSDE0}
dY_t=b(Y_t)dt+ \sigma(Y_t)\, dX_t,
\end{equation}
where the coefficients $b$ and $\sigma$ satisfy the same assumptions as above (in particular relation \eqref{eq:hyp-drift-intro}), and where $(X_t)_{t\ge0}$ is  a $d$-dimensional $H$-fBm  with Hurst parameter $H\in(\frac{1}{3},1)$. Notice that in the case $ H>\frac{1}{2}$  equation~\eqref{fractionalSDE0} makes sense owing to Young integration techniques, whereas the case $H\in(1/3,1/2)$ requires elements of rough paths theory (see Section \ref{subsec:classic-topos}).

The study of  ergodic properties for fractional SDEs (under the stability assumption \eqref{eq:hyp-drift-intro}) has been undertaken by Hairer \cite{hairer},  Hairer and Ohashi \cite{hairer2}, and by Hairer and Pillai \cite{hairer-pillai}, respectively in the  additive noise, multiplicative noise with $H>1/2$ and multiplicative hypoelliptic noise with $H\in(1/3,1/2)$.  Except \cite{hairer} which also deals with rate of convergence to equilibrium, these papers mainly focus on a way to define stationary solutions, and on extending tools of the ergodic Markovian theory to the fBm setting. In particular, criteria for uniqueness of the invariant distribution are proved in increasingly demanding settings. Let us also mention the references \cite{cohen-panloup,cohen-panloup-tindel} for some results on approximations of stationary solutions. In all those articles, the Markovian formalism is based on  the Mandelbrot-Van Ness representation of the fractional Brownian motion, namely:
\bqn\label{eq:mandel}
X_t=\alpha_H\int_{-\infty}^{0} (-r)^{H-\frac{1}{2}} \left(dW_{r+t}-dW_r\right),\quad t\ge0,
\eqn
where $(W_t)_{t\in\ER}$ is a two-sided $\ER^d$-valued Brownian motion and $\alpha_H$ is a normalization coefficient depending on $H$. It is then shown that
$(Y_t, (X_{s+t})_{s\le 0})_{t\ge0}$ can be realized through a Feller transformation $({\cal Q}_t)_{t\ge0}$ whose definition is recalled below (see Section \ref{markov-structure}). In particular,
an initial distribution of the dynamical system $(Y,X)$ is a distribution $\mu_0$ on $\ER^d\times{\cal W}_{-}$, where ${\cal W}_{-}$ is an appropriate Hölder space (see Section \ref{subsec:inva-distri} for more details). Rephrased in more probabilistic terms, an initial distribution is the distribution of  a couple $(Y_0, (X_s)_{s\le0})$ where $(X_s)_{s\le0}$ is an $\ER^d$-valued fBm on $(-\infty,0]$.
Then, such an initial distribution is called an
invariant distribution if it is invariant by the transformation ${\cal Q}_t$ for every $t\ge0$. As mentioned above, the uniqueness of such an invariant distribution is investigated in \cite{hairer,hairer2,hairer-pillai}.

Let us now go back to our original question concerning the rate of convergence to equilibrium, which is obviously a natural problem when uniqueness holds for the invariant distribution. This problem has been first considered in \cite{hairer}, for equation \eqref{fractionalSDE0} with an additive noise. In this context it is shown that the law of $Y_{t}$ converges in total variation to the stationary regime, with a rate upper-bounded by  $ C_\varepsilon t^{-(\alpha-\varepsilon)}$ for any $\varepsilon>0$,  
where 
\bqn\label{eq:rate-cvgce}
\alpha=\begin{cases}\frac{1}{8}&\textnormal{if $H\in (\frac{1}{4},1)\backslash\left\{\frac{1}{2}\right\}$}\\
 H(1-2H)&\textnormal{if $H\in(0,\frac{1}{4}]$.}
\end{cases}
\eqn
The upper bound above is believed to be non-optimal, though its sub-exponential character can be interpreted as an effect of the non-Markovianity of the fBm $X$. Referring to our previous discussion on methods to achieve rates of convergence, functional inequalities tools are ruled out in the fBm setting, due to the absence of a real semi-group related to equation \eqref{fractionalSDE0}. The method chosen in \cite{hairer} is thus based on coupling of solutions starting from different initial conditions. More specifically, the problem is reduced to a coupling between two paths starting from some initial conditions $\mu_0$ and $\mu$, where the second one denotes an invariant distribution of $({\cal Q}_t)_{t\ge0}$. The main step  consists (classically) in finding a stopping time ${\tau_{\infty}}$
such that $(Y_{t+{\tau_{\infty}}}^{\mu_0})_{t\ge0}= (Y_{t+{\tau_{\infty}}}^{\mu})_{t\ge0}$. The rate of convergence in total variation is then obtained by means of an accurate bound on 
$\PE({\tau_{\infty}}>t)$, $t\ge0$.

Within the general framework recalled above, the next challenge consists in extending the rate~\eqref{eq:rate-cvgce} to multiplicative noises.
This has been achieved in Fontbona and Panloup \cite{fontbona-panloup}, where the order of convergence~\eqref{eq:rate-cvgce} is obtained in the case  $H>\frac{1}{2}$, with the additional assumption that the diffusion component $\si$ is invertible and satisfies the following gradient type assumption: its inverse $\si^{-1}$ is a Jacobian matrix. Our paper has thus to be seen as an improvement of \cite{fontbona-panloup} in two different directions:

\noindent\emph{(i)} We get rid of the gradient type hypothesis assumed in \cite{fontbona-panloup}, which extends the scope of application of our result.\smallskip

\noindent\emph{(ii)} We treat the case of an irregular fBm, with Hurst parameter $H\in(1/3,1/2)$, which means that equation~\eqref{fractionalSDE0} has to be understood in the rough paths sense. Our main goal (see Theorem \ref{theo:principal} for a precise statement) is then to obtain the rate of convergence \eqref{eq:rate-cvgce} under those general conditions on $\si$ and in the rough case.

\noindent
One point should be made clear right now: the techniques displayed in this paper can cover both the case $H\in (1/3,1/2)$ (as mentionned in point \emph{(ii)} above) and the case $H>1/2$ (thus extending the results of \cite{fontbona-panloup} beyond the gradient type assumption, as reported in point \emph{(i)}). This being said, for the sake of conciseness, we shall only express our analysis within the rough setting, that is when $H\in (1/3,1/2)$, and therefore leave to the reader the details of the extension to the  (simpler) Young situation $H>1/2$ (see Remark \ref{remarque:cashsup12} for a few additional comments on this topic).

In order to achieve our claimed rate of convergence, we shall implement the coupling strategy alluded to above. Let us briefly recall how this coupling strategy is divided in 3 steps. As a preliminary step, one  waits that the two paths (starting respectively from $\mu$ and $\mu_{0}$) get close. This is ensured by the reinforcing condition \eqref{eq:hyp-drift-intro}. Then, at each trial, the coupling attempt is divided in two steps. First, one tries in Step~1 to cluster the positions on an interval of length $1$. Then, in Step 2, one tries to ensure that the paths stay clustered until $+\infty$.  Actually, oppositely to the Markovian case where the paths stay naturally together after a clustering (by putting the same noise on each coordinate), the main difficulty here is that, due to the memory, staying together is costly. In other words, this property can be guaranteed only with the help of a  non trivial coupling of the noises.  If one of the two previous steps fails, a new attempt will be made after a (long) waiting time which is called Step $3$. During this step one waits again for the paths to get close, but one also expects the memory of the coupling cost to vanish sufficiently.

In our general rough setting with non constant coefficient $\si$, the implementation of the coupling strategy requires some non trivial adaptations of the general scheme. Let us highlight our main contributions in order to achieve the desired convergence rate:
\begin{enumerate}
\item 
The binding preliminary step relies on Lyapunov type properties of the differential equation \eqref{fractionalSDE0}. We will invoke here some rough paths techniques based on discretization schemes.

\item
In the additive case, two paths driven by the same fBm differ from a drift term, which leads to a straightforward way of sticking the paths in Step 1. We are no longer able to use this trick here, and our coupling is based on a linearization of equation \eqref{fractionalSDE0}. The analysis of such a linearization turns out to be demanding, and is one of our main efforts in this article.

\item
The different trials we have to make in a context where nontrivial correlations occur force us to consider conditioning procedures. For these conditionings, we have chosen to decompose $X$ into a Liouville fBm plus a smooth process with singularity at $t=0$. The rough path formalism has to be adapted to this new setting. 

\end{enumerate}
Those steps are sometimes delicate, and will be detailed in the remainder of the article.

{Our paper is organized as follows:
In  Section \ref{section2} we detail our assumptions and state our main result, namely Theorem \ref{theo:principal}, which provides existence and uniqueness of an invariant distribution for the rough equation \eqref{fractionalSDE0}, as well as a rate of convergence towards this distribution. The Markov setting for this equation, as well as Lyapunov type inequalities, are given in Section \ref{sec:exist-invariant}, leading to the proof of the existence statement. Our global strategy to get uniqueness and the convergence rate is explained at Section \ref{sec:sketch-strategy}. The end of the proof of Theorem \ref{theo:principal} can thus be found in Section \ref{sec:prooftheoprinc}, slightly anticipating the technical results of the subsequent sections. The singular rough equations needed for the first step of the coupling (the \emph{clustering} or \emph{hitting} step) are detailed in Section \ref{sec:rougheqinholder}, and then applied at Section \ref{sec:roughcoupling} to some specific \emph{hitting} system. Eventually, the controls associated with Step 2 and Step 3 of the procedure are exhibited in Sections \ref{sec:step2} and \ref{sec:Kadmis}, respectively.}

\section{Setting and main result}\label{section2}

We recall here the minimal amount of rough paths considerations allowing to define and solve equation \eqref{fractionalSDE0} driven by a fBm with Hurst parameter $H>1/3$. These preliminaries will be presented using terminology taken from the so-called algebraic integration theory, which is a variant of the rough paths theory introduced in~\cite{Gu} (see also \cite{friz-hairer}). Then we shall state precisely the main result of this article.

\subsection{Hölder spaces, rough paths and rough differential equations}\label{subsec:classic-topos}

For an arbitrary real interval $I$, a vector space $V$ and an integer $k\ge1$, we denote by $\cac_k(I;V)$ the set of functions $g : I^{k} \to
V$ such that $g_{t_1 \cdots t_{k}} = 0$ whenever $t_i = t_{i+1}$ for
some $i\le k-1$. Such a function is called a $(k-1)$-increment.
Then, for every $f\in\cac_1(I;V)$ and $g\in\cac_2(I;V)$, we successively define
\begin{equation*}
  (\der f)_{st} := f_t - f_s \quad \text{and} \quad
(\der g)_{sut} = g_{st}-g_{su}-g_{ut} \quad ,
\end{equation*}
for any $s<u<t\in I$. Besides, throughout the paper, we will use the notation $(\delta\ct)_{st}=t-s$ for any $s<t$.

\smallskip

Our analysis will rely on some regularity considerations related to H\"older spaces. We thus start by introducing H\"older type norms for 1-increments: for every $f \in \cac_2(I;V)$, we set
\begin{align*} 
\cn[f;\cac_2^\mu(I;V)]:= 
\sup_{s,t\in I}\frac{\lVert f_{st}\rVert}{|t-s|^\mu},
\quad\text{and}\quad
\cac_2^\mu(I;V)=\lcl f \in \cac_2(I;V);\, \cn[f;\cac_2^\mu(I;V)]<\infty  \rcl.
\end{align*}
Observe now that the usual H\"older spaces $\cac_1^\mu(I;V)$  are determined in the following way: for a continuous function $f\in\cac_1(I;V)$, define
\begin{equation*} 
\cn[f;\cac_1^\mu(I;V)]=\cn[\delta f;\cac_2^\mu(I;V)],
\quad\text{and}\quad
\cac_1^\mu(I;V)=\lcl f \in \cac_1(I;V);\, \cn[f;\cac_1^\mu(I;V)]<\infty  \rcl.
\end{equation*}
We shall also use the supremum norm on spaces $\cac_k(I;V)$, which will be denoted by $\cn[\cdot \,;\cac_k^0(I;V)]$. Notice that when the context is clear, we will simply write $\cac^\mu_k(I)$ for $\cac^\mu_k(I;V)$.

The rough path theory can be seen as a differential calculus with respect to a H\"older continuous noise $x$, under a set of abstract assumptions. These assumptions are summarized in the following definition.

\begin{definition}\label{defi-rough-path}
Let $\ga$ be a constant greater than $1/3$ and consider a $\R^d$-valued $\ga$-H\"older path $x$ on some fixed interval $[0,T]$. We call a \emph{L\'evy area above $x$} any two-index map $\mathbf{x}^{\2}\in\cac_2^{2\ga}([0,T];\R^{d, d})$, which satisfies, for all $s<u<t\in [0,1]$ and all $i,j\in \{1,\ldots,d\}$,
\begin{equation}\label{levy-area-def}
  \der\mathbf{x}^{\mathbf{2};ij}_{sut}=\der x^{i}_{su} \, \der x^{j}_{ut} \quad \text{and}\quad
\mathbf{x}^{\mathbf{2};ij}_{st}+\mathbf{x}^{\mathbf{2};ji}_{st}=\der x^{i}_{st} \, \der x^{j}_{st} \ .
\end{equation}
The couple $\mathbf{x}:=(x,\mathbf{x}^{\2})$ is then called a \emph{$\ga$-rough path} above $x$, and we will use the short notation
$$\| \mathbf{x}\|_{\ga;I}:=\cn[ x;\cac_1^\ga(I;\R^d)]+\cn[ \mathbf{x}^{\2};\cac_2^{2\ga}(I;\R^{d, d})] \quad ,$$
for any interval $I\subset [0,T]$.
\end{definition}

When the rough path $\mathbf{x}$ can be approximated by smooth functions, one talks about a canonical lift, whose precise definition is given below.

\begin{definition}\label{defi-canonic-rp}
Given a path $x\in \cac_1^{\ga}([0,1];\R^d)$, we denote by $x^n=x^{\mathcal{P}_n}$ the sequence of (piecewise) smooth paths obtained through the linear interpolation of $x$ along the dyadic partition $\mathcal{P}_n$ of $[0,1]$. Then we will say that $x$ can be \emph{canonically lifted into a rough path} if there exists a $\ga$-rough path $\mathbf{x}:=(x,\mathbf{x}^{\2})$ above $x$ such that the sequence $\mathbf{x}^n:=(x^n,\mathbf{x}^{\mathbf{2},n})$ defined by
\begin{equation}\label{approxi-levy-area}
\mathbf{x}^{\mathbf{2},n}_{st}:=\int_s^t (\der x^n)_{su} \otimes dx^n_u
\end{equation}
 converges to $\mathbf{x}$ with respect to the norm 
$$\| \mathbf{x}\|_{(0,\ga');[0,1]}:=\cn[ x;\cac_1^0([0,1];\R^d)]+\cn[ x;\cac_1^{\ga'}([0,1];\R^d)]+\cn[ \mathbf{x}^{\2};\cac_2^{2\ga'}([0,1];\R^{d, d})] \quad ,$$
for every $0<\ga'<\ga$. In this case, we will also denote this (necessarily unique) limit $\mathbf{x}$ as $\mathfrak{L}(x)$.
\end{definition}

We finally give the definition of solution to a noisy differential equation, such as our main object of interest~\eqref{fractionalSDE0}. We are adopting here Davie's point of view (see \cite{Da07}). Namely, we characterize the solution $y$ by a Taylor expansion up to a remainder term whose H\"older regularity is strictly greater than 1.

\begin{definition}\label{def:sol-davie}$\mathbf{(Davie)}$
Let $\mathbf{x}:=(x,\mathbf{x}^{\2})$ be a $\ga$-rough path. Then, for all smooth vector fields
$$b:\R^d \to \R^d \quad \text{and} \quad \sigma:\R^d \to \mathcal{L}(\R^n;\R^d) \ , $$
we call $y\in \cac_1^\ga(I;\R^d)$ a solution (on $I$) of the equation
\begin{equation}\label{general-equation-davie}
dy_t =b(y_t) \, dt+\sigma(y_t) \, d\mathbf{x}_t \quad , \quad y_{t_0}=a \ ,
\end{equation}
if the two-parameter path $R^y$ defined as
$$R^y_{st}:=(\der y)_{st}-b(y_s) \, (\delta\ct)_{st}-\sigma_j(y_s) \, (\der x^j)_{st}-(D\sigma_j \cdot \sigma_k)(y_s) \, \mathbf{x}^{\mathbf{2},jk}_{st} \ $$
belongs to $\cac_{2}^\mu(I;\R^d)$, for some parameter $\mu>1$. Here, the notation $D\sigma_j \cdot \sigma_k$ stands for
\begin{equation}\label{nota-secd-order}
(D\sigma_i \cdot \sigma_k)(v):=(D\sigma_j)(v)(\sigma_k(v)) \ , \ \text{for every} \ v \in \R^d \ .
\end{equation}
\end{definition}

Applications of the abstract rough paths setting to a fractional Brownian motion $X$ depends on a proper construction of the L\'evy area $\xd$. The reader is referred to \cite[Chapter 15]{FV-bk} for a complete review of the methods enabling this construction. It can be summarized in the following way:

\begin{proposition}\label{prop:levy-fbm}
Let $1/3<H<1/2$ be a fixed Hurst parameter. Then the fBm $X$ belongs almost surely to any space $\cac_1^{\ga}$ for $\ga<H$, and can be lifted as a canonical rough path according to Definitions~\ref{defi-rough-path} and~\ref{defi-canonic-rp}. Furthermore, for any $0\le s<t\le T$, the random variable $\mathbf{X}^{\2}_{st}$ satisfies the following inequality:
\begin{equation*}
\be\lc  \lln\mathbf{X}^{\2}_{st}\rrn^{p}\rc \le c_p \, (t-s)^{2Hp}, \quad p\ge 1.
\end{equation*}
\end{proposition}

As we shall see in the next section, Proposition \ref{prop:levy-fbm} will allow us to solve equation~\eqref{fractionalSDE0} under reasonable assumptions on the coefficients $b$ and $\si$.

\subsection{Assumptions and Main Result}\label{subsec:assump-main-results}

Having defined the notion of solution to equation \eqref{fractionalSDE0}, we can now proceed to a description of our main result. We first have to introduce a set of hypothesis on $b$ and $\si$, beginning with a boundedness assumption.

\smallskip

\noindent \textbf{Hypothesis (H1):}  $b:\R^d \to \R^d$ and $\si:\R^d \to \mathcal{L}(\R^d,\R^d)$ are smooth vector fields such that
\begin{equation}
\sup_{v\in \R^d} \lVert (D^{(1)}b)(v)\rVert \ < \ \infty\ , \quad \text{and for every} \ \ell \geq 0\ , \ \sup_{v\in \R^d} \lVert (D^{(\ell)}\si)(v)\rVert \ < \ \infty \quad  .
\end{equation}

\smallskip

The second hypothesis is the Lyapunov-type assumption alluded to in the introduction, which is classically needed for the existence of an invariant distribution.

\smallskip

\noindent \textbf{Hypothesis (H2):} There exist $C_1,C_2 >0$ such that for every $v\in \R^d$, one has
\begin{equation}
\langle v,b(v)\rangle \leq C_1-C_2 \lVert v\rVert^2 \quad .
\end{equation}

Finally, one needs a non-degeneracy assumption on $\sigma$.

\smallskip

\noindent \textbf{Hypothesis (H3):} For every $x\in\ER^d$, $\sigma(x)$ is invertible  and 
\begin{equation}\label{eq:hyp-bnd-inverse-sigma}
\sup_{x\in\ER^d}\|\sigma(x)^{-1}\|<+\infty.
\end{equation}

We are now in a position to state our main result. One denotes by ${\cal L}((Y_{t}^{\mu_0})_{t\ge0})$ the distribution of the process $Y$ on ${\cal C}([0,+\infty),\ER^d)$ starting from a (generalized) initial condition $\mu_0$ (see Subsection \ref{subsec:inva-distri} below 
 for detailed definitions of initial condition and invariant distribution). We also denote by $\bar{\cal Q}\mu$ the distribution of the stationary solution (starting from an invariant distribution $\mu$). The distribution $\bar{\mu}_0(dx)$ stands for the first marginal of $\mu_0(dx,dw)$. Finally, the total variation norm is classically denoted by $\|\,.\,\|_{TV}$.
 
\begin{theorem}\label{theo:principal} 
Let $H\in(1/3,1/2)$, and  assume  ${\bf (H1)}$, ${\bf (H2)}$, ${\bf (H3)}$ hold true. Then:

\noindent
\emph{(i)} 
There exists a unique solution of equation \eqref{fractionalSDE0} in the sense of Definition~\ref{def:sol-davie}.

\noindent
\emph{(ii)} 
Existence and uniqueness hold for the invariant distribution $\mu$.

\noindent
\emph{(iii)} 
Let $\mu_0$ be an initial distribution such that there exists $r>0$ satisfying { $\int |x|^r \bar{\mu}_0(dx)<\infty$}. Then for each $\varepsilon>0$ there exists $C_{\varepsilon}>0$ such that 
\begin{equation}\label{main-result:rate-of-conv}
\| {\cal L}((Y_{t+s}^{\mu_0})_{s\ge0})-\bar{\cal Q}\mu\|_{TV}\le C_\varepsilon t^{-(\frac{1}{8}-\varepsilon)} \ .
\end{equation}
In particular, 
$$
\| {\cal L}(Y_{t}^{\mu_0})-\bar{\mu}\|_{TV}\le C_\varepsilon t^{-(\frac{1}{8}-\varepsilon)} \ .
$$
where $\bar{\mu}$ denotes the first marginal of $\mu$.
\end{theorem}

\begin{remark}
Item (i) in Theorem \ref{theo:principal} is classical in rough path theory, since Proposition~\ref{prop:levy-fbm} holds true for our fBm $X$. We refer to \cite{FV-bk} for the general theory of differential equations driven by a rough path. We prove {existence of the invariant distribution} below in Theorem \ref{theo:lyapou} and Corollary \ref{coro:lyapou-fbm}.  It is worth noting that even though this type of result is classical, its proof is highly technical in our rough and non-Markovian context. The main part of our work is then obviously to prove item (iii), which in turns implies uniqueness and achieves the proof of (ii).
Also notice that the reinforcing assumption ${\bf (H2)}$ is fundamental for both the Lyapunov and the  coupling steps in our proofs.
\end{remark}

\begin{remark}\label{remarque:cashsup12} As mentioned before, when $H>1/2$, Theorem \ref{theo:principal} has already be shown in \cite{fontbona-panloup}. However, an additional gradient type assumption on $\sigma$ was needed therein, that is: $\sigma^{-1}$ is the Jacobian matrix of a function $h:\ER^d\rightarrow\ER^d$. Up to slight adaptations (involving in particular  the non-integrability of $u\rightarrow u^{-H-\frac{1}{2}}$ when $H>1/2$), the proof developed in this paper  when $H\in(1/3,1/2)$ (especially in Step 1) extends to the case $H>1/2$ (and does not require the gradient assumption). In other words, the above result is still true when $H>1/2$. For the sake of simplicity, we however choose  to  only consider the real new case $H<1/2$ in the sequel.  
\end{remark}

\section{Existence of invariant distribution}\label{sec:exist-invariant}
The main result of this section is  Theorem \ref{theo:lyapou} where we establish a new Lyapunov property for rough equations and deduce that existence holds for the invariant distribution under
$\mathbf{(H1)}$ and $\mathbf{(H2)}$. Before, we need to recall some background about ergodic theory for rough equations. We assume that $H<1/2$.

\subsection{Markovian structure and invariant distribution}
\subsubsection{Background on the Markov structure above the solutions}\label{markov-structure}

As shown in \cite{hairer-pillai} (going back to \cite{hairer} and \cite{hairer2}), the system~\eqref{fractionalSDE0} can be endowed with a Markovian structure. Let us briefly recall
the construction. The starting point is to build an appropriate Hölder space on which $(B_t^H)_{t\in\ER}$ can be realized through a Markov transformation. Let ${\cal C}_0^\infty(\ER_{-})$ be the space of ${\cal C}^\infty$-functions $w$, with compact support on $\ER_{-}$ and with values in $\ER^d$, satisfying $w(0)=0$. Let ${\cal W}_\gamma$ denote the Hölder-type space being the (Polish) closure  of  
${\cal C}_0^\infty(\ER_{-})$ for the norm $\|\,.\,\|_{_{{\cal W}_\gamma}}$ defined by
\begin{equation}\label{espace:wgamma}
\|x\|_{_{{\cal W}_\gamma}}:=\sup_{s,t\in\ER_{-}}\frac{|x(t)-x(s)|}{|t-s|^\gamma(1+|t|+|s|)^{\frac{1}{2}}}.
\end{equation}

For any $\gamma\in(1/3,H)$, there exists a probability $\PE_{-}$ on ${\cal W}_\gamma$ such that the canonical process is a standard $d$-dimensional $H$-fBm indexed by $\ER_{-}$.
In the following, we set ${\cal W}_{-}:={\cal W}_\gamma$ and consider $x_{-}\in{\cal W}_{-}$. Set  ${\cal W}_+:={\cal D}_g^{0,\gamma}([0,1])$ the closure of ${\cal C}_0^\infty([0,1])$ with respect to the norm $\|\cdot\|_{\gamma;[0,1]}$. Then, with the help of operators related to the Mandelbrot representation (see \cite{hairer-pillai} for more precise statements), one can define a Feller transition kernel $\hat{\cal P}$ on ${\cal W}_{-}\times{\cal W}_+$ such that (with a slight abuse of notation), $\hat{\cal P}(x_{-},d{\bf x}_{+})=\PE(({\bf B}_t^H)_{t\in[0,1]}\in d{\bf x}_{+}| (B_t^H)_{t\in\ER_{-}}=x_{-})$.  Then, denoting by 
${\cal W}:={\cal W}_{-}\times {\cal W}_+$, $\Pi:{\cal W}\mapsto {\cal C}((-\infty,1],\ER^d)$ the map that concatenates $x_{-}$ with the path component $x_{+}$ of ${\bf x}_+$, and $\PE$ the probability measure on ${\cal W}$ defined by $\PE (dx_{-}\times dx_+):=\PE_-(dx_{-})\hat{\cal P}(x_{-},dx_+)$, $\Pi^*\PE$ corresponds to the law of $(B_t^H)_{t\in(-\infty,1]}$ on ${\cal C}((-\infty,1],\ER^d)$.  Denoting by $\Theta$ the $-1$-time shift from  ${\cal C}((-\infty,1],\ER^d)$ to ${\cal C}((-\infty,0],\ER^d)$, the previous construction implies that a two-sided fBm (on $\ER^d$) can be realized through the (discrete-time) Feller Markov transition kernel ${\cal P}$ on ${\cal W}$ defined by
$${\cal P}(x,.):=\delta_{\Theta(x)}\otimes \hat{{\cal P}}(\Theta(x),.).$$
The triplet $({\cal W},\PE,{\cal P})$ is called the noise space. Then, for a given initial condition $z$ and a given realization $x=(x_{-},{\bf x}_{+})$ of the driving noise, we denote by $({\Phi}_t(z,{\bf x}_+))_{t\in[0,1]}$, the unique solution to \eqref{fractionalSDE0} with initial condition $z$. Owing to \cite{FV-bk} and Assumption $\mathbf{(H1)}$, $(z,{\bf x}_+)\mapsto {\Phi}_t(z,{\bf x}_+)$ is continuous  on $\ER^d\times{\cal W}_+$. It follows that the solution to \eqref{fractionalSDE0} can be viewed as a Feller Markov process on $\ER^d\times{\cal W}$ with transition kernel ${\cal Q}$ defined by : ${\cal Q}(z,x,.):=\Psi_{z}^*{\cal P}(x,.)$ where $\Psi_{z}(x):=(\Phi_1(z,{\bf x}_+),x)$.
\begin{remark} 
In \cite{hairer}, the construction of the Markov structure is directly realized  with the underlying Wiener process. Note that such a  construction would be closer to the coupling viewpoint which is  introduced below.\smallskip

\noindent The reader can observe that the above construction only ensures the Markovian structure above the discrete-time process $(X_n)_{n\in\EN}$ and not for the whole process $(X_t)_{t\ge0}$.  However, an adaptation of the previous strategy leads to the construction of  a Feller Markov semi-group  $({\cal Q}_t)_{t\ge0}$ above $(Y_t)_{t\ge0}$ (on $\ER^d\times {\cal W}_\gamma$).  
\end{remark}   
\subsubsection{Invariant distribution}\label{subsec:inva-distri}
Following \cite{hairer2}, a probability $\mu$ on $\ER^d\times {\cal W}_{-}$ is called a \textit{generalized initial condition}  if $\Pi_{{\cal W}_{-}}\mu=\PE_{-}$ (defined in the previous section).
\begin{definition} Let $\nu$ be a generalized initial condition.  We say that $\nu$ is an invariant distribution for $(Y_t)_{t\ge0}$ if 
for every $t\ge0$, $\nu{\cal Q}_t=\nu$.
\end{definition}

\begin{definition} We say that $V:\ER^d\mapsto\ER$ is a Lyapunov function for ${\cal Q}$ if $V$ is continuous and positive, if $\lim_{|x|\rightarrow+\infty} V(x)=+\infty$ and if there exist $C>0$ and $\rho\in(0,1)$ such that for every $t\in[0,1]$,
\begin{equation}\label{lyapou-ineq}
\int V(x)(\mu{\cal Q}_t)(dx,dw)\le C+\rho\int V(x)\mu(dx,dw)
\end{equation}
for any generalized initial condition $\mu$ on $\ER^d\times {\cal W}_{-}$.
\end{definition}
We have the following (classical) result:
\begin{prop}\label{prop:existence22} The existence of a Lyapunov function $V$ for  the Feller semi-group $({\cal Q}_t)_{t\ge0}$ implies the existence of an invariant distribution for $(Y_t)_{t\ge0}$. Furthermore, for any generalized initial condition $\mu$ such that 
$\int V(x)\mu(dx,dw)<+\infty$,
$\sup_{t\ge0} \ES_{\mu}[V(Y_t)]<+\infty.$
\end{prop}
\begin{proof} Let $\mu$ denote an initial condition on $\ER^d\times{\cal W}_{-}$ such that $\int V(x)\mu(dx,dw)<+\infty$ and denote by $(\mu_t)_{t\ge1}$ the sequence defined by 
\begin{equation}\label{mut222}
\mu_t=\frac{1}{t}\int_{0}^t \mu {\cal Q}_s ds, \quad t\ge 1.
\end{equation}
By construction and by the Feller property, every weak limit of $(\mu_t)_{t\ge1}$ is an invariant distribution for $({\cal Q}_t)_{t\ge0}$. It is thus enough to prove the tightness of $(\mu_t)_{t\ge1}$: owing to the stationarity of the increments of the fBm, the second marginal of $\mu_t$ does not depend on $t$. ${\cal W}_{-}$ being Polish, we are thus reduced to prove the 
tightness of $(\nu_t)_{t\ge0}$, $\nu_t$ being  the first marginal of $\mu_t$. But the definition of the Lyapunov function implies (by an iteration) 
$\sup_{t\ge1}\nu_t(V)<+\infty$ which in turn implies the tightness (using that $V^{-1}([0,K])$ is compact for any $K>0$).
\end{proof}

The aim of the next subsection is the exhibition of such a Lyapunov function $V$ for $Q$. The result will actually be derived from a general (deterministic) Lyapunov property for rough differential equations.

\subsection{A Lyapunov property for rough differential equations}\label{subsec:lyapou}

We go back here to the general case of a rough equation
\begin{equation}\label{rgh-std-eq}
dy_{t} = b(y_{t}) \, dt + \si(y_{t})\, d\mathbf{x}_{t} \quad , \quad t\in [0,1] \quad , \quad y_0=a\in \R^d \ ,
\end{equation}
where $\mathbf{x}$ is a given (deterministic) $\ga$-rough path on $[0,1]$, for some fixed parameter $\ga\in (\frac13,\frac12)$. In what follows, we will write $\lVert \mathbf{x}\rVert_{\ga}$ for $\lVert \mathbf{x}\rVert_{\ga;[0,1]}$.

\begin{theorem}\label{theo:lyapou}
Under Hypothesis $\mathbf{(H1)}$ and for every initial condition $y_0\in \R^d$, Equation (\ref{rgh-std-eq}) admits a unique solution $y$ on $[0,1]$, in the sense of Definition \ref{def:sol-davie}. Besides, if we assume in addition that Hypothesis $\mathbf{(H2)}$ holds true, then there exists a constant $C$ (which depends on $b,\sigma,\ga,C_1,C_2$, but not on $\mathbf{x}$) such that
\begin{equation}\label{ineq-lyapoun}
\lVert y_1\rVert^2 \leq e^{-C_2/2} \lVert y_0\rVert^2+C \big\{1+\lVert \mathbf{x}\rVert_\ga^\mu\big\} \quad , \quad \text{with} \ \mu:=\frac{8}{3\ga-1}\quad .
\end{equation}
\end{theorem}

\

Injecting this result into the stochastic setting of Section \ref{markov-structure} (where $\mathbf{x}:=\mathbf{X}$ is the canonical rough path above the fBm), the derivation of (\ref{lyapou-ineq}) is immediate. It is indeed a well-known fact (see for instance \cite[Theorem 15.33]{FV-bk}) that the random variable $\lVert \mathbf{X}\rVert_\ga$ admits finite moments of any order, and we are therefore in a position to state the desired property: 

\begin{corollary}\label{coro:lyapou-fbm}
In the setting of Section \ref{markov-structure} and assuming that both Hypotheses $\mathbf{(H1)}$ and $\mathbf{(H2)}$ hold true, the map $V:x\mapsto \lVert x\rVert^p$ defines a Lyapunov function for $\mathcal{Q}$, for any $p\geq 1$. As a consequence, there exists at least one invariant distribution $\nu$ for $(Y_t)_{t\geq 0}$, which additionally admits finite moments of any order.
\end{corollary}

\

The rest of this section is devoted to the proof of Theorem \ref{theo:lyapou}. Under Hypothesis $\mathbf{(H1)}$, the fact that there exists at most one solution to (\ref{rgh-std-eq}) (in other words, the uniqueness part of our statement) is a standard result, which can for instance be found in \cite[Theorem 3.3]{Da07}. On the opposite, due to the unboundedness of $b$, it seems that the proof of existence of a global solution on $[0,1]$ cannot be found as such in the literature, and we shall therefore provide a few details below.

\smallskip

In brief, our strategy towards Theorem \ref{theo:lyapou} is based on a careful analysis of the natural discrete numerical scheme associated with (\ref{rgh-std-eq}), in the same spirit as in \cite{Da07}. Let us thus introduce the sequence of dyadic partitions $\cp_n:=\{t_i=t_i^n:=\frac{i}{2^n} \, ; \ i=0,\ldots,2^n\}$ of $[0,1]$, and consider the discrete path $y^n$ defined on $\cp_n$ along the iterative formula 
\begin{equation}\label{defi-scheme-y-n}
y^n_0:=a \quad , \quad \delta y_{t_{i}t_{i+1}}^{n}
=
b(y_{t_{i}}^{n})\, \delta\ct_{t_{i}t_{i+1}} + \si(y_{t_{i}}^{n})\, \delta x_{t_{i}t_{i+1}}
+ (D\si\cdot \si)(y_{t_{i}}^{n})\,  \mathbf{x}_{t_{i}t_{i+1}}^{\2} \ ,
\end{equation}
where we recall that $\delta\ct_{st}=t-s$.
We shall also be led to handle the following quantities associated with $y^n$: for $s,t\in\cp_{n}$,
\begin{eqnarray*}
L_{st}^{y,n}
&:=&
\delta y_{st}^{n} - \si(y_{s}^{n})\, \delta x_{st}
-  (D\si\cdot \si)(y_{s}^{n})\,  \mathbf{x}_{st}^{\2} \\
R_{st}^{y,n}
&:=&
\delta y_{st}^{n} - b(y_{s}^{n}) \, \delta\ct_{st} - \si(y_{s}^{n})\, \delta x_{st}
-  (D\si\cdot \si)(y_{s}^{n})\, \mathbf{x}_{st}^{\2}  \\
Q_{st}^{y,n}
&:=&
\delta y_{st}^{n} - \si(y_{s}^{n})\, \delta x_{st} \ .
\end{eqnarray*}
For every $s<t\in [0,1]$, we will write $\llbracket s,t\rrbracket=\llbracket s,t\rrbracket_n:=[s,t]\cap \cp_n$, and we extend the norms introduced in Section \ref{subsec:classic-topos} to discrete paths in a natural way, namely 
$$\cn[f;\cac_{2}^{\mu}(\llbracket \ell_1,\ell_2\tau\wedge 1 \rrbracket)]:=\sup_{s<t\in \llbracket \ell_1,\ell_2\rrbracket} \frac{\lVert f_{st}\rVert}{|t-s|^\mu} \quad , \quad \cn[f;\cac_{1}^{\mu}(\llbracket \ell_1,\ell_2\tau\wedge 1 \rrbracket)]:=\cn[\delta f;\cac_{2}^{\mu}(\llbracket \ell_1,\ell_2\tau\wedge 1 \rrbracket)] \ .$$ 

\smallskip

The starting point of our analysis is the following local estimate for $R^{y,n}$, which can be obtained as a straightforward application of our forthcoming general Proposition \ref{prop:bound-vf1}:
\begin{proposition}\label{prop:bnd-Rny}
Fix $\ka:=\frac12 \big( \frac13+\ga\big)$. Then, under Hypothesis $(\mathbf{H1})$, there exists a constant $c_0$ (which depends only on $b,\si,\ga$) such that if we set 
$$T_0=T_0(\lVert \mathbf{x}\rVert):=\min\Big(1,\big( c_0\{1+\lVert \mathbf{x}\rVert_{\ga}\} \big)^{-1/(\ga-\ka)} \Big) \ ,$$
one has, for every $\tau\in \mathcal{P}_n$ satisfying $0<\tau \leq T_0$ and every $k\leq 1/\tau$,
\begin{equation}\label{r-y-n-lyapou}
\cn[R^{y,n};\cac_{2}^{3\ka}(\llbracket k\tau,(k+1)\tau\wedge 1 \rrbracket)] \leq c_0  \big\{1+\|y^n_{k\tau}\| \big\} \ .
\end{equation}
\end{proposition}

\

\begin{corollary}\label{coro:bound-part-y}
In the setting of Proposition \ref{prop:bnd-Rny}, there exists a constant $c_1$ (which depends only on $b,\si,\ga$) such that for every $\tau\in \mathcal{P}_n$ satisfying $0<\tau \leq T_0$ and every $k\leq 1/\tau$, one has
\begin{equation}\label{y-n-infty-lyapou}
\cn[y^n;\cac_1^0( \llbracket k\tau,(k+1)\tau\wedge 1 \rrbracket)] \leq c_1 \{ 1 + \lVert y^n_{k\tau}\rVert\} \ ,
\end{equation}
\begin{equation}\label{y-n-hold-lyapou}
\cn[y^n;\cac_1^\ga( \llbracket k\tau,(k+1)\tau\wedge 1 \rrbracket)] \leq c_1 \{ 1 + \lVert y^n_{k\tau}\rVert\}\{1+ \lVert \mathbf{x}\rVert_\ga\} 
\end{equation}
and
\begin{equation}\label{q-y-n-lyapou}
\cn[Q^{y,n};\cac_2^{2\ga}( \llbracket k\tau,(k+1)\tau\wedge 1 \rrbracket)] \leq c_1 \{ 1 + \lVert y^n_{k\tau}\rVert\}\{1+ \lVert \mathbf{x}\rVert_\ga\} \ .
\end{equation}
\end{corollary}

\begin{proof}
For every $t\in \llbracket k\tau,(k+1)\tau \wedge 1\rrbracket$, write
\begin{equation*}
y_{t}^{n}
=
y_{k\tau}^{n} +
b(y_{k\tau}^{n})\, \delta\ct_{k\tau,t} + \si(y_{k\tau}^{n})\, \delta x_{k\tau,t}
+ (D\si \cdot\si)(y_{k\tau}^{n})\,  \mathbf{x}_{k\tau,t}^{\2} + R_{k\tau,t}^{y,n} \ ,
\end{equation*}
so that using (\ref{r-y-n-lyapou}), we get $\lVert y^n_t\rVert \lesssim 1+\lVert y^n_{k\tau}\rVert+\lVert \mathbf{x}\rVert_\ga T_0^\ga$, and (\ref{y-n-infty-lyapou}) now follows from the fact that $\lVert \mathbf{x}\rVert_\ga T_0^\ga\leq \lVert \mathbf{x}\rVert_\ga T_0^{\ga-\ka}\lesssim 1$.

\smallskip

\noindent
Then, in a more general way, we have for every $s<t\in \llbracket k\tau,(k+1)\tau\wedge 1\rrbracket$
\begin{equation*}
\delta y_{st}^{n}
=
b(y_{s}^{n})\, \delta\ct_{st} + \si(y_{s}^{n})\, \delta x_{st}
+ (D\si \cdot\si)(y_{s}^{n})\,  \mathbf{x}_{st}^{\2} + R_{st}^{y,n} 
\end{equation*}
and
\begin{equation*}
Q_{st}^{y,n}
=
b(y_{s}^{n})\, \delta\ct_{st} 
+ (D\si \cdot\si)(y_{s}^{n})\,  \mathbf{x}_{st}^{\2} + R_{st}^{y,n} \ .
\end{equation*}
Injecting (\ref{r-y-n-lyapou}) and (\ref{y-n-infty-lyapou}) into these expressions easily yields (\ref{y-n-hold-lyapou}) and (\ref{q-y-n-lyapou}).
\end{proof}

\begin{corollary}
Under Hypothesis $\mathbf{(H1)}$, Equation (\ref{rgh-std-eq}) admits a unique global solution $y$ on $[0,1]$. Besides, with the previous notations, there exists a subsequence of $(y^n)$, that we still denote by $(y^n)$, such that
\begin{equation}\label{conv-y-y-n-lyapou}
\max_{i=0,\ldots,2^n} \lVert y_{t_i}-y^n_{t_i}\rVert \stackrel{n\to \infty}{\longrightarrow} 0  \ .
\end{equation}
\end{corollary}

\begin{proof}
Although the two local estimates (\ref{y-n-infty-lyapou})-(\ref{y-n-hold-lyapou}) are not uniform as such (that is, the right-hand side still depends on $y^n$), they easily give rise, via an obvious iterative procedure, to a uniform estimate for $\cn[y^n;\cac_1^{0,\ga}(\llbracket 0,1\rrbracket)]:=\cn[y^n;\cac_1^{0}(\llbracket 0,1\rrbracket)]+\cn[y^n;\cac_1^{\ga}(\llbracket 0,1\rrbracket)]$. Still denoting by $y^n$ the continuous path obtained through the linear interpolation of $(y^n_{t_i})_{i=0,\ldots,2^n}$, we thus get a uniform estimate for $\cn[y^n;\cac_1^{0,\ga}([ 0,1])]$, which, by a standard compactness argument, allows us to conclude about the existence of a path $y\in \cac_1^\ga([0,1])$, as well as a subsequence of $y^n$ (that we still denote by $y^n$), such that $y^n \to y$ in $\cac_1^{0,\ga'}([0,1])$ for every $0<\ga'<\ga$.

\smallskip

The fact that $y$ actually defines a solution of (\ref{rgh-std-eq}) is then an easy consequence of the bound (\ref{r-y-n-lyapou}). The details of the procedure can for instance be found at the end of \cite[Section 3.3]{Deya}. Finally, and as we have already evoked it in the beginning of the section, the uniqueness of this solution is a standard result from the rough-path literature (see \cite[Theorem 3.3]{Da07}).
\end{proof}

Let us now turn to the proof of the second part of Theorem \ref{theo:lyapou}, that is to the proof of (\ref{ineq-lyapoun}) under Hypotheses $\mathbf{(H1)}$ and $\mathbf{(H2)}$. To this end, we introduce, for every $n\geq 0$, the additional discrete path $z^n:\mathcal{P}_n\to \R$ defined for every $t\in \mathcal{P}_n$ as
$$z^n_t:=\frac12 \lVert y^n_t\rVert^2 \ .$$
In the same vein as above, we will lean on the following quantities related to $z^n$: for every $s,t\in\cp_{n}$,
\begin{eqnarray*}
R_{st}^{z,n}
&:=&
\delta z^n_{st}-\langle y_{s}^{n},b(y_s^{n})\rangle \,  \delta\ct_{st} -\langle y_{s}^{n},\si(y_s^{n})\rangle \, \delta x_{st}
-\varSigma(y^n_s)\,   \mathbf{x}_{st}^{\2} \\
L_{st}^{z,n}
&:=&
\delta z^n_{st} -\langle y_{s}^{n},\si(y_s^{n})\rangle \, \delta x_{st}
-\varSigma(y^n_s)\,   \mathbf{x}_{st}^{\2}  \\
Q_{st}^{z,n}
&:=&
\delta z^n_{st} -\langle y_{s}^{n},\si(y_s^{n})\rangle \, \delta x_{st} \ ,
\end{eqnarray*}
where we have set
$$\varSigma(y^n_s):=\langle \si (y^n_s),\si (y^n_s) \rangle+\langle y^n_s,(D\si \cdot \si)(y^n_s)\rangle \ .$$
Just to be clear, the notation for the second-order term in $R^{z,n},L^{z,n}$ specifically refers to the sum
$$\varSigma(y^n_s)\,   \mathbf{x}_{st}^{\2}=\big\{\langle \si_j (y^n_{s}),\si_k (y^n_{s}) \rangle+\langle y^n_{s},(D\si_j \cdot \si_k)(y^n_{st})\rangle \big\}\,   \mathbf{x}_{st}^{\2,jk} \ .$$
Finally, along the same lines as in the subsequent Section \ref{sec:rougheqinholder}, we set, if $s=\frac{p}{2^n}$ and $t=\frac{q}{2^n}$ and $G:\llbracket 0,1\rrbracket \to \R^d$,
$$\cm^{\mu}\big[G;\llbracket s,t\rrbracket\big]:=\sup_{p\leq i\leq q} \frac{\| G_{t_it_{i+1}}\|}{|t_{i+1}-t_i|^\mu} \ .$$

Let us start with a few estimates on $R^{z,n}$, for which Hypothesis $\mathbf{(H2)}$ is still not required:

\begin{lemma}
Under Hypothesis $\mathbf{(H1)}$ and with the above notations, there exists a constant $c_2$ (which depends only on $b,\si,\ga$) such that for every $s<t\in \mathcal{P}_n$, one has
\begin{equation}\label{bound-m-r-y-n}
\cm^{3\ga}\big[ R^{z,n};\llbracket s,t\rrbracket \big] \leq c_2 \{1+\lVert \mathbf{x}\rVert_\ga^2 \}\{1+\cn[y^n;\cac_1^0(\llbracket s,t\rrbracket)]^2\}  \ .
\end{equation}
\end{lemma}

\begin{proof}
We have
$$\delta z^n_{t_it_{i+1}}=\langle y^n_{t_i},\delta y^n_{t_it_{i+1}} \rangle+\frac12 \langle \delta  y^n_{t_it_{i+1}},\delta y^n_{t_it_{i+1}} \rangle \ ,$$
and so, injecting (\ref{defi-scheme-y-n}) into the first term immediately gives, thanks to the second identity in (\ref{levy-area-def}), 
\begin{equation}\label{r-z-n-t-i}
R^{z,n}_{t_it_{i+1}} = \frac12 \langle \delta  y^n_{t_it_{i+1}},\delta y^n_{t_it_{i+1}} \rangle-\langle \si(y^n_{t_i}),\si(y^n_{t_i}) \rangle \mathbf{x}^{\bf 2}_{t_it_{i+1}}=\langle \si(y^n_{t_i}) \, \delta x_{t_it_{i+1}}+\frac12 Q^{y,n}_{t_i t_{i+1}},Q^{y,n}_{t_i t_{i+1}} \rangle\ .
\end{equation}
Finally, since $Q^{y,n}_{t_it_{i+1}}=b(y^n_{t_i}) \, \delta \mathcal{T}_{t_it_{i+1}}+(D\si \cdot \si)(y^n_{t_i}) \, \mathbf{x}^{\bf 2}_{t_it_{i+1}}$, it is immediate that 
$$\lVert Q^{y,n}_{t_it_{i+1}} \rVert \lesssim |t_{i+1}-t_i| \{1+ \lVert y^n_{t_i}\rVert \}+|t_{i+1}-t_i|^{2\ga} \lVert \mathbf{x}\rVert_\ga \ .$$
Going back to (\ref{r-z-n-t-i}), we get the conclusion.
\end{proof}

\begin{proposition}\label{prop:r-z-n}
Assume Hypothesis $\mathbf{(H1)}$ holds true and let $T_0=T_0(\lVert \mathbf{x}\rVert_{\ga})$ be the time defined in Proposition \ref{prop:bnd-Rny}. Then there exists a constant $c_3$ (which depends only on $b,\si,\ga$) such that for every $\tau\in \mathcal{P}_n$ satisfying $0<\tau\leq T_0$ and every $k\leq 1/\tau$, one has
$$
\cn[R^{z,n};\cac_2^{3\ga}(\llbracket k\tau,(k+1)\tau \wedge 1\rrbracket)] \le c_3 \{1+\lVert \mathbf{x}\rVert_\ga^3 \} \{1+z_{k\tau}^{n} \}  \ .
$$ 
\end{proposition}

\begin{proof}
Thanks to the forthcoming Lemma \ref{lem:discr-sewing}, we can rely on the estimate
$$
\cn[R^{z,n};\cac_2^{3\ga}(\llbracket k\tau,(k+1)\tau \wedge 1\rrbracket)] \lesssim \cm^{3\ga}\big[ R^{z,n};\llbracket k\tau,(k+1)\tau \wedge 1\rrbracket \big]+\cn[\delta R^{z,n};\cac_3^{3\ga}(\llbracket k\tau,(k+1)\tau \wedge 1\rrbracket)] \ .
$$
As far as the first term is concerned, combining (\ref{bound-m-r-y-n}) and (\ref{y-n-infty-lyapou}) allows us to assert that
$$\cm^{3\ga}\big[ R^{z,n};\llbracket k\tau,(k+1)\tau \wedge 1\rrbracket \big] \lesssim \{1+\lVert \mathbf{x}\rVert_\ga^2\} \{1+z^n_{k\tau}\} \ .$$
Then, for every $s<u<t\in \llbracket k\tau,(k+1)\tau \wedge 1\rrbracket$, decompose $\delta R^{z,n}_{sut}$ as
$$\delta R^{z,n}_{sut}=-\delta (\langle y^n,b(y^n)\rangle)_{su} \, \delta \mathcal{T}_{ut}+\delta L^{z,n}_{sut} \  .$$
On the one hand, one has, by (\ref{y-n-infty-lyapou}) and (\ref{y-n-hold-lyapou}),
$$
\big| \delta (\langle y^n,b(y^n)\rangle)_{su} \big| \leq  \big| \langle \delta y^n_{su},b(y^n_u)\rangle \big|+\big| \langle y^n_s,\delta b(y^n)_{su}\rangle \big| \lesssim |u-s|^\ga \{1+\lVert \mathbf{x}\rVert_\ga\} \{1+z^n_{k\tau}\} \ .
$$
On the other hand, combining Chen's identity with elementary Taylor expansions easily leads us to the decomposition 
$$\delta L^{z,n}_{sut}=\big\{ I^i_{su}+II^i_{su}+III^i_{su}+IV^i_{su}\big\} \, \delta x^i_{ut}+\delta \varSigma_{ij}(y^n)_{su} \, \mathbf{x}^{\mathbf{2},ij}_{ut} \ ,$$
with
$$\varSigma_{ij}(y^n):=\langle \si_i (y^n),\si_j (y^n) \rangle+\langle y^n,(D\si_i \cdot \si_j)(y^n)\rangle \ ,$$
$$I^i_{su}:=\langle \delta y^n_{su},\delta \si_i(y^n)_{su} \rangle \ , \quad II^i_{su}:=\langle \si_i(y^n_s),Q^{y,n}_{su}\rangle \ , \quad III^i_{su}:=\int_0^1 d\xi \, \langle y^n_s,D\si_i(y^n_s+\xi\, \delta y^n_{su}) Q^{y,n}_{su} \rangle,
$$
and finally
$$ IV^i_{su}:=\int_0^1 d\xi \, \langle y^n_s,\big[ D\si_i(y^n_s+\xi\, \delta y^n_{su})-D\si_i(y^n_s)\big](\si_j(y^n_s))\rangle \, \delta x^j_{su}  \ .$$
With the above expressions in mind and using the three estimates (\ref{y-n-infty-lyapou}), (\ref{y-n-hold-lyapou}) and (\ref{q-y-n-lyapou}), it is not hard to check that
$$\big| \delta L^{z,n}_{sut} \big| \lesssim |t-s|^{3\ga} \{1+\lVert \mathbf{x}\rVert_\ga^3 \} \{1+z_{k\tau}^{n} \} \ ,$$
which achieves the proof of our assertion.

\end{proof}

Let us finally involve Hypothesis $\mathbf{(H2)}$ into the picture:
\begin{corollary}\label{coro:toward-lyapou}
Assume Hypotheses $\mathbf{(H1)}$ and $\mathbf{(H2)}$ hold true and let $T_0=T_0(\lVert \mathbf{x}\rVert_{\ga})$ be the time defined in Proposition \ref{prop:bnd-Rny}. Then there exist constants $c_4,c_5$ (both depending only on $b,\sigma,\ga,C_1,C_2$) such that if we set
$$T_1=T_1(\lVert \mathbf{x}\rVert_\ga):=\min \bigg(T_0,\frac{2}{C_2},\bigg(\frac{1}{c_4 \{1+\lVert \mathbf{x}\rVert_\ga^3 \}}\bigg)^{1/(3\ga-1)}  \bigg) \ ,$$
one has, for every $\tau\in \mathcal{P}_n$ satisfying $0<\tau\leq T_1$ and every $k\leq 1/\tau$,
\begin{equation}\label{eq:recursive-zn}
z_{(k+1)\tau \wedge 1}^{n}
\le
\Big( 1 - \frac{C_2}{2} \tau \Big) z_{k\tau}^{n}+c_5\{1+\lVert \mathbf{x}\rVert_\ga^2\}\tau^{2\ga-1} \ ,
\end{equation}
where we recall that the two parameters $C_1,C_2$ have been introduced in Hypothesis (H2).
\end{corollary}

\begin{proof}
Using Hypothesis $\mathbf{(H2)}$, we get that for every $\tau\in \mathcal{P}_n$ and every such that $k\leq 1/\tau$,
\begin{eqnarray*}
z^n_{(k+1)\tau \wedge 1}& = & z^n_{k\tau}+\langle y_{k\tau}^{n},b(y_{k\tau}^{n})\rangle \,  \delta\ct_{k\tau,(k+1)\tau \wedge 1} \\
& & +\langle y_{k\tau}^{n},\si(y_{k\tau}^{n})\rangle \, \delta x_{k\tau,(k+1)\tau \wedge 1}
+\varSigma(y^n_{k\tau})\,   \mathbf{x}_{k\tau,(k+1)\tau \wedge 1}^{\2}+R^{z,n}_{k\tau,(k+1)\tau \wedge 1}\\
&\leq& (1-C_2 \tau) z^n_{k\tau}+C_1 \tau\\
& &+\langle y_{k\tau}^{n},\si(y_{k\tau}^{n})\rangle \, \delta x_{k\tau,(k+1)\tau \wedge 1}
+\varSigma(y^n_{k\tau})\,   \mathbf{x}_{k\tau,(k+1)\tau \wedge 1}^{\2}+R^{z,n}_{k\tau,(k+1)\tau \wedge 1} \ ,
\end{eqnarray*}
and so, thanks to Proposition \ref{prop:r-z-n}, we can conclude that for every $0< \tau \leq \min\big( T_0,\frac{2}{C_2}\big)$ and every $k\leq 1/\tau$, one has
\begin{equation}\label{a1}
z^n_{(k+1)\tau \wedge 1}\leq (1-C_2 \tau) z^n_{k\tau}+C_1\tau+c_4 \Big[ \lVert \mathbf{x}\rVert_\ga \tau^\ga \big\{1+(z^n_{k\tau})^{\frac12} \big\}+\frac{C_2}{4} \tau^{3\ga}\{1+\lVert \mathbf{x}\rVert_\ga^3 \} \{1+z_{k\tau}^{n} \} \Big] \ ,
\end{equation}
for some constant $c_4=c_4(b,\si,\ga,C_2)$. Now, by the very definition of $T_1$, we know that if $0<\tau\leq T_1$, then
$$c_4\tau^{3\ga} \{1+\lVert \mathbf{x}\rVert_\ga^3 \}\leq \tau \big( c_4 \tau^{3\ga-1} \{1+\lVert \mathbf{x}\rVert_\ga^3 \}\big) \leq \tau \ ,$$
and thus we can recast relation \eqref{a1} into:
\begin{equation*}
z^n_{(k+1)\tau \wedge 1}\leq \lp 1-\frac{3C_2}{4} \tau\rp z^n_{k\tau}+C_1\tau
+c_4  \lVert \mathbf{x}\rVert_\ga \tau^\ga \big\{1+(z^n_{k\tau})^{\frac12} \big\}  \ 
\end{equation*}
To achieve the proof, it now suffices to use the basic inequality
$$c_4 \lVert \mathbf{x}\rVert_\ga \tau^\ga (z^n_{k\tau})^{\frac12} \leq \frac{C_2}{4} \tau z^n_{k\tau}+\frac{c_4^2}{C_2} \lVert \mathbf{x}\rVert_\ga^2 \tau^{2\ga-1} \ .$$
\end{proof}

\smallskip

At this point, we are very close to (\ref{ineq-lyapoun}). With the notations of Corollary \ref{coro:toward-lyapou}, consider $n$ large enough such that we can exhibit $\tau_0\in \mathcal{P}_n$ satisfying $\frac12 T_1 \leq \tau_0\leq T_1$, and then let $K$ be the integer such that $(K-1)\tau_0 \leq 1 <K\tau_0$. Iterating the bound (\ref{eq:recursive-zn}) with $\tau=\tau_0$ yields that 
\begin{eqnarray*}
z^n_1 &\leq& \Big( 1-\frac{C_2}{2}\tau_0\Big)^Kz^n_0+c_5\, K\{1+\lVert \mathbf{x}\rVert_\ga^2\}\tau_0^{2\ga-1}\\
&\leq& \Big( 1-\frac{C_2}{2}\tau_0\Big)^{\frac{1}{\tau_0}}z^n_0+c_5\, K\{1+\lVert \mathbf{x}\rVert_\ga^2\}\tau_0^{2\ga-1}\\
&\leq& e^{-C_2/2}z^n_0+c_5\, K\{1+\lVert \mathbf{x}\rVert_\ga^2\}\tau_0^{2\ga-1} \ .
\end{eqnarray*}
Thanks to (\ref{conv-y-y-n-lyapou}), the conclusion is now immediate, by noting that $K \tau_0^{2\ga-1}\lesssim T_1^{2\ga-2}$ and then using the explicit description of $T_1,T_0$ in terms of $\lVert \mathbf{x}\rVert_\ga$.

\section{Sketch of the strategy}\label{sec:sketch-strategy}
We now turn to the second part of Theorem \ref{theo:principal} about the convergence in total variation of the process towards the stationary solution. This result is based on a coupling method first introduced in \cite{hairer}. We thus  begin by recalling the details of the strategy. To this end, we first introduce some notations about the Mandelbrot-Van Ness representation of the fBm.

\subsection{Decomposition of the fBm}\label{sec:decompo-fbm-ter}

As recalled in \eqref{eq:mandel}, the Mandelbrot-Van Ness formula allows us to realize any fBm $(X_t)_{t\ge0}$ (with Hurst parameter $H\in(0,1)$) through a standard two-sided Brownian motion $(W_t)_{t\in\ER}$. The representations immediately gives rise to the decomposition 
\begin{equation}\label{decompo-fbm-ter}
X_t=D_t+Z_t \ ,
\end{equation}
where the process $D$ defined by
\begin{equation*}
D_t:=\al_H \int_{-\infty}^{0} \{ (t-r)^{H-\frac12}-(-r)^{H-\frac12}\} \, dW_r \quad \text{and} \quad Z_t:=\al_H\int_{0}^t (t-r)^{H-\frac12} \, dW_r\ .
\end{equation*}
is seen a the 'past' component encoding the 'memory' of $W$, while
$$Z_t:=\al_H\int_{0}^t (t-r)^{H-\frac12} \, dW_r$$
stands for the 'innovation' process (when looking at $X$ after time $0$).

\smallskip

It turns out that, away from $0$, the process $D$ so defined is smooth (see Lemma \ref{lem:justifd} for details), so that the roughness of $X$ is essentially inherited from that of $Z$. This basic observation will be one the keys of our analysis, at every step of the strategy. {It is worth noting that the smoothness of $D$ was also already used in the past as a central ingredient while studying integration issues with respect to the fBm (see e.g. \cite{alos-mazet-nualart,viens-zhang}), and similar ideas can be found in \cite{hairer-pillai} as well.} All along the procedure, we will thus be led to control the past of the process through the quantity
\begin{equation}\label{eq:normeD}
\vvvert D\vvvert_{1;\ga}:=\sup_{t\in (0,1]} t^{1-\ga} |D'(t)| \ ,
\end{equation}
for some fixed parameter $\ga\in (0,H)$. Let us more generally introduce the following class of functions:
\begin{notation}\label{nota:w-ga-1}
For every $k\geq 1$ and every $\ga\in (0,1)$, we denote by $\mathcal{E}^k_\ga$ the space of paths $f:[0,1]\to \R^d$ which are continuous on $[0,1]$, $k$-times differentiable on $(0,1]$, and such that 
\begin{equation}\label{eq:conditioneinftygamma}
\vvvert f\vvvert_{k;\ga}:=\max_{1\leq \ell \leq k}\ \sup_{t\in (0,1]} t^{\ell-\ga}|f^{(\ell)}(t)| \ < \ \infty  \ .
\end{equation}
\end{notation}

\subsection{The general 3-step scheme}\label{subsec:general3stepschem}

Let $(X_t)_{t\in\ER}$ and $(\widetilde{X}_t)_{t\in\ER}$ denote two fractional Brownian motions with common Hurst parameter $H\in(1/3,1)$. {From now on and for the rest of the paper, we fix a parameter $\gamma \in (1/3,H)$ that will serve us throughout the reasoning.} Then, denote by $(Y_t,\widetilde{Y}_t)$, a couple  of solutions to \eqref{fractionalSDE0}:

\bqn\label{eq:eqcoup}
\begin{cases}
dY_t= b(Y_t)dt+\sigma(Y_t)\, d{\bf X}_t\\
d\widetilde{Y}_t= b(\widetilde{Y}_t)dt+\sigma(\widetilde{Y}_t)\, d{ \widetilde{\bf X}}_t
\end{cases}
\eqn
with initial conditions $(Y_0,(X_t)_{t\le0})$, $(\widetilde{Y}_0,(X_t)_{t\le0})$. We denote by  $({\cal F}_t)_{t\ge0}$ the usual augmentation of the filtration  $(\sigma(X_s,\widetilde{X}_s, (Y_0,\widetilde{Y}_0))_{s\le t})_{t\ge0}$.   
To initiate the coupling procedure without ``weight of the past'', we will certainly  assume that $a.s$, 
\begin{equation}\label{eq:xegalxtilde0}
(X_t)_{t\le0}=(\widetilde{X}_t)_{t\le0}
\end{equation}
and that the initial distribution $\widetilde{\mu}$ of $(Y_0,\widetilde{Y}_0)$ is of the form
\bqn\label{eq:mutilde}
\widetilde{\mu}(dx,d\widetilde{x},dw)=\nu_0(w,dx){\nu}(w,d\widetilde{x})\PE_H(dw)
\eqn
where $\PE_H$ denotes the distribution of a fBm $(X_t)_{t\le0}$ on ${\cal C}(\ER^-,\ER^d)$ and the transitions probabilities $\nu_0(.,dx)$ and $\nu(.,d\tilde{x})$
correspond respectively to the conditional distributions of $Y_0$ and $\widetilde{Y}_0$ given $(X_t)_{t\le0}$. Furthermore, we assume that 
\begin{equation}\label{eq:condmutilde}
{\nu}_0(w,d\widetilde{x})\PE_H(dw)=\mu_0\quad\textnormal{and that }{\nu}(w,d\widetilde{x})\PE_H(dw)=\mu.
\end{equation}
In other words,  $\widetilde{Y}$ is a stationary solution whereas $Y$ starts with a given initial condition $\mu_0$. {At this point, let us remember that thanks to Corollary \ref{coro:lyapou-fbm}, we can choose $\mu$ in such a way that for every $r>0$, 
$$
\int |x|^r\bar{\mu}(dx) \ <+\infty \ ,
$$
where, as usual, $\bar{\mu}$ stands for the first marginal of $\mu$. In fact, for the rest of the paper and along the assumptions of Theorem \ref{theo:principal}, we fix $r>0$ such that one has simultaneously
\begin{equation}
\int |x|^r\bar{\mu_0}(dx) \ <+\infty \quad \text{and} \quad \int |x|^r\bar{\mu}(dx) \ <+\infty \ .
\end{equation}} 

The processes  $(X_t)_{t\in\ER}$ and $(\widetilde{X}_t)_{t\in\ER}$ can be realized through the decomposition introduced in the previous subsection with respect to some  two-sided Brownian motions respectively denoted by $W$ and $\widetilde{W}$ .  In particular, the filtration $({\cal F}_t)_{t\ge0}$  is also generated by  $(\sigma(W_s,\widetilde{W}_s, (Y_0,\widetilde{Y}_0))_{s\le t})_{t\ge0}$.

Furthermore, we will assume in all the proof that on $[0,\infty)$,  $W$ and $\widetilde{W}$ (resp. $X$ and $\widetilde{X}$)
differ by  a (random) drift term denoted by $\gw$ (resp. $g_X$): 
\bqn\label{eq:gw}
d\widetilde{W}_t=dW_t+\gw(t) dt\quad \textnormal{and}\quad d\widetilde{X}_t=dX_t+g_X(t) dt.
\eqn
In the continuity of \eqref{eq:xegalxtilde0}, we assume that for every $t\in \ER_-$,
$$W_t=\widetilde{W}_t\quad a.s.\quad\textnormal{or equivalently that}\quad \gw(t)=g_{X}(t)=0 \quad a.e.$$
Note that the functions $\gw$ and $g_X$ are linked by the following formulas, whenever the latter make sense (see \cite{hairer}, Lemma 4.2 for details):
\begin{align}\label{eq:defgwgb} 
&\gw(t)=\alpha_H \frac{d}{dt}\int_{-\infty}^t(t-s)^{\frac 12 -H} g_X(s) ds\\
&g_X(t)=\gamma_H\alpha_{1-H} \frac{d}{dt}\int_{-\infty}^t(t-s)^{H-\frac 12 } \gw(s) ds.\label{eq:defgb} 
\end{align}


The idea is to build $\gw$ (resp. $g_X$) in order to stick $Y$ and $\widetilde{Y}$. We set 
\begin{equation}\label{defi-tau-infty}
\tau_\infty:=\inf\{t\ge0,\; Y_s=\widetilde{Y}_s\ \ \text{for all}\ s\ge t\} \ .
\end{equation}
Before going further, let us recall a classical relation between $\tau_\infty$ and the total variation distance. Denote by ${\cal B}$ the set of Borel functions $F:{\cal C}(\ER_+,\ER^d)\rightarrow \ER$. Then, owing to the stationarity of $\widetilde{Y}$, we have for any $t\ge0$
\begin{equation}\label{eq:lienTVtauinfty}
\begin{split}
\|{\cal L}(Y_{t+.})-\bar{\cal Q}{\mu}\|_{TV}&=\sup_{F \in{\cal B}, \|F\|_{\infty}\le 1}\ES\left[\left(F(Y_{t+.})-F(\widetilde{Y}_{t+.})\right)\right]\\
&=\sup_{F \in{\cal B}, \|F\|_\infty\le 1}\ES\left[\left(F(Y_{t+.})-F(\widetilde{Y}_{t+.})\right)1_{\{\tau_\infty>t\}}\right]\le 2\PE(\tau_\infty>t)\ .
\end{split}
\end{equation}
As a consequence, in the sequel, we will focus on the quantity $t\mapsto \PE(\tau_\infty>t)$ in order to prove the main theorem. \smallskip

\noindent As usual, the construction  of the coupling  is achieved through  a series of trials. As mentioned in the introduction, each trial is decomposed in three steps:
\begin{itemize}
\item{} Step 1: Try to couple the positions with a \textit{controlled cost} (in a sense made precise below).
\item{} Step 2 (specific to non-Markov processes): Try to keep the paths  fastened together ($i.e.$ to ensure that $g_X(t)=0$).
\item{} Step 3: If Step 2 fails, wait a sufficiently long time in order that in the next trial,  Step 1 be achieved with a \textit{controlled cost} and with  (uniformly lower-bounded away from $0$) probability. During this step, we suppose that $\gw(t)=0$.
\end{itemize}
Let us make  a few precisions:

\smallskip
\noindent $\rhd$ We denote by $\tau_0\ge0$ the beginning of the first trial and by $\tau_k$, $k\ge1$, the end of each trial. {This also means that $\tau_{k-1}$ designates the beginning of the $k$th trial.} We will assume in the sequel that $\tau_0=0$.

 If $\tau_k=+\infty$,
the coupling tentative has been successful. Otherwise, $\tau_k$ is the end of Step 3 of trial $k$.\smallskip

\noindent $\rhd$  Step 1  is  carried out on each interval $[\tau_{k-1},\tau_{k-1}+1]$. The ``cost'' of coupling is represented by the function $\gw$ that one needs to 
build on $[\tau_{k-1},\tau_{k-1}+1]$ in order to get  $Y$ and $\widetilde{Y}$  stuck together at time $\tau_{k-1}+1$. Oppositely to the Markovian case, this cost does not only depend  on the positions of  $Y_{\tau_{k-1}}$ and $\widetilde{Y}_{\tau_{k-1}}$  but also on the past of the Brownian motions,  which have a (strong) influence on the dynamics of $X$ and $\widetilde{X}$.  \smallskip

\noindent {If Step 1 fails, one begins Step 3 (see below) at time $\tau_{k-1}+1$}. Otherwise,
one begins Step 2.\smallskip

\noindent $\rhd$  Step 2 is in fact a series of trials on some intervals {$I_{k,\ell}$ with length 
\bqn\label{eq:defc2}
|I_{k,\ell}|=c_2 2^\ell \ ,
\eqn
independently of $k$, where $c_2$ is a constant greater than one which will be calibrated in the sequel}. More precisely, one successively tries to keep $Y$ and $\widetilde{Y}$ as being equal on intervals $[\tau_{k-1}+1+c_2\sum_{u=1}^{\ell-1} 2^{k},\tau_{k-1}+1+c_2\sum_{u=1}^{\ell} 2^k]$ (with the convention $\sum_{\emptyset}=0$).  Equivalently, this means that on these successive intervals,
$$ g_X(t)=0\quad a.s.$$
Thus, the first natural question is the following: how to build the coupling  $(W,\widetilde{W})$ of the innovations in order to ensure this property, $i.e.$ what is the corresponding function $\gw$ on these successive intervals which ensures that $g_X=0$ ?
The answer is given by Lemma 4.3 of \cite{hairer} that we choose to recall here in a slightly different way: 
\begin{lemma}\label{lemma:formdrift}
Let $\tau$, $t_0$ be fixed positive numbers, and assume that $\gw(t)=0$ on $(-\infty,0]$. Then  $g_X=0$ on $ (\tau,\tau+t_0]$ if and only if for all $t\in (0,t_0]$,
\bqn\label{a11}
\gw^{\tau}(t)=({\cal R}_0 \gw^\tau)(t)
\eqn 
where $\gw^\tau(.)=\gw(\tau+.)$ and where the operator ${\cal R}_0$ is defined as
\begin{equation}\label{eq:def-R-0}
({\cal R}_0 g)(t)=C_H\int_{-\infty}^0 \frac{t^{\frac{1}{2}-H} (-s)^{H-\frac{1}{2}}}{t+T-s} g(s) ds\ ,\quad t\in(0,+\infty) \ ,
\end{equation}
for some appropriate constant $C_H$ (depending only on $H$). 
\end{lemma}

\begin{remark}\label{rmk:RT} Since $\gw=0$ on $(-\infty,0]$, observe that $({\cal R}_0 \gw^\tau(t))_{t\ge0}$ only involves the values of $\gw$ on $[0,\tau]$, and therefore relation (\ref{a11}) provides an explicit description of the values of $\gw$ on $(\tau,\tau+t_0]$ in terms of $(\gw(t))_{0\leq t\leq \tau}$, as expected. 
\end{remark}

\begin{remark}
For our further computations, we will also use the following generalization of the operator ${\cal R}_{0}$, defined for a parameter $T\ge 0$:
\begin{equation}\label{eq:def-R-T}
({\cal R}_T g)(t)=C_H\int_{-\infty}^0 \frac{t^{\frac{1}{2}-H} (T-s)^{H-\frac{1}{2}}}{t+T-s} g(s) ds\ ,\quad t\in(0,+\infty) \ .
\end{equation}
\end{remark}

The attempt is successful if Step $1$ and Step $2$ ($i.e.$ all the sub-attempts of this step) are. To ensure a positive probability to the success of the $k$-th attempt, one needs certainly to impose some conditions on the system at time $\tau_{k-1}$.  

\medskip

In the continuity of \cite{hairer} and \cite{fontbona-panloup}, we thus introduce an admissibility condition (we recall that we have fixed a parameter $\gamma\in (1/3,H)$ for the whole study):

\begin{definition}\label{def:admissible-state} 
Let $K$ and $\alpha$ be some positive constants and fix a time $\tau \geq 0$.  Then we will say that a state $\pi:=(a,\tilde{a},w,\widetilde{w})\in (\R^d)^2\times \cac((-\infty,\tau];\R^d)^2$ is $(K,\alpha,\gamma)$-admissible at time $\tau$ and if the following conditions are satisfied:

\smallskip

\noindent
$(i)$ It holds that $\wti=w+\int_{-\infty}^. \gw(s) \, ds$, with $\gw$ satisfying:
\begin{equation}\label{hairerassumpcond}
 \sup_{T\ge0} \int_0^{+\infty} (1+t)^{2\alpha}|({\cal R}_T |\gw^\tau|)(t)|^2dt\le 1  \ ; 
\end{equation}

\smallskip

\noindent
$(ii)$ It holds that
\bqn\label{Kadmiscond}
 |a|+ |\tilde{a}|+  \vvvert D^{(\tau)}(w) \vvvert_{1;\gamma}+ \vvvert D^{(\tau)}(\wti) \vvvert_{1;\gamma}\le K
\eqn
where we have set, for $t\in [0,1]$,
\begin{equation}\label{eq:def-D-tau-w}
D^{(\tau)}(w)_t:=\int_{-\infty}^ {\tau} \left((t+\tau-r)_+^{H-\frac{1}{2}}-(-r)_+^{H-\frac{1}{2}} \right)dw_r \ .
\end{equation}
\end{definition}

\

\begin{remark}
We are aware that, following the subsequent Lemma \ref{lem:justifd}, the above transformation $D^{(\tau)}$ is only defined on a subspace $\Omega^\tau_-\subset \cac((-\infty,\tau];\R^d)$ of full Wiener measure (obtained through a shifting of $\Omega_-$). Nevertheless, using the stability properties reported in the very same Lemma \ref{lem:justifd}, it is readily checked that, on top of being of full Wiener measure, such a space is left invariant by the successive constructions of our strategy. We can therefore assume that at any time $\tau$, the noise component of the system under consideration takes value in $\Omega_-^\tau$, which allows us to justify this slight abuse of notation. 
\end{remark}

If the system is admissible, that is if 
$$\big(Y(\omega)_{\tau_{k-1}(\omega)},\widetilde{Y}(\omega)_{\tau_{k-1}(\omega)},(W(\omega)_{t})_{t\leq \tau_{k-1}(\omega)},(\widetilde{W}(\omega)_{t})_{t\leq \tau_{k-1}(\omega)}\big)$$
defines a $(K,\alpha,\gamma)$-admissible state at time $\tau_{k-1}(\omega)$, one attempts the coupling. Otherwise, one waits for the next one, $i.e.$ one sets $\gw(t)=0$ on $[\tau_{k-1},\tau_k]$ (One will come back below on the duration $\Delta \tau_k:=\tau_k-\tau_{k-1}$).

\smallskip

\noindent Regarding Lemma \ref{lemma:formdrift}, one can remark that Condition \eqref{hairerassumpcond} plays a fundamental role in Step $2$. More precisely, it can be understood as a sufficient condition to ensure the success of the series of attempts involved by Step~2.\smallskip

\noindent 
Item (ii) in Definition \ref{def:admissible-state} is mainly linked to Step 1. Roughly, it ensures that each marginal is in a sufficiently controllable state to couple the positions with a bounded cost.   The fact that the positions live in a compact set 
at the beginning of the attempt is classical. The second condition (on $D^{(\tau)}$) is of course specific to this  non-Markovian setting.

\noindent Finally, note that the first condition will be ensured with probability $1$ at the beginning of the attempt, whereas, of course, the second one will occur only with a (lower-bounded) positive probability. We denote by 
\bqn\label{ensdesadmis}
\AD_k(K,\alpha,\gamma):={\{\tau_{k-1} < \infty\}\cap\{\textnormal{the system is  $(K,\alpha,\gamma)$-admissible at time $\tau_{k-1}$}\}.}
\eqn

\smallskip
\noindent $\rhd$  If the coupling attempt fails, one begins Step $3$, $i.e.$ one waits sufficiently before another attempt. This waiting time will be chosen exponentially proportional to the length of the failing attempt. More precisely,  let $\ell\ge 1$ denote the 
numbers of trials in Step $2$ and adopt the convention, that $\ell=0$ if Step $1$ fails (including the case where the coupling is not attempted since the system is not admissible at time $\tau_{k-1}$). 
Set, {for $k\geq 1$ and $\ell \geq 0$,
\bqn\label{eq:akl}
\FA_{k,\ell}:=\{\tau_{k-1} <\infty\}\cap \{\textnormal{At trial $k$, Step 2 fails exactly at the $\ell$-th attempt}\}
\eqn
and
\bqn\label{eq:bkl-in-general-proof}
{\cal B}_{k,\ell}:=\{\tau_{k-1} <\infty\}\cap \{\textnormal{At trial $k$, Step 2 succeeds at least up to the $\ell$-th attempt (included)}\} \ .
\eqn
Note that with our convention, ${\cal B}_{k,0}$ thus corresponds to the event where Step 1 succeeds.}
Let us finally label the following family of sets, defined inductively, for further use:
\begin{equation}\label{eq:def-cal-Ek}
\ce_{k}: = \lp \tau_{k} < \infty \rp =  \lp \tau_{k-1} < \infty \rp \bigcap \lp \bigcup_{\ell\geq 0} F_{k,\ell}\rp \ .
\end{equation}

{In fact, we will assume in the sequel (see Section \ref{sec:Kadmis}) that for every $\omega\in\FA_{k,\ell}$,
\bqn\label{durationstep3}
\tau_k-\tau_{k-1}^3=\Delta_3(\ell,k)\quad\textnormal{with}\quad \Delta_3(\ell,k):=c_3 \varsigma^k 2^{\beta\ell}
\eqn
where $\tau_{k-1}^3$ denotes the beginning of Step $3$ (in the $k$-$th$ attempt),
and $c_3$, $\varsigma$ and $\beta$ are deterministic parameters that will be calibrated in the course of the reasoning. In particular, conditionally to $\FA_{k,\ell}$, the length of each step will be assumed to be deterministic. During this waiting time, we simply set 
\bqn\label{eq:assumpstep3}
\gw(t)=0 \quad \textnormal{on $[\tau_{k-1}^3,\tau_k]$,}\quad \mbox{ i.e. }
\quad 
\widetilde{W}_t-\widetilde{W}_{\tau_{k-1}^3}
=W_t-W_{\tau_{k-1}^3}.
\eqn}

\subsection{Proof of Theorem \ref{theo:principal}: uniqueness and rate of convergence}\label{sec:prooftheoprinc}

{Denoting by $(Y,\widetilde{Y})$ the coupling of solutions derived from the above-described construction, and with notation (\ref{defi-tau-infty}) in mind, our aim is to prove the following assertion: for every $\alpha\in(0,H)$ and every $p\in (0,\alpha(1-2\alpha))$, there exists a constant $C_{\alpha,p}>0$ such that, for some appropriate calibration of the 3-step scheme, one has
\begin{equation}\label{aim-uniq}
\PE(\tau_\infty>t)\le C_{\alpha,p} t^{-p}\ .
\end{equation}
Optimizing the latter bound with respect to $p$ and then using \eqref{eq:lienTVtauinfty}, the uniqueness assertion of Theorem~\ref{theo:principal} follows immediately, as well as the convergence rate (\ref{main-result:rate-of-conv}), thus completing the proof of our main result.}

\

{The strategy towards (\ref{aim-uniq}) is based on a combination of the successive controls that will be obtained at each stage of our 3-step scheme, namely the results of Proposition \ref{prop:step1}, Proposition \ref{lemme:step2.2} and Proposition~\ref{prop:minokalp}. Although the controls in question will only be shown in the next sections, we have found it important to anticipate these results so as to provide the proof of (\ref{aim-uniq}) right now, which will allow us to both give the reader a general overview of our arguments and also to motivate the forthcoming technical considerations.}

\

{Let us fix $\alpha\in(0,H)$, $p\in (0,\alpha(1-2\alpha))$ and $\beta\in ((1-2\alpha)^{-1},\alpha/p)$. In order to suitably calibrate the scheme, we first consider the parameter $K>0$ given by Proposition \ref{prop:minokalp} for the particular choice $\varepsilon=1/2$. In other words, with the notations of Proposition \ref{prop:minokalp}, we set
\begin{equation}\label{def-K-proof}
K:=\mathbf{K}(1/2,\alpha) \ .
\end{equation}
Then we denote by $\delta=\mathbf{\delta}(K,\alpha)>0$ the constant provided by Proposition \ref{prop:step1}, and we fix a constant $\varsigma=\varsigma(\delta,p,\alpha)>1$ such that $\varsigma^p<(1-\frac{\mathbf{c}_{\alpha,K}\cdot \delta}{2})^{-1}$, where $\mathbf{c}_{\alpha,K} :=(1-\rho^2_{\alpha,K})\prod_{\ell =1}^\infty (1-2^{-\alpha \ell})$, $\rho_{\alpha,K}^2\in (0,1)$ being here the constant given by Proposition \ref{lemme:step2.2}. Finally, with the notations of Proposition \ref{prop:minokalp}, we define $c_3$ by
\begin{equation}\label{def-c-3-proof}
c_3:=\mathbf{c}_{\mathbf{3}}(1/2, \alpha,\beta,\varsigma) \ .
\end{equation}
With these parameters in hand, we will assume in the sequel that Step 2 and Step 3 of the above-described coupling scheme are respectively calibrated along the formulas
\begin{equation}\label{calib-proof}
c_2:=(\mathbf{C}^{\mathbf{2}}_{\alpha,K})^{\frac{1}{2\alpha}} \quad \text{and} \quad \Delta_3(k,\ell):=c_3 \varsigma^k 2^{\beta \ell} \ ,
\end{equation}
where $\mathbf{C}^{\mathbf{2}}_{\alpha,K} \geq 1$ stands for the constant provided by Proposition \ref{lemme:step2.2}. We can now turn to the reasoning towards (\ref{aim-uniq}). In what follows, we denote by $C_p$, resp. $C_{\alpha,p,\beta}$, any generic constant that depends only on $p$, resp. $(\alpha,p,\beta)$.}

\smallskip

\noindent Set $k^*:=\inf\{k\ge1, \Delta \tau_k=+\infty\}${, where $\Delta \tau_k:=\tau_{k}-\tau_{k-1}$}. Using that $\tau_0=0$ {and $\tau_\infty \leq \tau_{k^*-1}+1$, we have 
\bqn\label{b1}
\PE(\tau_\infty>t)\leq \PE\left(\sum_{k=1}^{+\infty}\Delta \tau_k 1_{k^*>k}>{t-1} \right) \ .
\eqn
By Markov's inequality and the fact that $|u+v|^p\le|u|^p+|v|^p$ (because $p\in(0,1]$), we then deduce, provided $t>1$,
\begin{equation*}
\PE\left(\sum_{k=1}^{+\infty}\Delta \tau_k 1_{k^*>k}>{t-1} \right)
\le
\frac{C_p}{t^p}\sum_{k=1}^{+\infty}\ES[ |\Delta \tau_k|^p1_{\{k^*>k\}}] \ ,
\end{equation*}
and as {$(k^*>k)=\bigcup_{\ell \geq 0} F_{k,\ell}$, this yields
\begin{equation}\label{a0}
\PE\left(\sum_{k=1}^{+\infty}\Delta \tau_k 1_{k^*>k}>{t-1}\right)
\le\frac{C_p}{t^p}\sum_{k=1}^{+\infty}\sum_{\ell=0}^{\infty}
\ES\left[|\Delta \tau_k|^p1_{F_{k,\ell}} \right].
\end{equation}
On $F_{k,0}$ (that is the case where Step 1 fails), we have set $\Delta \tau_k=1+c_3\varsigma^k$ according to \eqref{durationstep3}. For $\ell\ge 1$, we have by definition, on the event $F_{k,\ell}$,
\begin{eqnarray*}
\Delta \tau_k&=& 1+c_2 (1+\ldots +2^\ell)+c_3 \varsigma^{k}2^{\beta \ell}\\
&\leq &c_2+c_2 (1+\ldots +2^\ell)+c_3 \varsigma^{k}2^{\beta \ell}\ \leq\  \frac{c_3}{2}+\frac{c_3}{2} (2^{\ell+1}-1)+c_3 \varsigma^{k}2^{\beta \ell} \ \leq \ c_3 \varsigma^{k} 2^{\beta \ell+1} \ ,
\end{eqnarray*}
where we have also used the fact that, according to Proposition \ref{prop:minokalp}, one has $c_3\geq 2c_2$.}
Thus we can recast relation \eqref{a0} as
\begin{equation}\label{a01}
\PE\left(\sum_{k=1}^{+\infty}\Delta \tau_k 1_{k^*>k}>{t-1} \right)
\le
\frac{{C_{\alpha,p,\beta}}}{t^p}\sum_{k=1}^{+\infty} \varsigma^{kp}  \sum_{\ell=0}^{\infty} 2^{\beta\ell p}
\PE(F_{k,\ell}) \ .
\end{equation}
In addition, owing to our definitions \eqref{eq:akl} and \eqref{eq:bkl-in-general-proof}, it is readily checked that $F_{k,\ell}\subset \cb_{k,\ell-1} \cap \cb_{k,\ell}^c$ for every $\ell \geq 1$, and so by Proposition \ref{lemme:step2.2} (remember that we are working under the calibration (\ref{calib-proof}) for $c_2$) we have for every $\ell \geq 1$
\begin{align}\label{applic-step2-1}
\PE({F}_{k,\ell}|{\ce_{k-1}})\le\PE({\cal B}_{k,\ell}^c| {\cal B}_{k,\ell-1})\le { 2^{-\alpha\ell-1}} \ .
\end{align}}
Plugging this inequality into \eqref{a01}, we end up with
\begin{eqnarray}
\PE\left(\sum_{k=1}^{+\infty}\Delta \tau_k 1_{k^*>k}>{t-1} \right)
&\le&
\frac{{C_{\alpha,p,\beta}}}{t^p} \sum_{k=1}^{+\infty}
\varsigma^{k p} \left(\sum_{\ell=0}^{+\infty} 2^{(\beta p-\alpha)\ell}\right) 
\PE\lp \ce_{k-1}  \rp \nonumber \\
&\le&
\frac{{C_{\alpha,p,\beta}}}{t^p} \sum_{k=1}^{+\infty} \varsigma^{k p} \PE\lp \ce_{k-1}  \rp \ ,\label{bou2}
\end{eqnarray}
where we have used the fact that $\beta p-\alpha <0$. Now
\begin{equation}\label{decompo-tau-k-fini}
\PE\lp \ce_{k-1} \rp
=\prod_{m=1}^{k-1} \PE({\cal E}_{m}|{\cal E}_{m-1})=\prod_{m=1}^{k-1} (1-\PE({\cal E}_{m}^c|{\cal E}_{m-1}))\ ,
\end{equation}
and {observe that for every $m\geq 1$,
\begin{equation}\label{lw-bou}
\PE({\cal E}_{m}^c|{\cal E}_{m-1})\ge \PE\big(\Delta \tau_m=+\infty\, |\, \ce_{m-1}\cap \AD_m(K,\alpha,\gamma)\big) \PE(A_m(K,\alpha,\gamma)|{\cal E}_{m-1}) \ .
\end{equation}
At this point, we can first apply Proposition \ref{prop:minokalp} (remember that $K$ and $c_3$ are defined by (\ref{def-K-proof}) and (\ref{def-c-3-proof})) to derive that
\begin{equation}\label{sec-term}
\PE(A_m(K,\alpha,\gamma)|{\cal E}_{m-1}) \geq \frac12 \ .
\end{equation}
On the other hand, using the decomposition
\begin{equation}\label{applic-step2-2}
\PE\big(\Delta \tau_m=+\infty\, |\, \ce_{m-1}\cap \AD_m(K,\alpha,\gamma)\big)=\PE\big(Y_{\tau_{m-1}+1}=\widetilde{Y}_{\tau_{m-1}+1}\, |\, \ce_{m-1}\cap \AD_m(K,\alpha,\gamma)\big) \prod_{\ell=1}^{+\infty}\PE({\cal B}_{m,\ell}|{\cal B}_{m,\ell-1}) \ ,
\end{equation}
we can easily combine the results of Proposition \ref{prop:step1} and Proposition \ref{lemme:step2.2} to assert that
\begin{equation}\label{prem-term}
\PE(\Delta \tau_m=+\infty|\ce_{m-1}\cap A_m(K,\alpha,\gamma)) \geq \mathbf{c}_{\alpha,K}\, \delta \ ,
\end{equation}
where $\mathbf{c}_{\alpha,K} >0$ and  $\delta>0$ have been introduced at the beginning of the proof. }

\smallskip

\noindent
{Injecting (\ref{sec-term})-(\ref{prem-term}) into  (\ref{decompo-tau-k-fini}) and going back to (\ref{lw-bou}), we have thus shown that
\begin{equation}\label{bou3}
\sum_{k=1}^{+\infty}\varsigma^{kp} \PE\lp \tau_{k-1}<+\infty  \rp\le 
\sum_{k=1}^{+\infty} \varsigma^{pk}\lp 1-\frac{\mathbf{c}_{\alpha,K}\,\delta}{2} \rp^{k-1} \ ,
\end{equation}
and the latter quantity is known to be finite due to our choice of $\varsigma$ (that is $\varsigma$ such that such that $\varsigma^p<(1-(\mathbf{c}_{\alpha,K} \delta)/2)^{-1}$). The expected bound (\ref{aim-uniq}) then follows from the combination of (\ref{b1}), (\ref{bou2}) and~(\ref{bou3}), and this achieves the proof of Theorem \ref{theo:principal}.}

\smallskip

\

{
The remainder of the paper is now devoted to the proof of the intermediate results at the core of the above arguments, \textit{i.e.} the results of Proposition \ref{prop:step1}, Proposition \ref{lemme:step2.2} and Proposition \ref{prop:minokalp}, which actually correspond to controls at Step 1, Step 2 and Step 3, respectively. As we mentioned it in the introduction, the main difficulties of this procedure (more precisely, the most innovative part with respect to the analysis in \cite{fontbona-panloup}) essentially lie in the \emph{hitting} Step 1.  We thus propose in the next subsection to give a heuristic description of the coupling construction during this first stage, before we go into the technical details of Sections \ref{sec:rougheqinholder} and \ref{sec:roughcoupling}.}

\subsection{Heuristic description of the coupling system in Step 1}\label{subsec:heuri-coupl-sys}

At time $\tau_{k-1}$, set $a_0:=Y_{\tau_{k-1}}$ and $a_1=\widetilde{Y}_{\tau_{k-1}}$. As reported in Subsection \ref{subsec:general3stepschem}, our aim in Step 1 (the \emph{hitting} step) will be to build $g_X$ on the interval $[\tau_{k-1},\tau_{k-1}+1]$ in such a way that $Y_{\tau_{k-1}+1}=\widetilde{Y}_{\tau_{k-1}+1}$ with strictly positive probability. This construction will actually be the topic of both Sections \ref{sec:rougheqinholder} and \ref{sec:roughcoupling}. However, let us try here to give an idea, at some heuristic level, of the motivations behind this forthcoming strategy.

\

To this end, let us simplify the framework by assuming that $\tau_{k-1}=0$, $d=1$, and consider for the moment the case of a smooth deterministic driver $x$. In brief, our purpose is to exhibit a triplet of paths $(y^0_t,y^1_t,g_t)_{t\in [0,1]}$ satisfying the system
$$
\begin{cases}
dy^0_t=b(y^0_t)\, dt+\si(y^0_t) \, dx_t\\
dy^1_t=b(y^1_t)\, dt+\si(y^1_t) \, \big( dx_t+g_t \, dt \big) 
\end{cases}
,
$$
as well as the constraints $y^0_0=a_0$, $y^1_0=a_1$ and $y^0_1=y^1_1$. Using our invertibility assumptions on $\sigma$ (that is, Hypothesis $\mathbf{(H3)}$), this amounts to finding $(y^0_t,y^1_t,h_t)_{t\in [0,1]}$ such that 
$$
\begin{cases}
dy^0_t=b(y^0_t)\, dt+\si(y^0_t) \, dx_t\\
dy^1_t=b(y^1_t)\, dt+\si(y^1_t) \, dx_t+h_t \, dt  
\end{cases}
,
$$
and $y^0_0=a_0$, $y^1_0=a_1$, $y^0_1=y^1_1$. In fact, let us consider the slightly more general issue of exhibiting a family of paths $(y^\xi_t,h^\xi_t)_{t\in [0,1],\xi\in [0,1]}$ that satisfy the equation
\begin{equation}\label{equat-heur-y-xi}
dy^\xi_t=b(y^\xi_t)\, dt+\si(y^\xi_t) \, dx_t+h^\xi_t \, dt
\end{equation}
as well as the constraints $y^\xi_0=a_0+\xi (a_1-a_0)$, $y^0_1=y^1_1$ and $h^0_.\equiv 0$. Then, using the basic identity $y^0_1-y^1_1=\int_0^1 d\xi \, \partial_\xi y^\xi_1$, we are led to the following sufficient formulation of the problem: finding a family $(y^\xi_t,h^\xi_t)_{t\in [0,1],\xi\in [0,1]}$ that satisfies both (\ref{equat-heur-y-xi}) and the constraints $y^\xi_0=a_0+\xi (a_1-a_0)$, $\partial_\xi y^._1\equiv 0$, $h^0_.\equiv 0$.

\smallskip

A natural way to answer the latter question is to let the so-called tangent path (associated with $y$) come into the picture. Namely, set $h^\xi_t:=-\int_0^\xi d\eta \, \jmath^\eta_t$, where for each $\xi$, $\jmath^\xi$ stands for the solution of the equation
$$d\jmath^\xi_t=b'(y^\xi_t) \jmath^\xi_t \, dt+\sigma'(y^\xi_t) \jmath^\xi_t \, dx_t \quad , \quad \jmath^\xi_0=a_1-a_0 \ .$$
With this specific choice of $h^\xi_t$, it is readily checked that the two paths $t\mapsto \partial_\xi y^\xi_t$ and $t\mapsto \jmath^\xi_t (1-t)$ satisfy the very same equation
$$dz^\xi_t=\big[b'(y^\xi_t) z^\xi_t-\jmath^\xi_t\big] \, dt+\sigma'(y^\xi_t)z^\xi_t \, dx_t \ .$$
Accordingly, $\partial_\xi y^\xi_t=\jmath^\xi_t (1-t)$ and the above constraints $\partial_\xi y^._1\equiv 0$, $h^0_.\equiv 0$ are indeed satisfied.

\smallskip

As a conclusion of these considerations, the problem now reduces to solving the parametric (or functional-valued) system
\begin{equation}\label{sys-rough-heuri}
\begin{cases}
dy^\xi_t=\big[ b(y^\xi_t)-\int_0^\xi d\eta \, \jmath^\eta_t\big]\, dt+\si(y^\xi_t) \, dx_t\\
d\jmath^\xi_t=b'(y^\xi_t) \jmath^\xi_t \, dt+\sigma'(y^\xi_t) \jmath^\xi_t \, dx_t 
\end{cases},
\end{equation}
with initial conditions $y_0^\xi=a_0+\xi(a_1-a_0)$, $\jmath^\xi_0=a_1-a_0$, and $\xi$ varying in $[0,1]$. This new expression of the problem can of course be extended to the case of rough $x$ (and to any dimension $d$), with $(y^\xi_t,\jmath^\xi_t)$ understood as a rough solution of (\ref{sys-rough-heuri}), in the sense of Definition \ref{def:sol-davie}. Just as above, once endowed with such a solution $(y,\jmath)$, defining $g_X$ as $g_X(t):=-\sigma(y^1_t)^{-1}\int_0^1 d\eta\, \jmath^\eta_t$ would then close the procedure.

\

Unfortunately, as the reader can easily check it, the vector fields involved in (\ref{sys-rough-heuri}) do not meet the usual boundedness assumptions that guarantee the existence of a \emph{global} rough solution defined on $[0,1]$ (compare for instance with the conditions in \cite[Theorem 6.1]{Da07} or in \cite[Theorem 10.26]{FV-bk}). In fact, we have not been able to establish such a global existence in the general situation, and we even suspect that an explosion phenomenon might occur in some cases. What we will prove in the next section is a weaker result according to which global existence on $[0,1]$ holds provided the norm $\lVert \mathbf{x}\rVert_{\ga;[0,1]}$ is small enough (depending on $a_0$ and $a_1$).

\smallskip

Going back to the stochastic setting of our study (where $x=X$ is a fBm), the latter existence result is still not satisfying though, because the required smallness condition on $\lVert \mathbf{X}\rVert_{\ga;[0,1]}$ implicitly involves the past trajectory $(W_t)_{t\leq 0}$, which is somehow fixed (and not necessarily small) at this stage of our three-step procedure. In order to overcome this difficulty, we shall appeal (once again) to the 'past-innovation' decomposition (\ref{decompo-fbm-ter}) of $X$, and rewrite the \emph{hitting} system as
\begin{equation}\label{sys-rough-heuri-sto}
\begin{cases}
dY^\xi_t=\big[ b(Y^\xi_t) \, dt-\int_0^\xi d\eta \, J^\eta_t \, dt+\si(Y^\xi_t) \, dD_t\big]+\si(Y^\xi_t) \, dZ_t\\
dJ^\xi_t=\big[ b'(Y^\xi_t) J^\xi_t \, dt+\sigma'(Y^\xi_t) J^\xi_t \, dD_t \big]+\sigma'(Y^\xi_t) J^\xi_t \, dZ_t 
\end{cases},
\end{equation}
with initial conditions $Y_0^\xi=a_0+\xi(a_1-a_0)$, $J^\xi_0=a_1-a_0$, $\xi\in [0,1]$. Now recall that, at the sole price of a singularity at time $0$, the process $D$ can be considered as smooth and therefore the whole terms into brackets in (\ref{sys-rough-heuri-sto}) can somehow be seen as drift terms, to be distinguished from the real rough perturbation driven by $Z$. Based on these properties and still using a pathwise approach (as developed in Section \ref{sec:rougheqinholder} below), we will derive the following refined version of the previous existence statement: there exists a unique global solution to (\ref{sys-rough-heuri-sto}) on $[0,1]$ provided the norm of $\lVert \mathbf{Z}\rVert_{\ga;[0,1]}$ (which no longer depends on $(W_t)_{t\leq 0}$) is small enough. This result, essentially summed up by Proposition \ref{prop:appli}, will turn out to be sufficient for our purpose.

\section{Singular rough equations}\label{sec:rougheqinholder}

\indent
This section is devoted to the presentation of a natural setting to study the hitting system (\ref{sys-rough-heuri-sto}) (properly extended to $\R^d$) and exhibit sharp conditions on $(D,Z)$ for this system to admit a unique global solution on $[0,1]$. To this end, it will turn out to be fundamental that the trajectories of the process $D$ should somehow be considered as differentiable paths that induce some drift term in the equation. However, as we have evoked it in Section \ref{sec:decompo-fbm-ter} (see also Lemma \ref{lem:justifd} below), this differentiability assumption is not exactly satisfied, due to a possible singularity at time $0$ for the derivative of $D$.

\smallskip

With these observations in mind, the purpose of the section is essentially twofold: 

\smallskip

$\bullet$ Introduce appropriate \emph{singular} extensions of the Hölder spaces defined in Section \ref{subsec:classic-topos} and then extend the classical study of rough systems to this setting, for general Banach-valued equations;

\smallskip

$\bullet$ Exhibit sufficient conditions on the driver for these rough singular equations to have a unique global solution, even in situations where the usual boundedness requirements on the vector fields are not met (see Hypothesis (VF2) below), which is the case for (\ref{sys-rough-heuri-sto}).  

\

The effective application of these general (pathwise) considerations to the particular fractional system~(\ref{sys-rough-heuri-sto}) will then be analyzed in Section \ref{subsec:rough-hitting}.

\smallskip

From now on and for the rest of the section, we fix two parameters: $\ga \in (\frac13,\frac12)$ (for the general Hölder roughness) and $\beta\in [\ga,1]$ (encoding the singularity at time $0$).

\subsection{Singular rough solutions and well-posedness results}

Our singular extensions of the usual Hölder spaces are specifically defined through the following seminorms: given a Banach space $V$, an interval $I\subset [0,1]$ and two parameters $\al\in (0,1],\mu\geq \al$, set, for any map $f:I^2 \to V$, resp. $f:I^3 \to V$, 
\begin{equation}\label{topo-singu-1}
\cn[f;\cac_{2;\beta}^{\al,\mu}(I;V)]
:=
\max \bigg( \sup_{s<t\in I} \frac{\| f_{st}\|_V}{\lln t-s \rrn^{\al}},
\sup_{0<s<t\in I} \frac{\|f_{st}\|_V }{\lln t-s \rrn^{\mu}s^{\beta-1}} \bigg) \ , 
\end{equation}
resp.
\begin{equation}\label{topo-singu-2}
\cn[f;\cac_{3;\beta}^{\al,\mu}(I;V)]:=\max \bigg( \sup_{s<u<t\in I} \frac{\|f_{sut}\|_V}{\lln t-s \rrn^{\al}},\sup_{0< s<u<t\in I} \frac{\|f_{sut}\|_V}{\lln t-s \rrn^{\mu}s^{\beta-1}} \bigg) \ ,
\end{equation}
and then 
\begin{equation}\label{topo-singu-3}
 \cac_{1;\beta}^{\al,\mu}(I;V):=\big\{ f\in \cac_1(I;V): \ \delta f \in \cac_{2;\beta}^{\al,\mu}(I;V) \big\} \ .
\end{equation}

\smallskip

Of course, it holds that $\cac_{i;\beta}^{\al,\mu}(I;V)\subset \cac_i^\al(I;V)$ and $\cac_{i;1}^{\al,\mu}(I;V)= \cac_i^\mu(I;V)$. What actually led us to the above definitions is the following readily-checked (and relatively sharp) inclusion:
\begin{lemma}
Let $\mathcal{E}^1_\ga([0,1];\R^d)$ be the space introduced in Notation \ref{nota:w-ga-1}.
It holds that $\mathcal{E}^1_\ga([0,1];\R^d)\subset \cac_{1;\ga}^{\ga,1}([0,1];\R^d)$ and for every $g\in \mathcal{E}^1_\ga([0,1];\R^d)$,
\begin{equation}
\cn[\delta g;\cac_{2;\ga}^{\ga,1}([0,1];\R^d)]\leq c_\ga \! \vvvert g \vvvert_{1;\ga} \ .
\end{equation}
\end{lemma}

\

Let us now introduce the related notion of a singular rough solution. In the sequel, given two Banach spaces $V,W$ and a smooth map $F:V \to W$, we will denote by $D^{(\ell)}F: V\to \mathcal{L}(V^{\otimes \ell};W)$ the $\ell$-th derivative of $F$, understood in the usual Fréchet sense.

\begin{definition}\label{def:general-sol}
Consider a path $h\in \cac_{1;\beta}^{\ga,1}([0,1];\R^m)$ and a $\ga$-rough path $\mathbf{z}=(z,\mathbf{z}^{\mathbf{2}})$, in the sense of Definition \ref{defi-rough-path}. Then, for any fixed Banach space $V$, any interval $I=[t_0,t_1]\subset [0,1]$, any $v_0\in V$ and all smooth vector fields
$$B:V \to \mathcal{L}(\R^m;V) \quad , \quad \varSigma:V \to \mathcal{L}(\R^n;V) \ , $$
we call $y\in \cac_1^\ga(I;V)$ a solution (on $I$) of the equation
\begin{equation}\label{general-equation}
dy_t =B(y_t) \, dh_t+\varSigma(y_t) \, d\mathbf{z}_t \quad , \quad y_{t_0}=v_0 \ ,
\end{equation}
if the two-parameter path $R^y$ defined as
$$R^y_{st}:=(\der y)_{st}-B_i(y_s) \, (\der h^i)_{st}-\varSigma_j(y_s) \, (\der z^j)_{st}-(D\varSigma_j \cdot \varSigma_k)(y_s) \, \mathbf{z}^{\mathbf{2},jk}_{st} \ $$
belongs to $\cac_{2;\beta}^{\ga,\mu}(I;V)$, for some parameter $\mu>1$. Here, the notation $D\varSigma_j \cdot \varSigma_k$ stands for
$$(D\varSigma_j \cdot \varSigma_k)(v):=(D\varSigma_j)(v)(\varSigma_k(v)) \ , \ \text{for every} \ v \in V \ .$$
\end{definition}

\

\begin{remark}
We are aware that the space $\mathcal{E}_\ga^1$ could also be continuously embedded into the space of paths with finite $1$-variation, so that the whole problem could certainly receive an analog treatment (with $h$ still considered as inducing a drift term) in the $p$-variation setting used in \cite{Da07,deya-hofmanova,FV-bk}, instead of our singular Hölder setting. Nevertheless, switching the equation to a $p$-variation framework could expose us to the risk of a loss of topological sharpness in the results, with solutions possibly leaving the space of Hölder paths (see for instance the general definition of a solution in \cite[Definition 3.1]{Da07}). This is not the case in the above formulation, where the solution is still expected to belong to $\cac_1^\ga$.
\end{remark}

Let us now turn to the presentation of the main results of this section about existence/uniqueness of a solution for the rough singular equation (\ref{general-equation}). We will either be concerned with the classical situation of bounded vector fields (Hypothesis (VF1)) or the more general possibility of linear growth (Hypothesis (VF2)).

\

\noindent
\textbf{Hypothesis (VF1).} The vector field $\varSigma$ and all its derivatives are uniformly bounded on $V$. Besides, the derivative $D^{(1)}B: V\to \mathcal{L}(V;\mathcal{L}(\R^m;V))$ is uniformly bounded on $V$.

\

\noindent
\textbf{Hypothesis (VF2).} The following bounds on $B$ and $\varSigma$ hold true: for all $\ell \geq 0$, 
\begin{equation}\label{regu-b-sigma-1}
\|(D^{(\ell)}B)(v)\| \lesssim 1+\|v\| \quad , \quad \|(D^{(\ell)}\varSigma)(v)\| \lesssim 1+\|v\| \ ,
\end{equation}
and also, for every $v,w\in V$,
\begin{equation}\label{regu-b-sigma-2}
\| (D \varSigma \cdot \varSigma) (v)\| \lesssim 1+\|v\| \quad , \quad \| (D \varSigma \cdot \varSigma) (v)-(D \varSigma \cdot \varSigma) (w)\| \lesssim \|v-w\| \{ 1+\|v\| \} \  .
\end{equation}

\

\begin{theorem}[(VF1)-situation]\label{theo:gene-vf1}
Under Hypothesis (VF1), and for any $v_0\in V$, Equation (\ref{general-equation}) admits a unique solution on $[0,1]$ with initial condition $v_0$, in the sense of Definition \ref{def:general-sol}.
\end{theorem}

\begin{theorem}[(VF2)-situation]\label{theo:gene}
Under Hypothesis (VF2), the following assertions hold true:

\smallskip

\noindent
$(i)$ For any $v_0\in V$, Equation (\ref{general-equation}) admits at most one solution on $[0,1]$ with initial condition $v_0$, in the sense of Definition \ref{def:general-sol}.

\smallskip

\noindent
$(ii)$ For every $K\geq 1$, there exists $M_K>0$ such that if $\|v_0\|\leq K$, $\cn[\der h;\cac_{2;\beta}^{\ga,1}([0,1])] \leq K$ and $\|\mathbf{z}\|_{\ga;[0,1]} \leq M_K$, then Equation (\ref{general-equation}) admits a unique solution $y$ on $[0,1]$ with initial condition $v_0$, in the sense of Definition \ref{def:general-sol}. Besides,
\begin{equation}\label{boun-solut}
\cn[y;\cac_1^0([0,1];V)]+\cn[y;\cac_1^\ga([0,1];V)] \leq C(K) \ ,
\end{equation}
for some growing function $C:\R^+\to \R^+$.
\end{theorem}

\

Just as in Section \ref{subsec:lyapou}, and in the same spirit as in~\cite{Da07}, our proof for both Theorem~\ref{theo:gene-vf1} and Theorem~\ref{theo:gene} relies on the examination of the discrete scheme associated with the equation. Set $t_i=t_i^n:=\frac{i}{2^n}$, $\mathcal{P}_n:=\{t_i: \ i=0,\ldots,2^n\}$ and define $y^n$ along the iterative formula: $y^n_0=v_0$ and
$$(\der y^n)_{t_i t_{i+1}}=B(y^n_{t_i})\, (\der h)_{t_i t_{i+1}}+\varSigma(y^n_{t_i})\, (\der z)_{t_i t_{i+1}}+(D \varSigma \cdot \varSigma)(y^n_{t_i})\, \mathbf{z}^{\mathbf{2}}_{t_i t_{i+1}} \ .$$
Then, for every $s,t \in \mathcal{P}_n$, set
$$R^n_{st}:=(\der y^n)_{st}-B(y^n_s)\, (\der h)_{st}-\varSigma(y^n_s)\, (\der z)_{st}-(D \varSigma \cdot \varSigma)(y^n_s)\, \mathbf{z}^{\mathbf{2}}_{st}\ ,$$
noting in particular that $R^n_{t_it_{i+1}}=0$. We will also consider the paths
$$L^n_{st}:=(\der y^n)_{st}-\varSigma(y^n_s)\, (\der z)_{st}-\, (D \varSigma \cdot \varSigma)(y^n_s)\, \mathbf{z}^{\mathbf{2}}_{st} \quad \big(=R^n_{st}+B(y^n_s) (\der h)_{st}\ \big)$$
and
$$Q^n_{st}:=(\der y^n)_{st}-\varSigma(y^n_s)\, (\der z)_{st} \ .$$
Finally, for every $s<t\in [0,1]$, we will write $\llbracket s,t\rrbracket=\llbracket s,t\rrbracket_n:=[s,t]\cap \cp_n$, and we extend (or rather restrict) the norms (\ref{topo-singu-1})-(\ref{topo-singu-3}) to discrete paths as
$$\cn[f;\cac_{2;\beta}^{\al,\mu}(\llbracket s,t\rrbracket;V)]:=\max \bigg( \sup_{u<v\in \llbracket s,t\rrbracket} \frac{\| f_{uv}\|_V}{\lln v-u \rrn^{\al}},\sup_{0<u<v\in \llbracket s,t\rrbracket} \frac{||f_{uv}\|_V}{\lln v-u \rrn^{\mu}u^{\beta-1}} \bigg) \ ,$$
with a similar definition for $\cn[f;\cac_{i;\beta}^{\al,\mu}(\llbracket s,t\rrbracket;V)]$, $i\in \{1,3\}$.

\smallskip

The whole key towards the desired estimates lies in the following ``singular sewing lemma'':

\begin{lemma}\label{lem:discr-sewing}
Let $0< \al \leq \la \leq 1$, $\mu_1 \geq 1$ and $\mu_2 >1$. Then there exists a constant $c_{\al,\la,\mu_1,\mu_2} >0$ such that for every path $G:\mathcal{P}_n^2 \to V$ and all $s\leq t \in \mathcal{P}_n$, one has
$$
\cn[G;\cac_{2;\la}^{\al,\mu_1 \wedge \mu_2}(\llbracket s,t\rrbracket;V)] \leq
c_{\al,\la,\mu_1,\mu_2} \big\{ \cm_{\la}^{\al,\mu_1}\big[G;\llbracket s,t\rrbracket\big]+\cn[\der G;\cac_{3;\la}^{\al,\mu_2}(\llbracket s,t \rrbracket;V)] \big\} \ ,
$$
where we have set, if $s=\frac{p}{2^n}$ and $t=\frac{q}{2^n}$,
$$\cm_{\la}^{\al,\mu_1}\big[G;\llbracket s,t\rrbracket\big]:=\max\bigg( \sup_{p\leq i\leq q} \frac{\| G_{t_it_{i+1}}\|}{|t_{i+1}-t_i|^\al} \ , \ \sup_{p+1 \leq i\leq q}\frac{\| G_{t_it_{i+1}}\|}{|t_{i+1}-t_i|^{\mu_1}t_i^{\la-1}} \bigg) \ .$$
\end{lemma}

\begin{proof}
See Appendix \ref{proofoflem:discr-sewing}.
\end{proof}

\subsection{Existence of a solution in the (VF2)-situation}

\begin{proposition}\label{prop:l-n}
Let Hypothesis (VF2) prevail and assume additionally that 
$$\cn[\der h;\cac_{2;\beta}^{\ga,1}([0,1];V)] \leq K \ , \ \text{for some} \ K\geq 1 \ .$$
 Then there exists a constant $c_0$ (which depends only on $B$, $\varSigma$, $\ga$ and $\beta$) such that if we set $T_0=T_0(K):=\min\big(1,(c_0 K)^{-6/(3\ga-1)} \big)$, the following assertion holds true for every $k\leq 1/T_0$: if $\|\mathbf{z}\|_{\ga;[0,1]} \leq \big( 1+\|y^n_{kT_0}\|\big)^{-1}$, then
\begin{equation}\label{bound-l-n-local}
\cn[L^n;\cac_{2;\beta}^{\ga,1}(\llbracket kT_0,(k+1)T_0\wedge 1\rrbracket;V)] \leq c_0\, K \big\{1+\|y^n_{kT_0}\| \big\} \ .
\end{equation}
\end{proposition}

\begin{proof} The strategy consists in an iteration procedure over the points of the partition. So, assume that (\ref{bound-l-n-local}) holds true on an interval $\llbracket 0, t_q\rrbracket$, with $t_q \leq T_0$ (for some time $T_0$ to be determined along the proof). In other words, assume that
\begin{equation}\label{bound-l-n-proof}
\cn[L^n;\cac_{2;\beta}^{\ga,1}(\llbracket 0,t_q\rrbracket;V)] \leq c_K \big\{1+\|y^n_{0}\| \big\} \ ,
\end{equation}
where we denote from now on $c_K:=c_0 K$ (for some constant $c_0$ to be fixed later on). Due to (\ref{regu-b-sigma-1}) and (\ref{regu-b-sigma-2}), it is then easy to check that the following bounds hold true as well:
\begin{equation}\label{boun-infi-proof}
\cn[y^n;\cac_1^0(\llbracket 0,t_q\rrbracket)] \lesssim \{1+\|y^n_0\| \} \{1+c_K T_0^\ga\} \ ,
\end{equation}
and
\begin{equation}\label{boun-k-n-proof}
\max \big( \cn[y^n;\cac_1^\ga(\llbracket 0,t_q\rrbracket)] , \cn[Q^n;\cac_{2;\beta}^{\ga,2\ga}(\llbracket 0,t_q\rrbracket)]  \big) \lesssim \{1+\|y^n_0\| \} \{1+c_K \} \ .
\end{equation}
Now, in order to extend (\ref{bound-l-n-proof}) on $\llbracket 0,t_{q+1}\rrbracket$ (assuming that $t_{q+1}\leq T_0$), let us first apply Lemma \ref{lem:discr-sewing} to $L^n$ and assert that
\begin{equation}\label{sewing-l-n-proof}
\cn[L^n;\cac_{2;\beta}^{\ga,1}(\llbracket 0,t_{q+1}\rrbracket)]\lesssim \cm_{\beta}^{\ga,1}\big[L^n;\llbracket 0,t_{q+1}\rrbracket)]+\cn[\der L^n;\cac_{3;\beta}^{\ga,3\ka}(\llbracket 0,t_{q+1}\rrbracket)] \ ,
\end{equation}
where we set from now on $\ka:=\frac12 \big( \frac13+\ga\big)$, so that $1<3\ka<3\ga$.

\smallskip

As far as the first term is concerned, we can use the fact that $R^n_{t_it_{i+1}}=0$, and then (\ref{regu-b-sigma-1}) and (\ref{boun-infi-proof}), to deduce that
\begin{eqnarray*}
\cm_{\beta}^{\ga,1}\big[L^n;\llbracket 0,t_{q+1}\rrbracket)]&=&\cn[B(y^n) \, \der h;\cac_{2;\beta}^{\ga,1}(\llbracket 0,t_{q+1}\rrbracket)]\\
& \leq& K \cdot \cn[B(y^n);\cac_1^0(\llbracket 0,t_q\rrbracket] \ \lesssim \ K \{1+\|y^n_0\| \} \{1+c_K T_0^\ga\} \ . 
\end{eqnarray*}

\smallskip

In order to estimate $\cn[\der L^n;\cac_{3;\beta}^{\ga,3\ka}(\llbracket 0,t_{q+1}\rrbracket)]$, let us first rely on Chen relation and decompose the increments of $L^n$ as $\der L^n=I^i \, \der z^i+II^i \, \der z^i+III^{ij} \, \mathbf{z}^{\mathbf{2},ij}$, where we have set
\begin{equation}\label{defi-i}
I^i_{st}:=\int_0^1 d\la \, (D\varSigma_i)(y^n_s+\la(\der y^n)_{st}) \, Q^n_{st} \ ,
\end{equation}
\begin{equation}\label{defi-ii}
II^i_{st}:=\int_0^1 d\la \, \big[(D\varSigma_i)(y^n_s+\la (\der y^n)_{st})-(D\varSigma_i)(y^n_s)  \big] \, \varSigma_j(y^n_s) \, (\der z^{j})_{st} \ ,
\end{equation}
\begin{equation}\label{defi-iii}
III^{ij}_{st}:=\der (D\varSigma_i \cdot \varSigma_j)(y^n)_{st} \ .
\end{equation}
For $I^i \, \der z^i$, we can combine (\ref{regu-b-sigma-1}), (\ref{boun-infi-proof}) and (\ref{boun-k-n-proof}) to get that
\begin{eqnarray*}
\cn[I^i \, \der z^i;\cac_{3;\beta}^{\ga,3\ka}(\llbracket 0,t_{q+1}\rrbracket)] &\lesssim & T_0^{3(\ga-\ka)} \big( \{1+\|y^n_0\| \}\{ 1+c_K\}\big)^2\, \|\mathbf{z}\|_{\ga;[0,1]}\\
&\lesssim& \{1+\|y^n_0\| \}\{ 1+T_0^{3(\ga-\ka)} c_K^2\} \cdot \Big( \{1+\|y^n_0\| \} \|\mathbf{z}\|_{\ga;[0,1]} \Big)\\
&\lesssim& \{1+\|y^n_0\| \}\{ 1+T_0^{3(\ga-\ka)} c_K^2\} \ ,
\end{eqnarray*}
where we have used the assumption $\{1+\|y^n_0\| \} \|\mathbf{z}\|_{\ga;[0,1]}\leq 1$ to derive the third inequality.

\smallskip

\noindent
With similar arguments, we can show that
\begin{eqnarray*}
\cn[II^i \, \der z^i;\cac_{3;\beta}^{\ga,3\ka}(\llbracket 0,t_{q+1}\rrbracket)] 
&\lesssim& \{1+\|y^n_0\| \}\{ 1+T_0^{3(\ga-\ka)} c_K^3\} \cdot \Big( \{1+\|y^n_0\| \}^2 \|\mathbf{z}\|_{\ga;[0,1]}^2 \Big)\\
&\lesssim& \{1+\|y^n_0\| \}\{ 1+T_0^{3(\ga-\ka)} c_K^3\} \ .
\end{eqnarray*}
Finally, thanks to the second estimate in (\ref{regu-b-sigma-2}), we obtain that
$$\cn[III^{ij} \, \mathbf{z}^{\mathbf{2},ij};\cac_{3;\beta}^{\ga,3\ka}(\llbracket 0,t_{q+1}\rrbracket)] 
\lesssim \{1+\|y^n_0\| \}\{ 1+T_0^{3(\ga-\ka)} c_K^2\} \ .$$

\smallskip

Going back to (\ref{sewing-l-n-proof}), we have shown that, for some constant $c_1$ depending only on $B$, $\varSigma$ and $(\ga,\ka,\beta)$, 
$$\cn[L^n;\cac_{2;\beta}^{\ga,1}(\llbracket 0,t_{q+1}\rrbracket)] \leq \{1+\|y^n_0\|\} \cdot \Big( c_1K \{1+T_0^{3(\ga-\ka)}c_K^3\} \big) \ .$$
Let us now set $c_0:=2c_1$, $c_K:=c_0K$ and $T_0:=\min(1,(2c_1 K)^{-1/(\ga-\ka)})$, in such a way that
$$c_1K \{1+T_0^{3(\ga-\ka)}c_K^3\} \leq c_K \ ,$$
and accordingly $\cn[L^n;\cac_{2;\beta}^{\ga,1}(\llbracket 0,t_{q+1}\rrbracket)] \leq c_K\{1+\|y^n_0\|\}$ as desired. 

\

This iteration procedure allows us to extend the bound (\ref{bound-l-n-proof}) over the interval $\llbracket 0,T_0\rrbracket$. Then it is easy to see that the very same arguments can be used for any interval $\llbracket kT_0,(k+1)T_0\rrbracket$, which completes the proof.

\end{proof}

\begin{corollary}\label{coro:unif-est}
Let Hypothesis (VF2) prevail and assume additionally that 
$$\cn[\der h;\cac_{2;\beta}^{\ga,1}([0,1];V)] \leq K \quad  \text{and} \quad \|v_0\| \leq K \ , \ \text{for some}\ K\geq 1 \ .$$
Then there exists $M_K>0$ such that if $\|\mathbf{z}\|_{\ga;[0,1]} \leq M_K$, one has
\begin{equation}\label{boun-quanti-inter}
\sup_{n\geq 0}  \, \max \bigg( \cn[y^n;\cac_1^0(\llbracket 0,1\rrbracket)],\cn[y^n;\cac_1^\ga(\llbracket 0,1\rrbracket)],\cn[Q^n;\cac_{2;\beta}^{\ga,2\ga}(\llbracket 0,1\rrbracket)],\cn[L^n;\cac_{2;\beta}^{\ga,1}(\llbracket 0,1\rrbracket)] \bigg)\leq C(K) \ , 
\end{equation}
for some growing function $C:\R^+ \to \R^+$. As a result, under the same assumptions and if $\|\mathbf{z}\|_{\ga;[0,1]} \leq M_K$, it holds that
\begin{equation}\label{boun-r-n}
\sup_{n\geq 0} \, \cn[R^n;\cac_{2;\beta}^{\ga,3\ga}(\llbracket 0,1\rrbracket)] \leq D(K) \ ,
\end{equation}
for some growing function $D:\R^+ \to \R^+$.
\end{corollary}

\begin{proof}
Using (\ref{bound-l-n-local}) as well as its spin-offs (\ref{boun-infi-proof}) and (\ref{boun-k-n-proof}), it is not hard to exhibit a growing sequence $(c_k)$ that depends only on $(B,\varSigma,\ga,\beta)$ (and not on $K$) such that the following property holds true: for every $k\geq 0$, if $\|\mathbf{z}\|_{\ga;[0,1]} \leq (1+c_k \{1+K\})^{-1}$, then one has both
\begin{equation}\label{boun-seq-1}
\cn[y^n;\cac_1^0(I_k) ] \leq c_{k+1}  \{1+K\} 
\end{equation}
and
\begin{equation}\label{boun-seq-2}
\max \big( \cn[y^n;\cac_1^\ga(I_k)] , \cn[Q^n;\cac_{2;\beta}^{\ga,2\ga}(I_k)],\cn[L^n;\cac_{2;\beta}^{\ga,1}(I_k)]  \big) \leq c_{k+1}  \{1+K^2\} \ ,
\end{equation}
where we have set $I_k:=\llbracket kT_0,(k+1)T_0\rrbracket$. As a result, if we denote by $N_K$ the smallest integer such that $T_0 N_K \geq 1$ and assume that $\|\mathbf{z}\|_{\ga;[0,1]} \leq M_K:=(1+c_{N_K}(1+K))^{-1}$, then both bounds (\ref{boun-seq-1}) and (\ref{boun-seq-2}) hold true for $k=0,\ldots,N_K-1$. The extension of these local bounds into global ones (that is, on the interval $\llbracket 0,1\rrbracket$) is then a matter of standard arguments, which achieves the proof of (\ref{boun-quanti-inter}).

\smallskip

As far as (\ref{boun-r-n}) is concerned, apply first Lemma \ref{lem:discr-sewing} to the path $R^n$, which, since $R^n_{t_it_{i+1}}=0$, entails that
$$
\cn[R^n;\cac_{2;\beta}^{\ga,3\ga}(\llbracket 0,1\rrbracket;V)] \lesssim \cn[\der R^n;\cac_{3;\beta}^{\ga,3\ga}(\llbracket 0,1 \rrbracket;V)] \big\} \ .
$$
Then, just as in the proof of Proposition \ref{prop:l-n}, observe that we can decompose the increments of $R^n$ as
\begin{equation}\label{decompo-incr-r-n}
(\der R^n)_{sut}=\der B(y^n)_{su} (\der h)_{ut}+(\der L^n)_{sut}=\der B(y^n)_{su} (\der h)_{ut}+I^i_{su} \der z^i_{ut}+II^i_{su}  \der z^i_{ut}+III^{ij}_{su} \mathbf{z}^{\mathbf{2},ij}_{ut} \ ,
\end{equation}
where the paths $(I,II,III)$ have been defined through (\ref{defi-i})-(\ref{defi-iii}). The conclusion is now easy to derive from the bound (\ref{boun-quanti-inter}).

\end{proof}

\begin{proof}[Proof of Theorem \ref{theo:gene}, point $(ii)$] Consider the sequence (still denoted by $y^n$) of continuous paths on $[0,1]$ defined through the linear interpolation of the points of the previous (discrete) sequence $y^n$. Define $M_K$ as in Corollary \ref{coro:unif-est} and assume that $\|\mathbf{z}\|_{\ga;[0,1]} \leq M_K$. Then it is readily checked that (\ref{boun-quanti-inter}) gives rise to a uniform bound for $\cn[y^n;\cac_1^\ga([0,1];V)]$, and we can therefore conclude about the existence of a path $y\in \cac_1^\ga([0,1];V)$, as well as a subsequence of $y^n$ (that we still denote by $y^n$), such that $y^n \to y$ in $\cac_1^\ka([0,1];V)$ for every $0<\ka<\ga$.

\smallskip

The fact that $y$ actually defines a solution of (\ref{general-equation}) is essentially obtained by passing to the limit in the uniform estimate (\ref{boun-r-n}). The details of this (easy) procedure can for instance be found at the end of \cite[Section 3.3]{Deya}. As for the bound (\ref{boun-solut}), it is a straightforward consequence of (\ref{boun-quanti-inter}).
\end{proof}

\subsection{Existence of a solution in the (VF1)-situation}

Under Hypothesis (VF1), the exhibition of a uniform bound for $\cn[R^n;\cac_{2;\beta}^{\ga,\mu}(\llbracket 0,1\rrbracket)]$ (with $\mu>1$) essentially follows the same general procedure as in the classical ('non-singular') situation treated in \cite{Da07} or \cite{FV-bk}. As we here consider slightly more specific topologies, let us briefly review the result at the core of this procedure.

\begin{proposition}\label{prop:bound-vf1}
Let Hypothesis (VF1) prevail and assume additionally that 
$$\cn[\der h;\cac_{2;\beta}^{\ga,1}([0,1];\R^m)] \leq K \ , \ \text{for some}\ K\geq 1 \ .$$
Also, fix a parameter $\ka$ such that $1<3\ka<3\ga$. Then there exists a a constant $c_0$ (which depends only on $B$, $\varSigma$, $\ga$, $\beta$ and $\ka$) such that if we set 
$$T_0=T_0(\lVert \mathbf{z}\rVert_{\ga},K):=\min\Big(1,\Big( c_0\big\{1+\lVert \mathbf{z}\rVert_{\ga}\big\} K\Big)^{-1/(\ga-\ka)} \Big) \ ,$$
the following property holds true: for every $0<T_1<T_0$ and every $k\leq 1/T_1$,
\begin{equation}\label{firs-est-vf1}
\cn[R^n;\cac_{2;\beta}^{\ga,3\ka}(\llbracket kT_1,(k+1)T_1\wedge 1 \rrbracket)] \leq c_0\, K \big\{1+\|y^n_{kT_1}\| \big\} \ .
\end{equation}
\end{proposition}

\begin{proof}
Just as in the proof of Proposition \ref{prop:l-n}, the strategy consists in an iteration procedure over the points of $\mathcal{P}_n$. The argument actually relies on the following two readily-checked assertions: $(i)$ If $\cn[R^n;\cac_{2;\beta}^{\ga,3\ka}(\llbracket s,t\rrbracket)] \leq c_0\, K \big\{1+\|y^n_{s}\| \big\}$, then one has
$$\max\big(\cn[y^n;\cac_1^\ga(\llbracket s,t\rrbracket)],\cn[Q^n;\cac_{2;\beta}^{\ga,2\ga}(\llbracket s,t\rrbracket)]\big)\leq c_1 \, K \big[ \lVert \mathbf{z}\rVert_\ga+\big\{ 1+\lVert y^n_s\rVert\big\}\big\{ 1+c_0\big\}\big\{ 1+K|t-s|^\ga\big\}\big]$$
for some constant $c_1$ that depends only on $(B,\varSigma)$; $(ii)$ With decomposition (\ref{decompo-incr-r-n}) in mind, one has 
$$\cn[\der R^n;\cac_{3;\beta}^{\ga,3\ka}(\llbracket s,t\rrbracket)] \leq c_2 \lln t-s \rrn^{3(\ga-\ka)}\big[ \cn[y^n;\cac_1^\ga(\llbracket s,t\rrbracket)] \{1+K+\lVert \mathbf{z}\rVert_\ga^2 \}+\cn[Q^n;\cac_{2;\beta}^{\ga,2\ga}(\llbracket s,t\rrbracket)]\lVert \mathbf{z}\rVert_\ga \big] $$
for some constant $c_2$ that depends only on $(B,\varSigma)$.

\smallskip

It is now easy to inject $(i)$ and $(ii)$ into the iteration scheme exhibited in the previous section for $L^n$ (note that we can additionally use the fact that $R^n_{t_it_{i+1}}=0$ here). The details of the procedure are therefore left to the reader.
\end{proof}

\begin{proof}[Proof of the existence statement in Theorem \ref{theo:gene-vf1}] 
Starting from (\ref{firs-est-vf1}) and using the same steps as in the proof of Corollary \ref{coro:unif-est}, one easily gets uniform estimates for both $\cn[y^n;\cac_1^\ga(\llbracket 0,1\rrbracket)]$ and $\cn[R^n;\cac_{2;\beta}^{\ga,3\ka}(\llbracket 0, 1 \rrbracket)]$. The derivation of a solution then follows from the same convergence argument as in the above proof of Theorem \ref{theo:gene}, point $(ii)$.
\end{proof}

\subsection{Uniqueness of the solution}
It is a well-known fact that uniqueness statements are usually less demanding than existence statements as far as global boundedness of the vector fields is concerned. Accordingly, in opposition with the previous existence proof (where specific sharp estimates had to be displayed), the strategy towards uniqueness essentially follows the same lines as in the standard situation. We briefly review the transposition of the main arguments in this singular setting.

\smallskip

Assume here that either Hypothesis (VF1) or Hypothesis (VF2) prevails and consider two solutions $U,\Uti$ of (\ref{general-equation}) with identical initial conditions. Then set 
$$R_{st}=R(y)_{st}:=(\der y)_{st}-B_i(y_s) \, (\der h^i)_{st}-\varSigma_j(y_s) \, (\der z^j)_{st}-(D\varSigma_j \cdot \varSigma_k)(y_s) \, \mathbf{z}^{\mathbf{2},jk}_{st} \ ,$$
$$Q_{st}=Q(y)_{st}:=(\der y)_{st}-\varSigma_j(y_s) \, (\der z^j)_{st} \ ,$$
and similarly $\Rti:=R(\yti)$, $\Qti:=Q(\yti)$. Also, fix $\mu$, resp. $\tilde{\mu} >1$ such that $\cn[R;\cac_{2;\beta}^{\ga,\mu}([0,1])]<\infty$, resp. $\cn[\widetilde{R};\cac_{2;\beta}^{\ga,\mu}([0,1])]<\infty$, as well as a parameter $\ka$ satisfying both $\frac13< \ka < \ga $ and $3\ka < \mu \wedge \tilde{\mu}$.

\begin{lemma}
There exists a finite constant $c_{R,\Rti}>0$ such that for every $s<t\in \mathcal{P}_n$, one has
\begin{equation}\label{boun-uni-1}
\cn[R-\Rti;\cac_{2;\beta}^{\ka,3\ka}(\llbracket s,t\rrbracket)] \leq c_{R,\Rti}\cdot \big\{ 2^{-n\ep}+\cn[\der(R-\Rti);\cac_{3;\beta}^{\ka,3\ka}(\llbracket s,t\rrbracket)]\big\} \ ,
\end{equation}
where $\ep:=\inf(\ga-\ka,(\mu \wedge \tilde{\mu})-3\ka)>0$.
\end{lemma}

\begin{proof}
It is a mere application of Lemma \ref{lem:discr-sewing}. Observe indeed that
\begin{eqnarray*}
\cm_{\beta}^{\ka,3\ka}\big[R-\Rti;\llbracket s,t\rrbracket\big]&\leq & \cm_{\beta}^{\ka,3\ka}\big[R;\llbracket s,t\rrbracket\big]+\cm_{\beta}^{\ka,3\ka}\big[\Rti;\llbracket s,t\rrbracket\big]\\
&\leq& c_{R,\Rti} \, \{2^{-n(\ga-\ka)}+2^{-n(\mu_1-3\ka)}+2^{-n(\mu_2-3\ka)} \} \ .
\end{eqnarray*}
\end{proof}

\begin{lemma}
There exists a finite constant $C_{y,\yti}>0$ such that for every $s<t \in \mathcal{P}_n$, one has
\begin{equation}\label{boun-uni-2}
\cn[\der(R-\Rti);\cac_{3;\beta}^{\ka,3\ka}(\llbracket s,t\rrbracket)] \leq C_{y,\yti} \lln t-s \rrn^{\ga-\ka} \, \cn_{\beta}^{\ka,2\ga}\big[(y,\yti);\llbracket s,t \rrbracket\big] \ ,
\end{equation}
where we have set
\begin{equation}\label{norm-proc-contr}
\cn_{\beta}^{\ka,2\ga}\big[(y,\yti);\llbracket s,t \rrbracket\big]:=\cn[y-\yti;\cac_1^0(\llbracket s,t\rrbracket)]+\cn[y-\yti;\cac_1^\ka(\llbracket s,t\rrbracket)]+\cn[Q-\Qti;\cac_{2;\beta}^{\ka,2\ga}(\llbracket s,t\rrbracket)] \ .
\end{equation}
\end{lemma}

\begin{proof}
First, note that the increments of $R$ (or $\Rti$) can be decomposed just as the increments of $R^n$ in the proof of Corollary \ref{coro:unif-est} (see (\ref{decompo-incr-r-n})), which allows us to write
$$
\der(R-\Rti)_{sut}=\der(B(y)-B(\yti))_{su}\, \der h_{ut}
+\big[ I^i_{su}-\widetilde{I}^i_{su}\big] \, \der z^i_{ut}+\big[ II^i_{su}-\widetilde{II}^i_{su}\big] \, \der z^i_{ut}+\big[ III^{ij}_{su}-\widetilde{III}^{ij}_{su}\big] \, \mathbf{z}^{\mathbf{2},ij}_{ut} \ ,
$$
where the paths $I,II,III$, resp. $\widetilde{I},\widetilde{II},\widetilde{III}$, are defined along (\ref{defi-i})-(\ref{defi-iii}) (replace $(y^n,Q^n)$ with $(y,Q)$, resp. $(\yti,\Qti)$). The bound (\ref{boun-uni-2}) is then obtained through standard differential-calculus arguments based on relations~(\ref{regu-b-sigma-1}) and~(\ref{regu-b-sigma-2}).
\end{proof}

\begin{proof}[Proof of Theorem \ref{theo:gene}, point $(i)$, and uniqueness property of Theorem \ref{theo:gene-vf1}] 
Consider the above setting and notations. First, going back to the very definitions of $(K,R)$ and $(\Kti,\Rti)$, it is not hard to check that for every $s<t\in \mathcal{P}_n$, one has, with the notation (\ref{norm-proc-contr}),
$$\cn_{\beta}^{\ka,2\ga}\big[ (y,\yti);\llbracket s,t \rrbracket \big] \leq c_{y,\yti} \Big\{ \|y_s-\yti_s \|+\lln t-s \rrn^{\ka} \cn_{\beta}^{\ka,2\ga}\big[ (y;\yti);\llbracket s,t\rrbracket\big]+\cn\big[R-\Rti;\cac_{2;\beta}^{\ka,3\ka}(\llbracket s,t\rrbracket\big]  \Big\} \ ,$$
where the constant $c_{y,\yti}$ does not depend on $n$. We can then combine (\ref{boun-uni-1})-(\ref{boun-uni-2}) and assert that for every $s<t\in \mathcal{P}_n$, 
$$\cn_{\beta}^{\ka,2\ga}\big[ (y,\yti);\llbracket s,t \rrbracket \big] \leq c_{y,\yti} \Big\{ \|y_s-\yti_s \|+\lln t-s \rrn^{\ga-\ka} \cn_{\beta}^{\ka,2\ga}\big[ (y;\yti);\llbracket s,t\rrbracket\big]+2^{-n\ep}  \Big\} \ .$$
The uniqueness result is now immediate. Indeed, for $T_0>0$ such that $c_{y,\yti} \, T_0^{\ga-\ka} \leq \frac12$, and since $y_0=\yti_0$, we first get that
$$\cn\big[y-\yti;\cac_1^0(\llbracket 0,T_0 \rrbracket)\big] \leq \cn_{\beta}^{\ka,2\ga}\big[(y,\yti);\llbracket 0,T_0\rrbracket \big] \leq C_{y,\yti} \cdot 2^{-n\ep} \ ,$$
and accordingly $y_t=\yti_t$ for every $t\in [0,T_0]$. The argument can then be repeated on $[T_0,2T_0]$, $[2T_0,3T_0]$, and so on.
\end{proof}

\section{Hitting step}\label{sec:roughcoupling}

Keeping in mind the strategy sketched out in Section \ref{subsec:heuri-coupl-sys}, the route to Step 1, that is the hitting step, is now quite clear: we need to check that the vector fields involved in the hitting system (\ref{sys-rough-heuri-sto}) do satisfy the assumptions of the previous section, and then see how the conditions therein exhibited (for the driver) can be injected into the general coupling machinery. 

\smallskip

We recall that we have fixed $H\in (1/3,1/2)$, $\ga\in (1/3,H)$, as well as vector fields $b:\R^d \to \R^d$ and $\si:\R^d \to \mathcal{L}(\R^d,\R^d)$ satisfying Hypotheses (\textbf{H1}) and (\textbf{H3}) (note that Hypothesis (\textbf{H2}) is not required at this stage of the procedure).

\subsection{Rough hitting}\label{subsec:rough-hitting}
This first section focuses on the hitting issue at the level of the general (deterministic) rough system, and therefore it settles the bases for our forthcoming stochastic analysis. Let us recall that the space $\mathcal{E}^2_\ga([0,1];\R^d)$ has been introduced through Notation \ref{nota:w-ga-1}, and let us fix two paths $h\in \mathcal{E}^2_\ga([0,1];\R^d)$, $z\in \mathcal{C}^\ga([0,1];\R^d)$, assuming in addition that $z$ can be canonically lifted into a $\ga$-rough path $\mathbf{z}:=\mathfrak{L}(z)$, in the sense of Definition~\ref{defi-canonic-rp}.

\begin{lemma}\label{lem:vf-appli}
Consider the Banach space $V_2:=\mathcal{W}^{1,\infty}([0,1];\R^d) \times L^\infty([0,1];\R^d)$, and define the vector fields $(B,\varSigma)$ on $V$ along the following formulas:
$$B^i_0\begin{pmatrix}
y\\
\jmath
\end{pmatrix}(\xi):=
\begin{pmatrix}
b^i(y(\xi))-\int_0^\xi d\eta \, \jmath^i(\eta)\\
(\partial_k b^i)(y(\xi))  \jmath^k(\xi)
\end{pmatrix}  \quad , \quad B^i_j\begin{pmatrix}
y\\
\jmath
\end{pmatrix}(\xi)=\varSigma^i_j\begin{pmatrix}
y\\
\jmath
\end{pmatrix}(\xi):= 
\begin{pmatrix}
\si_j^i(y(\xi))\\
(\partial_k \si_j^i)(y(\xi))  \jmath^k(\xi)
\end{pmatrix}$$
for $i,j=1,\dots,d$. Then, under Hypothesis ${\bf (H1)}$, the pair $(B,\varSigma)$ satisfies Hypothesis (VF2).
\end{lemma}

\begin{proof}
We have the following explicit expressions: 
$$(D\varSigma_j^i)\begin{pmatrix}
y\\
\jmath
\end{pmatrix}\begin{pmatrix}
y_1\\
\jmath_1
\end{pmatrix}= 
\begin{pmatrix}
(\partial_k \si_i^j)(y) \, y_1^k\\
(\partial_k \partial_{\ell} \si_j^i)(y) \, y_1^\ell\, \jmath^k+(\partial_k \si_i^j)(y) \, \jmath_1^k
\end{pmatrix} \ ,$$
and more generally, for every $q\geq 1$,
\begin{multline*}
(D^{(q)}\varSigma_j^i)\begin{pmatrix}
y\\
\jmath
\end{pmatrix}
\left(\begin{pmatrix}
y_1\\
\jmath_1
\end{pmatrix},\ldots,\begin{pmatrix}
y_q\\
\jmath_q
\end{pmatrix}\right)\\
= 
\begin{pmatrix}
(\partial_{k_1} \cdots \partial_{k_q} \si_j^i)(y) \, y_1^{k_1}\cdots y_q^{k_q} \\
(\partial_k \partial_{k_1}\cdots \partial_{k_q}  \si_j^i)(y) \, y_1^{k_1}\cdots y_q^{k_q}\, \jmath^k+(\partial_{k_1}\cdots \partial_{k_q}  \si_j^i)(y)  \sum_{r=1,\ldots,q} y_1^{k_1}\cdots y_{r-1}^{k_{r-1}} \jmath_r^{k_r}y_{r+1}^{k_{r+1}}\cdots y_q^{k_q}
\end{pmatrix} \ .
\end{multline*}
In particular,
$$(D\varSigma_j^i \cdot \varSigma_k)\begin{pmatrix}
y\\
\jmath
\end{pmatrix}= 
\begin{pmatrix}
(\partial_\ell \si_j^i)(y) \si_k^\ell(y)\\
\big\{(\partial_m \partial_{\ell} \si_j^i)(y)  \si_k^\ell(y) +(\partial_\ell \si_i^j)(y) (\partial_m\si_k^\ell)(y)\big\}\, \jmath^m
\end{pmatrix} \ .$$
Based on these formulas, the two conditions (\ref{regu-b-sigma-1}) and (\ref{regu-b-sigma-2}) for $\varSigma$ are immediate. We can then exhibit a similar expression for $D^{(q)}B_0$.

\end{proof}

Combining Lemma \ref{lem:vf-appli} with the well-posedness results of Theorems \ref{theo:gene-vf1} and \ref{theo:gene} yields the following statement:

\begin{proposition}\label{prop:appli}
Under Hypothesis ${\bf (H1)}$, the following assertions hold true:

\smallskip

\noindent
$(a)$ Let $V_1:=L^\infty([0,1];\R^d)^2$. Then for every $A\in V_1$ and every smooth function $\vp:\R^d \to \R^d$ bounded with bounded derivatives, the rough system
\begin{eqnarray}
dy_t(\xi)&=&\Big[\vp(b(y_t(\xi)))-\int_0^\xi \varphi(\jmath_t(\eta)) \, d\eta \Big] \, dt+\si(y_t(\xi)) \, dh_t+\si(y_t(\xi)) \, d\mathbf{z}_t \ ,\label{sys-1-bis}\\
d\jmath_t(\xi)&=&(\partial_k b)(y_t(\xi)) \vp(\jmath_t(\xi))_k \, dt+(\partial_k \si)(y_t(\xi))\vp(\jmath_t(\xi))_k \, dh_t +(\partial_k \si)(y_t(\xi))\vp(\jmath_t(\xi))_k \, d\mathbf{z}_t\ ,\label{sys-2-bis}
\end{eqnarray}
with initial condition $(y_0,\jmath_0)=A$, admits a unique solution 
$$(y,\jmath)=: \Psi_{V_1}(A,\vp,(h,\mathbf{z}))\in \cac_1^\ga([0,1];V_1) \ ,$$
in the sense of Definition \ref{def:general-sol}.
\smallskip

\noindent
$(b)$ Let $V_2:=\mathcal{W}^{1,\infty}([0,1];\R^d) \times L^\infty([0,1];\R^d)$ and pick $A\in V_2$. Assume that $\|A\|_{V_2} \leq K$ and $\vvvert h\vvvert_{1;\ga} \leq K$, for some fixed $K\geq 1$. Then there exists a constant $M_K >0$ such that if $\|\mathbf{z}\|_{\ga;[0,1]} \leq M_K$, the conclusion of point $(a)$ is still true for $\vp\equiv \id$ and $V_1$ replaced with $V_2$, and one has
\begin{equation}\label{bou-psi-v2-k}
\cn\big[  \Psi_{V_2}(A,\id,(h,\mathbf{z}));\cac_1^0([0,1];V_2)] \leq C(K) \ ,
\end{equation}
for some growing function $C:\R^+ \to \R^+$.
\end{proposition}

\begin{proof}
Point $(b)$ is obtained through the combination of Lemma \ref{lem:vf-appli} and Theorem \ref{theo:gene}. As for point $(a)$, it suffices to observe that for every fixed $\vp$, the vector fields involved in (\ref{sys-1-bis})-(\ref{sys-2-bis}) satisfy Hypothesis (VF1), and we can therefore appeal to Theorem \ref{theo:gene-vf1} to conclude in this case.

\end{proof}

Let us now rigourously check that when $\vp\equiv \id$, the hitting system (\ref{sys-1-bis})-(\ref{sys-2-bis}) indeed satisfies the desired property, namely offering a way to see two rough solutions (with different initial conditions and drivers differing from a sole drift term) hit a time $1$.
\begin{proposition}\label{prop:consistency}
Let $V_2:=\mathcal{W}^{1,\infty}([0,1];\R^d) \times L^\infty([0,1];\R^d)$ and consider a $V_2$-valued solution $(y,\jmath)$ on $[0,1]$ (in the sense of Definition \ref{def:general-sol}) of the rough system (\ref{sys-1-bis})-(\ref{sys-2-bis}) with $\vp\equiv \id$ and inital condition
$$y_0(\xi)=(1-\xi) a_0+\xi a_1\quad , \quad \jmath_0(\xi)=a_1-a_0 \ ,$$
for fixed $a_0,a_1 \in \R^d$. Then the following assertions hold true:

\smallskip

\noindent
$(a)$ The $\R^d$-valued path $y^{(0)}:=y_.(0)$, is the solution on $[0,1]$ (in the sense of Definition \ref{def:sol-davie}) of the rough equation
\begin{equation}\label{prop:identif-2}
dy_t=b(y_t) \, dt+\si(y_t) \, d\mathbf{x}_t \quad , \quad y_0=a_0 \ ,
\end{equation}
where $\mathbf{x}$ is the canonical rough path above $x:=z+h$.

\smallskip

\noindent
$(b)$ The $\R^d$-valued path $y^{(1)}:=y_.(1)$ is the solution on $[0,1]$ (in the sense of Definition \ref{def:sol-davie}) of the rough equation
\begin{equation}\label{rgh-sys-btilde}
dy_t=b(y_t) \, dt+\si(y_t) \, d\widetilde{\mathbf{x}}_t \quad , \quad y_0=a_1 \ ,
\end{equation}
where $\widetilde{\mathbf{x}}$ is the canonical rough path above $\widetilde{x}:=z+(h+e)$, with
$$e_t:=-\int_0^t ds \, \si(y_s(1))^{-1}\int_0^1 d\eta \, \jmath_s(\eta) \ .$$
\noindent
$(c)$ It holds that $y^{(0)}_1=y^{(1)}_1$.
\end{proposition}

\begin{proof}

$(a)$ Let $y$ be a solution of (\ref{prop:identif-2}) in the sense of Definition \ref{def:sol-davie}. Then clearly it is also a solution in the sense of Definition \ref{def:general-sol}, and by Corollary \ref{coro:identif-sol}, we can conclude that $y$ is a solution of the equation
\begin{equation}\label{back-equa-std-2}
dy_t=\big[b(y_t) \, dt+\si(y_t) \, dh_t\big]+\si(y_t) \, d\mathbf{z}_t\quad , \quad y_0=a_0 \ ,
\end{equation} 
(that is, in Definition \ref{def:general-sol}, we take $V:=\R^d$, $m-1=n=d$, $B_0(y):=b(y)$, $B_i(y):=\si_i(y)$, $\varSigma_i(y):=\si_i(y)$, $h_t \leftrightarrow (t,h_t)$, $\mathbf{z} \leftrightarrow\mathbf{z}$). The conclusion then comes from the uniqueness statement contained in Theorem \ref{theo:gene}, since $y_.(0)$ trivially satisfies Equation (\ref{back-equa-std-2}) as well.

\smallskip

\noindent
$(b)$ As above, observe that due to the regularity of the path $e$ and thanks to Corollary \ref{coro:identif-sol}, the solution of~(\ref{rgh-sys-btilde}) (in the sense of Definition \ref{def:sol-davie}) is also the solution of
\begin{equation}\label{back-equa-std-3}
dy_t=b(y_t) \, dt+\si(y_t) \, d(h+e)_t+\si(y_t) \, d\mathbf{z}_t\quad , \quad y_0=a_1 \ ,
\end{equation} 
in the sense of Definition \ref{def:general-sol} (that is, with $V:=\R^d$, $m-1=n=d$, $B_0(y):=b(y)$, $B_i(y):=\si_i(y)$, $\varSigma_i(y):=\si_i(y)$, $h_t \leftrightarrow (t,h_t+e_t)$, $\mathbf{z} \leftrightarrow\mathbf{z}$). It turns out that the path $w:=y_.(1)$ satisfies Equation (\ref{back-equa-std-3}) as well. This can be easily derived from the fact that
\begin{eqnarray*}
\lefteqn{\bigg| \si(w_s) \, (\der e)_{st}-\bigg( -\int_0^1 d\eta \, \jmath_s(\eta)\bigg) \, (t-s)\bigg|}\\ & =& |\si(w_s)| \, \bigg| \int_s^t dr \int_0^1 d\eta \, \bigg[ \si(w_r)^{-1} \jmath_r(\eta)-\si(w_s)^{-1}\jmath_s(\eta)\bigg]\bigg|\\
&\lesssim &\lln t-s\rrn^{1+\ga} \big\{ \cn[\jmath;\cac_1^0([0,1])]\, \cn[w;\cac_1^\ga([0,1])]+\cn[\jmath;\cac_1^\ga([0,1])] \big\} \ .
\end{eqnarray*}
Therefore, just as for point $(a)$, we can conclude with the help of the uniqueness property stated in Theorem~\ref{theo:gene}.

\smallskip

\noindent
$(c)$ The assertion relies on the following identity: for every $t,\xi\in [0,1]$, one has
\begin{equation}\label{identif-for-coupl}
(\partial_\xi y_t)(\xi)=\jmath_t(\xi) (1-t) \ .
\end{equation}
It can indeed be checked that, when seen as paths with values in $V:=L^\infty([0,1];\R^d)^3$, the triplets $(y,\jmath,\partial_\xi y)$ and $(y,\jmath,g)$ (where we have set $g_t(\xi):=\jmath_t(\xi) (1-t)$) are both solution of the system obtained by adding to (\ref{sys-1-bis})-(\ref{sys-2-bis}) the third equation
$$dg_t=[(\partial_k b)(y_t) g^k_t-\jmath_t] \, dt+(\partial_k \si)(y_t)g^k_t \, dh_t+(\partial_k \si)(y_t)g^k_t \, d\mathbf{z}_t \quad , \quad g_0(\xi)=a_1-a_0 \ .$$
The conclusion is now immediate:
$$y^{(1)}_1-y^{(0)}_1=y_1(1)-y_1(0)=\int_0^1  g_1(\xi) \, d\xi=0 \ .$$

\end{proof}

Before we summarize the previous results into a single statement (Theorem \ref{theo:pathwise-rough-coupl} below), let us introduce an auxiliary system which will later serve us as an ingredient to 'invert' the hitting system. This system (or rather this family of systems) takes values in $V_1:=L^\infty([0,1])^2$, and is defined for every smooth compactly-supported $\vp:\R^d \to \R^d$ as follows:
\begin{eqnarray}
d\bar{y}_t(\xi)&=&\bigg[\vp(b(\bar{y}_t(\xi)))-\int_0^\xi \varphi(\bar{\jmath}_t(\eta)) \, d\eta+\si(\bar{y}_t(\xi))\si(\bar{y}_t(1))^{-1}\int_0^1 \vp(\bar{\jmath}_t(\eta)) \, d\eta  \bigg] \, dt\nonumber\\
& &\hspace{3cm}+\si(\bar{y}_t(\xi)) \, dh_t+\si(\bar{y}_t(\xi)) \, d\mathbf{z}_t \ ,\label{syst-tilde-1}\\
d\bar{\jmath}_t(\xi)&=&\bigg[(\partial_k b)(\bar{y}_t(\xi)) \vp(\bar{\jmath}_t(\xi))_k+ (\partial_k \si)(\bar{y}_t(\xi)) \vp(\bar{\jmath}_t(\xi))_k\si(\bar{y}_t(1))^{-1}\int_0^1 \vp(\bar{\jmath}_t(\eta)) \, d\eta\bigg] dt\nonumber\\
& & \hspace{3cm} +(\partial_k \si)(\bar{y}_t(\xi))\vp(\bar{\jmath}_t(\xi))_k \, dh_t +(\partial_k \si)(\bar{y}_t(\xi))\vp(\bar{\jmath}_t(\xi))_k \, d\mathbf{z}_t\ .\label{syst-tilde-2}
\end{eqnarray}
with initial condition $(y_0(\xi),\jmath_0(\xi))=((1-\xi)a_0+\xi a_1,a_1-a_0)$. It is not hard to see that for every such fixed $\vp$, the vector fields involved in (\ref{syst-tilde-1})-(\ref{syst-tilde-2}) satisfy the conditions of Theorem \ref{theo:gene-vf1} (that is, Hypothesis (VF1)), and therefore the system admits a unique solution 
\begin{equation}\label{defi-flow-tilde}
(\bar{y},\bar{\jmath})=:\bar{\Psi}_{V_1}((a_0,a_1),\vp,(h,\mathbf{z})) \in \cac_1^\ga([0,1];V_1)\ .
\end{equation}

\smallskip

\noindent
\textbf{Notation.} For all $h:[0,1] \to \R^d$ and $g\in L^1([0,1];\R^d)$, we set, for all $t\in [0,1]$, 
\begin{equation}\label{eq:def-T}
T(h,g)_t:=h_t+\int_0^t g_s \, ds \ .
\end{equation}

\begin{theorem}\label{theo:pathwise-rough-coupl}
Fix $K\geq 1$ and $V_1:=L^\infty([0,1];\R^d)^2$. Then there exists a smooth compactly-supported function $\vp_K:\R^d \to \R^d$ such that, for all $a_0,a_1\in \R^d$, the following assertions hold true:

\smallskip

\noindent
$(i)$ The system
\begin{eqnarray}
dy_t(\xi)&=&\Big[\vp_K(b(y_t(\xi)))-\int_0^\xi \varphi_K(\jmath_t(\eta)) \, d\eta \Big] \, dt+\si(y_t(\xi)) \, dh_t+\si(y_t(\xi)) \, d\mathbf{z}_t \ ,\nonumber\\
d\jmath_t(\xi)&=&(\partial_k b)(y_t(\xi)) \vp_K(\jmath_t(\xi))_k \, dt+(\partial_k \si)(y_t(\xi))\vp_K(\jmath_t(\xi))_k \, dh_t +(\partial_k \si)(y_t(\xi))\vp_K(\jmath_t(\xi))_k \, d\mathbf{z}_t\ ,\nonumber
\end{eqnarray}
with initial condition $(y_0(\xi),\jmath_0(\xi))=((1-\xi)a_0+\xi a_1,a_1-a_0)$, admits a unique solution 
\begin{equation}\label{defi-flow-k}
(y,\jmath)=: \Psi_{V_1}((a_0,a_1),\vp_K,(h,\mathbf{z}))\in \cac_1^\ga([0,1];V_1) \ ,
\end{equation}
in the sense of Definition \ref{def:general-sol}.

\smallskip

\noindent
$(ii)$ There exists a constant $M_K>0$ such that if $\max(|a_0|,|a_1|,|a_1-a_0|) \leq K$, $\vvvert h\vvvert_{1;\ga}\leq K$ and $\|\mathbf{z}\|_{\ga;[0,1]} \leq M_K$, then, defining $(y,\jmath)$ through (\ref{defi-flow-k}), one has: $(ii$-$a)$ the $\R^d$-valued path $y^{(0)}=y_.(0)$ is the solution of
$$dy_t=b(y_t) \, dt+\si(y_t) \, d\mathbf{x}_t \ , \ y_0=a_0 \ ,$$
where $\mathbf{x}$ is the canonical rough path above $x:=z+h$; $(ii$-$b)$ the $\R^d$-valued path $y^{(1)}=y_.(1)$ is the solution of
\begin{equation}\label{eq:y(1)}
dy_t=b(y_t) \, dt+\si(y_t) \, d\widetilde{\mathbf{x}}_t \ , \ y_0=a_1 \ ,
\end{equation}
where $\widetilde{\mathbf{x}}$ is the canonical rough path above $\widetilde{x}:=z+T(h,g)$, with
\begin{equation}\label{drif-g}
g_{t}:=-\si(y_t(1))^{-1}\int_0^1 d\eta \, \vp_K(j_t(\eta)) \quad , \quad (y,j):=\Psi_{V_1}((a_0,a_1),\vp_K,(h,\mathbf{z})) \  ;
\end{equation}
$(ii$-$c)$ $y^{(0)}_1=y^{(1)}_1$.

\smallskip

\noindent
$(iii)$ With notations (\ref{defi-flow-tilde}) and (\ref{defi-flow-k}) in mind, we have the following identities:
\begin{equation}\label{ident-psi-psiti-0}
(y,j)
=\Psi_{V_1}((a_0,a_1),\vp_K,(h,\mathbf{z}))=\bar{\Psi}_{V_1}( (a_0,a_1),\vp_K,( T(h,g),\mathbf{z})) \ ,
\end{equation}
with $g$ defined just as in (\ref{drif-g}), and
\begin{equation}\label{ident-psi-psiti}
(\bar{y},\bar{\jmath}):=\bar{\Psi}_{V_1}((a_0,a_1),\vp_K,(h,\mathbf{z}))=\Psi_{V_1}( (a_0,a_1),\vp_K,( T(h,\bar{g}),\mathbf{z})) \ ,
\end{equation}
with
\begin{equation}\label{drif-gti}
\bar{g}(t):=\si(\bar{y}_t(1))^{-1} \int_0^1 d\eta\, \vp_K(\bar{\jmath}_t(\eta)) \ .
\end{equation}
\end{theorem}

\begin{proof}
With the notation $C(K)$ used in Proposition \ref{prop:appli} point $(b)$, consider any smooth function $\vp_K:\R^d \to \R^d$ such that $\vp_K \equiv \id$ on $[-C(K),C(K)]^d$ and $\vp_K(x)=0$ for every $| x| \geq 2\, C(K)$. Then $(i)$ follows immediately from Proposition \ref{prop:appli} point $(a)$. Besides, owing to (\ref{bou-psi-v2-k}), it is clear that by defining $M_K$ just as in Proposition \ref{prop:appli} point $(b)$, and under the assumptions of the above point $(ii)$, one has the identity
$$\Psi_{V_1}((a_0,a_1),\vp_K,(h,\mathbf{z}))=\Psi_{V_2}(A,\id,(h,\mathbf{z})) \quad \text{with} \ A(\xi):=((1-\xi)a_0+\xi a_1,a_1-a_0) \ .$$
Therefore, the three points $(ii$-$a)$-$(ii$-$b)$-$(ii$-$c)$ can be readily deduced from Proposition \ref{prop:consistency}.

\smallskip

In order to prove $(iii)$, observe first that with the notations in (\ref{drif-g}), one has, at least at a formal level,
\begin{eqnarray*}
dy_t(\xi)&=&\bigg[\vp_K(b(y_t(\xi)))-\int_0^\xi \varphi_K(\jmath_t(\eta)) \, d\eta
+\si(y_t(\xi))\si(y_t(1))^{-1}\int_0^1 \vp_K(\jmath_t(\eta)) \, d\eta  \bigg] dt \\
& &\hspace{3cm}+\si(y_t(\xi)) \, d\big( T(h,g)\big)_t+\si(y_t(\xi)) \, d\mathbf{z}_t \ ,
\end{eqnarray*}
with a similar transformation for the equation involving $\jmath$. Given the regularity of $g$, the latter transformations can actually be justified in a rigourous way, that is in the framework settled through Definition~\ref{def:general-sol}: one can for instance mimick the arguments of the proof of Proposition \ref{prop:consistency} point $(b)$. Identity (\ref{ident-psi-psiti-0}) now follows from the uniqueness of the solution to the system (\ref{syst-tilde-1})-(\ref{syst-tilde-2}) (with fixed $\vp:=\vp_K$). Identity (\ref{ident-psi-psiti}) can then be derived from a similar transformation of (\ref{syst-tilde-1})-(\ref{syst-tilde-2}), which completes the proof of our statement.
\end{proof}


\subsection{Toward a Girsanov transformation}
Let us go back to our stochastic setting, where $x=X$ stands for a $H$-fBm. The aim now is to translate the previous results at the level of the underlying Wiener paths, so as to construct the expected coupling $(W,\widetilde{W})$ on $[\tau_k,\tau_k+1]$ via a Girsanov-type argument. To this end, we will deduce from Theorem \ref{theo:pathwise-rough-coupl} how to build an appropriate drift function $g_{_W}$ for the hitting objective to be achieved. This property is the topic of Theorem~\ref{theo:construc:appli} below, that we write (without loss of generality) with $\tau_k=0$.  Just before we state this result, we need to specify, through the following technical lemma, how the Wiener space can be somehow 'decomposed' in accordance with the past-innovation splitting (\ref{decompo-fbm-ter}). 

\smallskip

We recall that we have fixed $H\in (1/3,1/2)$ and $\ga\in (1/3,H)$ for the whole study. Besides, in the sequel, we will indifferently denote by $\mathbb{P}_W$, and call the Wiener measure, the Wiener measure on $\cac((-\infty,0];\R^d)$ (reversed Brownian motion), the Wiener measure on $\cac([0,1];\R^d)$, as well as the law of a two-sided Brownian motion on $\cac((-\infty,1];\R^d)$. We also define the following two sets of functions, which will be used in order to define our perturbations on Wiener's space:
\begin{eqnarray}
\cb_{-}^{c}&=&\lcl g: (-\infty,0]\to\R^{d} ; \, g \text{ is bounded measurable with compact support}  \rcl
\label{eq:def-Bc-minus}  \\
\cb_{+}&=&\lcl g: [0,1]\to\R^{d} ; \, g \text{ is bounded and measurable}  \rcl \ .
\label{eq:def-B-plus}
\end{eqnarray}

\begin{lemma}\label{lem:justifd}
There exist two spaces $\Omega_-\subset \cac((-\infty,0];\R^d)$ and $\Omega_+\subset \cac([0,1];\R^d)$ of full Wiener measure such that the following properties are satisfied:
\begin{itemize}
\item[(i)] Let ${\cal D}_{X}^-$ be defined for every smooth compactly-supported $\vp:(-\infty,0] \to \R^d$ vanishing at $0$ by
\begin{equation}\label{eq:def-DX-minus}
{\cal D}_X^- \vp(t):=\alpha_H\int_{-\infty}^0 \left((t-r)^{H-\frac 12}-(-r)^{H-\frac 12}\right) d\vp(r)\quad \textnormal{if $t\in (0,1]$},
\end{equation}
and ${\cal D}_X^- \vp(0)=0$.  Then ${\cal D}_X^-$ extends to $\Omega_-$ as an application with values in ${\cal E}^2_\ga$. Besides, for every $w_-\in \Omega_-$ and every function $\gw^-\in \cb_{-}^{c}$, the path $\wti_-:=w_-+\int_{-\infty}^. \gw^-(s)\, ds$ still belongs to $\Omega_-$.  
\item[(ii)] Let ${\cal D}_{X}^+$ be defined for every smooth compactly-supported $\vp:[0,1] \to \R^d$ vanishing at $0$ by
\begin{equation}\label{eq:def-DX-plus}
{\cal D}_X^+ \vp(t)=\alpha_H\int_{0}^t (t-r)^{H-\frac 12} d \vp(r) \ , \ t\in [0,1] \ .
\end{equation}
Then ${\cal D}_X^+$ extends to $\Omega_+$ as an application with values in $\cac^\ga([0,1];\R^d)$, and for every $w_+ \in\Omega_+$, ${\cal D}_X^+ w_+$ can be canonically lifted into a rough path $\mathfrak{L}({\cal D}_X^+ w_+)$, in the sense of Definition \ref{defi-canonic-rp}. Besides, for every $w_+\in \Omega_+$ and every continuous $\gw^+:[0,1]\to \R^d$, the path $\wti_+:=w_++\int_0^. \gw^+(s)\, ds$ still belongs to $\Omega_+$.
\item[(iii)] For every $\varepsilon>0$, it holds that
$$\PE_W(w_+\in \Omega_+: \,  \|\mathfrak{L}({\cal D}_X^+ w_+)\|_{\gamma;[0,1]}\le \varepsilon)>0 \ .$$
\item[(iv)] Set $\Omega:=\{w_-\sqcup w_+:\,  w_-\in \Omega_- , w_+\in \Omega_+\}\subset \cac((-\infty,1];\R^d)$ and for every $w=w_-\sqcup w_+\in \Omega$,
$${\cal D}_X w:={\cal D}_X^- w_-+{\cal D}_X^+  w_+ \ .$$
Then ${\cal D}_X w$ belongs to $\cac^\ga([0,1];\R^d)$ and can be canonically lifted as a rough path, in the sense of Definition \ref{defi-canonic-rp}. Furthermore, as a random variable on $(\Omega,\mathbb{P}_W)$, ${\cal D}_X$ has the law of a fBm of Hurst index $H$. 
\item[(v)] It holds that $\Theta_{-1}(\Omega)\subset \Omega_-$, where $\Theta$ stands for the shift operator, that is $\Theta_\tau(w)_t=w_{t+\tau}$.
\end{itemize}
\end{lemma}

\begin{proof}
Let us fix $\varepsilon \in (0,H-\ga)$.

\smallskip

\noindent
$(i)$ Note first that, using an elementary integration-by-parts formula, ${\cal D}_X^-$ can be equivalently defined as
\begin{equation}\label{ipp-d-x-minus}
{\cal D}_X^- \vp(t):=\alpha_H\left(H-\frac 12\right)\int_{-\infty}^0 \left((t-r)^{H-\frac 32}-(-r)^{H-\frac 32}\right) \vp(r)\, dr\quad \textnormal{if $t\in (0,1]$}
\end{equation}
and ${\cal D}_X^- \vp(0)=0$, for every test-function $\vp$ vanishing at $0$. Fixing a parameter $\varepsilon'>0$, we can now specify the space $\oom_{-}$ alluded to in our claim above:
\begin{multline}\label{eq:def-omega-minus}
\Omega_-:={\cal  C}_{\varepsilon,\varepsilon'}
=
\Big\{  w_-:(-\infty,0]\to \R^d; \, w_{-} \text{ is $(\frac{1}{2}-\varepsilon)$-Hölder continuous on compacts intervals, } \\
w_{-}(0)=0, \text{ and }   \lim_{t\rightarrow-\infty} \frac{w_-(t)}{|t|^{\frac{1}{2}+\varepsilon'}}=0 \Big\}.
\end{multline}
Owing to some classical properties on the Wiener process, this subspace is of Wiener measure $1$ for any fixed $\ep,\varepsilon'>0$, and from (\ref{ipp-d-x-minus}), it is easy to check that ${\cal D}_X^-$ continuously extends to $\Omega_-$, as an application with values in $\mathcal{E}^2_\ga$.  Indeed, for every test-function $\vp$ vanishing at $0$, every $k\geq 1$ and every $t\in (0,1]$, it holds that, for some constant $c_k\geq 0$,
\begin{eqnarray*}
t^{k-\ga}\big| ({\cal D}_X^- \vp)^{(k)}(t)\big|&=&c_k\, t^{k-\ga}\bigg| \int_{-\infty}^0 (t-r)^{H-\frac 32-k} \vp(r)\, dr\bigg|\\
&\leq & c_k\, \bigg\{ \int_{-\infty}^{-1} (t-r)^{H-\frac 32-k} |\vp(r)|\, dr+t^{k-\ga}\int_{-1}^{0} (t-r)^{H-\frac 32-k} |\vp(r)|\, dr\bigg\}\\
&\lesssim& \bigg( \sup_{r\leq -1} \frac{|\vp(r)|}{|r|^{\frac12+\varepsilon'}}\bigg) \, \int_{-\infty}^{-1} \frac{dr}{|r|^{1+k-H-\varepsilon'}}+ \lVert \vp\rVert_{\ga;[-1,0]} t^{k-\ga} \int_{-1}^0 (t-r)^{H-1-k-\varepsilon} \, dr \\
&\lesssim& \sup_{r\leq -1} \frac{|\vp(r)|}{|r|^{\frac12+\varepsilon'}}+ \lVert \vp\rVert_{\ga;[-1,0]}  \ .
\end{eqnarray*}

\

\noindent
$(ii)$ Just as above, note that ${\cal D}_X^+$ can be equivalently defined as
\begin{equation}\label{ipp-d-x-plus}
{\cal D}_X^+ \vp(t):=\alpha_H \vp(t)\, t^{H-\frac12}+\alpha_H\left(H-\frac 12\right)\int_0^t (t-r)^{H-\frac 32}\big(\vp(t)- \vp(r)\big)\, dr \quad \textnormal{if $t\in (0,1]$}
\end{equation}
and ${\cal D}_X^- \vp(0)=0$, for every test-function $\vp$ on $[0,1]$ vanishing at $0$. From this expression, it is easy to check that, as a map with values in $\cac^\ga([0,1])$, ${\cal D}_X^+$ continuously extends to the (full-Wiener-measure) space $\widetilde{\Omega}_+$ of $(\frac12-\ep)$-Hölder paths on $[0,1]$ which vanish at $0$. We can also check that the covariance function of the Gaussian process ${\cal D}_X^+:(\widetilde{\Omega}_+,\mathbb{P}_W)\to \cac^\ga([0,1])$ so defined satisfies the conditions of \cite[Theorem 15.33]{FV-bk}, which allows us to assert that the subspace
\begin{equation}\label{eq:def-omega-plus}
\Omega_+:=\big\{w_+\in \widetilde{\Omega}_+: \, {\cal D}_X^+w_+ \ \text{can be canonically lifted as a $\ga$-rough path}\big\}.
\end{equation}
is of full Wiener measure. Finally, the stability of $\Omega_+$ through the transformation $w_+\mapsto w_++\int_0^.\gw^+(s) \, ds$ follows from the definition of $\widetilde{\Omega}_+$ and the result of Proposition \ref{prop:lift-sum}, since $\int_0^.\gw^+(s) \, ds$ obviously belongs to $\cac^1([0,1])$.

\

\noindent
$(iii)$ By \cite[Theorem 15.60]{FV-bk} (and using the terminology therein introduced), the assertion reduces to showing that the Gaussian process ${\cal D}_X^+:(\Omega_+,\mathbb{P}_W) \to \cac^\ga([0,1])$ satisfies the complementary Young regularity condition. It turns out that this specific result has been proved in \cite[Example 2.11]{CYR-friz}, which immediately yields the conclusion.

\

\noindent
$(iv)$ It holds that $\mathcal{E}^1_\ga\subset \cac^\ga([0,1])$, and so ${\cal D}_Xw$ does belong to $\cac^\ga([0,1])$, for every $w\in \Omega$. The fact that it can be canonically lifted as a $\ga$-rough path follows from point $(ii)$ (that is, we can lift ${\cal D}_X^+w_+$) and Proposition~\ref{prop:lift-sum} (due to ${\cal D}_X^- w_- \in \mathcal{E}^2_\ga$). Finally, when dealing with a two-sided Brownian motion $W$ on $(-\infty,1]$, and starting from the explicit formulas (\ref{ipp-d-x-minus})-(\ref{ipp-d-x-plus}) for ${\cal D}_X^-$ and ${\cal D}_X^+$, we can apply Itô formula to identify ${\cal D}_X W$ with the Mandelbrot-Van Ness transformation of $W$.

\

\noindent
$(v)$ It can be immediately checked from the previous constructions.
\end{proof}

We will also rely on the following inversion formula, borrowed from \cite[p. 741]{hairer}:
\begin{lemma}\label{lem:inversi}
Consider $w_-,\wti_- \in \Omega_-$, $w_+ \in \Omega_+$ and $\gw^{-}\in\cb_{-}^{c}$ (recall that those spaces are respectively defined by \eqref{eq:def-omega-minus}, \eqref{eq:def-omega-plus} and \eqref{eq:def-Bc-minus}). We assume that:
\begin{equation}\label{eq:relation-w-wtilde-minus}
\wti_-=w_-+\int_{-\infty}^{.} \gw^-(s) \, ds,
\end{equation}
Also consider a generic function $g_{X}^{+}\in\cb_{+}$.  Then there exists a map $\ck:\cb_{-}^{c}\times\oom_{+}\times\cb_{+}\to\oom_{+}$ such that
\begin{equation*}
{\cal D}_X(w_- \sqcup w_+)_{|[0,1]}+\int_0^. g_X^+(s) \, ds
=
{\cal D}_X\lp \wti_- \sqcup \ck(g_{_W}^-,\om_{+},g_{X}^{+})\rp_{|[0,1]}.
\end{equation*}
Specifically, the map $\ck$ is given by the following formula:
\begin{equation*}
\ck(g_{_W}^-,\om_{+},g_{X}^{+}) 
=
w_{+} + \int_{0}^{.}  \ch(g_{W}^{-},g_{X}^{+}) \, ds
=
T\lp w_{+}, \ch(g_{W}^{-},g_{X}^{+})\rp,
\end{equation*}
where we recall that $T$ has been introduced in \eqref{eq:def-T}, and where the integral transformation $\ch$ is defined by:
\begin{align}\label{eq:def-cal-H}
{\cal H}(g_1,g_2)_t:=C_1\int_{-\infty}^0 \frac{t^{\frac{1}{2}-H}(-s)^{H-\frac{1}{2}}}{t-s} g_1(s)\,  ds+\alpha_H\frac{d}{dt}\left(\int_0^t (t-s)^{\frac{1}{2}-H} g_2(s)\, ds\right) \notag\\
=C_1{\cal R}_0 g_1(t)+C_2\int_0^t (t-s)^{-\frac 12-H} g_2(s) \, ds \ .
\end{align}
In \eqref{eq:def-cal-H}, observe that the notation ${\cal R}_{T}$ has been introduced in \eqref{eq:def-R-T}.
\end{lemma}
The above formula can be interpreted as follows: in the spirit of \eqref{eq:defgwgb}, the second term of ${\cal H}(g_{_W}^-,g_X^+)$ corresponds to the the drift on the Wiener component induced by the ``fractional drift'' $g_X^+$ whereas
by Lemma \ref{lemma:formdrift}, the first term is the drift on the Wiener component on $[0,1]$ which ensures that, given a past $\gw^-$, the corresponding fractional drift is equal to $0$.

\smallskip

\noindent In the next statement, we denote by $\Phi(a;\mathbf{x})$ the unique solution on $[0,1]$ of the rough equation
$$dy_t=b(y_t) \, dt+\si(y_t) \, d\mathbf{x}_t \quad , \quad y_0=a \ ,$$
understood in the sense of Definition \ref{def:sol-davie}. Let us recall that $b$ and $\sigma$ are assumed to satisfy Hypothesis (\textbf{H1}), so that the above equation indeed admits a unique solution on $[0,1]$.




\begin{theorem}\label{theo:construc:appli}
Fix two parameters $K,\alpha >0$, as well as a $(K,\alpha,\gamma)$-admissible state $\pi=(a,\widetilde{a},w_-,\widetilde{w}_-)\in (\R^d)^2 \times \Omega_-^2$. Also, consider $\gw^-\in \cb_{-}^{c}$, where $\cb_{-}^{c}$ is defined by \eqref{eq:def-omega-minus}, such that $\wti_-=w_-+\int_{-\infty}^. \gw^-(s)\, ds$. 
Then there exists an application $\Lambda=\Lambda_\pi:\Omega_+ \to \Omega_+$ such that: 
\begin{enumerate}
\item[(i)] For every $w_+\in \Omega_+$ and every $t\in[0,1]$,  $\Lambda(w_+)_t=w_+(t)+\int_0^t \gw^+(w_+,s) ds$, for some $({\cal G}_t)_{t\in[0,1]}$-adapted function $\gw^+$.
\item[(ii)] There exists a positive constant $\delta_K$ depending only on $K$ such that 
$$\PE_{W}\big(w_+\in \Omega_+: \, \Phi(a;\mathfrak{L}({\cal D}_X(w_- \sqcup w_+)))_1=\Phi(\tilde{a};\mathfrak{L}({\cal D}_X(\wti_- \sqcup \Lambda(w_+)))_1\big)\ge \delta_K \ .$$
\item[(iii)]  $\Lambda$ is bijective with inverse $\Lambda^{-1}$ satisfying $\Lambda^{-1}(w_+)_t=w_{+}(t)+\int_0^t \bar{g}_{_{W}}^+(w_+,s) ds$, for some $({\cal G}_t)_{t\in[0,1]}$-adapted function $\bar{g}_{_W}^+$.
\item[(iv)] \label{statement4} There exists  $C_K>0$ depending only on $K$ such that for $\mathbb{P}_W$-every $w_+ \in \Omega_+$, 
$$\int_0^1 (| \gw^+|^2 +|\bar{g}_{_W}^+|^2)(w_+,s)\, ds\le C_K \ .$$
\item[(v)]  \label{statement5}For every fixed $w_+\in \Omega_+$, consider the function $\gw(w_+,.):(-\infty,1]\to \R^d$ defined as $\gw(w_+,t)=\gw^-(t)$ if $t\leq 0$ and $\gw(w_+,t)=\gw^+(w_+,t)$ if $t\in (0,1]$, and denote by $g_X(w_+,.)$ the image of $\gw(w_+,.)$ through the transformation (\ref{eq:defgb}).  Then there exists $C_K>0$ depending only on $K$ such that for $\mathbb{P}_W$-every $w_+\in \Omega_+$,
$$\sup_{t\in[0,1]} |g_X(w_+,t)|\le C_K \ .$$
\end{enumerate}
\end{theorem}

\begin{proof} Set $h_{w_-}:={\cal D}_X^- w_-$ (where we recall that  ${\cal D}_X^-$ is defined by \eqref{eq:def-DX-minus}) and write, for every $w_+\in \Omega_+$, ${\bf z}_{w_{+}}:=\mathfrak{L}({\cal D}_X^+ w_+)$ (see Lemma \ref{lem:justifd}$(ii)$). 

\smallskip

\noindent
$(i)$ With the notations of Theorem \ref{theo:pathwise-rough-coupl}, consider the function $g_X:\Omega_+\times [0,1]\to \R^d$ given by
\begin{equation}\label{defi-g-x-w-plus}
[g_X(w_+)](t):=-\si(y_t(1))^{-1}\int_0^1 d\eta \, \vp_K(j_t(\eta)), 
\quad \textnormal{with} \quad 
(y,j):=\Psi_{V_1}((a,\widetilde{a}),\vp_K,(h_{w_-},{\bf z}_{w_+})) \ .
\end{equation}
Then, with the notations of Lemma \ref{lem:inversi}, we define $\Lambda$  on $[0,1]$ as
\begin{equation}\label{eq:def-Lambda}
\Lambda(w_+)_t
:=
\ck\lp g_{W}^{-}, w_{+}, g_{X}(w_{+}) \rp
=
w_+(t)+\int_0^t{\cal H}(g_{_W}^-,g_X(w_+) )_s \, ds
=
T\big(w_+,{\cal H}(g_{_W}^-,g_X(w_+))\big)_t .
\end{equation}

\smallskip

\noindent
$(ii)$ By the very definition of $\ck,\mathcal{H}$ and invoking Lemma \ref{lem:inversi}, we have, for every $w_+\in \Omega_+$,
$${\cal D}_X(\wti_- \sqcup \Lambda(w_+))_{|0,1]}={\cal D}_X(w_- \sqcup w_+)_{|[0,1]}+\int_0^. g_X(w_+,s) \, ds \ . 
$$
Besides, it is readily checked that
\begin{equation*}
\Phi\lp \widetilde{a}, \mathfrak{L}\lp{\cal D}_X(\wti_- \sqcup \Lambda(w_+))\rp\rp
=
\Phi\lp \widetilde{a}, \mathfrak{L}\lp{\cal D}_X(w_- \sqcup w_{+}) + \int_{0}^{.} [g_X(w_+)](s) ds\rp\rp
=
y^{(1)},
\end{equation*}
where $y^{(1)}$ is defined by \eqref{eq:y(1)}. In addition, by the admissibility condition \eqref{Kadmiscond}, we know that $\vvvert h_{w_-}\vvvert_{1;\ga}\le K$ and $\max(|a|,|\widetilde{a}|,|\widetilde{a}-a|)\le 2 K$. Therefore, we are exactly in a position to apply Theorem  \ref{theo:pathwise-rough-coupl} and deduce the existence of a positive constant $M_K>0$ such that
$$
\PE_{W}\big(w_+\in \Omega_+: \, \Phi(a;\mathfrak{L}({\cal D}_X(w_- \sqcup w_+)))_1=\Phi(\tilde{a};\mathfrak{L}({\cal D}_X(\wti_- \sqcup \Lambda(w_+)))_1\big)
$$
$$
\ge\ \PE_W(w_+\in \Omega_+: \,  \|{\bf z}_{w_{+}}\|_{\gamma;[0,1]}\le M_K) \ .
$$
The conclusion now comes from Lemma \ref{lem:justifd}, point $(iii)$.

\smallskip

\noindent 
$(iii)$ Set $h_{\wti_-}:={\cal D}_X^-\wti_-$ and with the notations of Theorem \ref{theo:pathwise-rough-coupl}-$(iii)$, define, for every $w_+\in \Omega_+$,
$$\bar{g}_X(w_+,t):=\si(\bar{y}_t(1))^{-1}\int_0^1 d\eta \, \vp_K(\bar{\jmath}_t(\eta)), \quad \textnormal{with} \quad (\bar{y},\bar{\jmath}):=\bar{\Psi}_{V_1}((a,\widetilde{a}),\vp_K,(h_{\widetilde{w}_-},{\bf z}_{w_+}))\ ,
$$
where the flow $\bar{\Psi}_{V_1}$ has been introduced in relation \eqref{defi-flow-tilde}.
Then consider the application $\bar{\Lambda}:\Omega_+\to \Omega_+$ given by
\begin{equation}\label{eq:def-Lambda-bar}
\bar{\Lambda}(w_+)_t
:=\ck\lp -g_{W}^{-}, w_{+}, \bar{g}_{X}(w_{+}) \rp
=T_t(w_+,{\cal H}(-g_{{_W}}^-,\bar{g}_X(w_+,.))) .
\end{equation}
Let us check that $\bar{\Lambda}$ is actually the inverse of $\Lambda$, by showing first that $\bar{\Lambda}\circ\Lambda={\rm Id}$. To this end, fix $w_+\in \Omega_+$ and set $\bar{w}_+:= \Lambda(w_+)= T(w_+, g_{_W}^+(w_+,.))$, where, according to $(i)$, $g_{_W}^+(w_+,.):={\cal H}(g_{_W}^-,g_X(w_+,.) )$, with $g_X(w_+,.)$ given by (\ref{defi-g-x-w-plus}). Then consider the functions $G_{_W}^-=\int_{-\infty}^.g_{_W}^-(s)\, ds$ and $G_{_W}^+(w_+,.):=\int_0^.g_{_W}^+(w_+,s)\, ds$ (defined respectively on  $(-\infty,0]$ and $[0,1]$), so that, by construction, 
$$z_{\bar{w}_+}=
\cd_{X}^{+}\bar{w}_{+}
=
z_{w_+}+{\cal D}_X^+( G_{_W}^+(w_+,.)) \ .$$
\noindent Thus, using Corollary \ref{coro:identif-sol}, we get
\begin{align}
\bar{\Psi}_{V_1}((a,\widetilde{a}),\vp_K,(h_{\widetilde{w}_-},{\bf z}_{\bar{w}_+}))\nonumber&=\bar{\Psi}_{V_1}((a,\widetilde{a}),\vp_K,\big(h_{\widetilde{w}_-}+{\cal D}_X^+( G_{_W}^+(w_+,.)),{\bf z}_{{w}_+}))\nonumber\\
&=\bar{\Psi}_{V_1}((a,\widetilde{a}),\vp_K,(h_{{w}_-}+ {\cal D}_X^-G_{_W}^-+{\cal D}_X^+( G_{_W}^+(w_+,.)),{\bf z}_{{w}_+})\ .\label{identif-tow}
\end{align}
At this point, observe that by the inversion formula \eqref{eq:defgb}, one has for any $t\in(0,1]$,
\begin{eqnarray*}
\lefteqn{\frac{d}{dt}\left({\cal D}_X^-G_{_W}^-+{\cal D}_X^+( G_{_W}^+(w_+,.))\right)(t)}\\
&=&\al_H \frac{d}{dt}\left(\int_{-\infty}^{0}\left((t-s)^{H-\frac{1}{2}}-(-s)^{H-\frac{1}{2}} \right)g_{_W}^-(s)\, ds +\int_0^t (t-s)^{H-\frac{1}{2}} {\cal H}(g_{_W}^-,g_X(w_+,.) )_s\,  ds\right)\\
&=&g_X(w_+,t)\ .
\end{eqnarray*}
Therefore,  $h_{{w}_-}+{\cal D}_1 G_{W}^-+{\cal D}_2 G_{_W}^+=T(h_{{w}_-}, g_X(w_+,.))$, which, going back to (\ref{identif-tow}), gives us
$$\bar{\Psi}_{V_1}((a,\widetilde{a}),\vp_K,(h_{\widetilde{w}_-},{\bf z}_{\bar{w}_+}))=\bar{\Psi}_{V_1}((a,\widetilde{a}),\vp_K,T(h_{{w}_-}, g_X(w_+,.)),{\bf z}_{w_+})) \ .$$
We can now apply identity \eqref{ident-psi-psiti-0} to assert that 
$$\bar{\Psi}_{V_1}((a,\widetilde{a}),\vp_K,(h_{\widetilde{w}_-},{\bf z}_{\bar{w}_+}))={\Psi}_{V_1}((a,\widetilde{a}),\vp_K,(h_{{w}_-},{\bf z}_{{w}_+}))\ ,$$
which readily entails that $\bar{g}_X(\bar{w}_+,.)=-g_X(w_+,.)$. The conclusion is now immediate: according to \eqref{eq:def-Lambda-bar} we have
\begin{align*}
\bar{\Lambda}(\bar{w}_+)&=T(\bar{w}_+,{\cal H}(-g_{{_W}}^-,\bar{g}_X(\bar{w}_+,.)))\\
&=T(w_+,{\cal H}(g_{_W}^-,g_X(w_+,.))+{\cal H}(-g_{{_W}}^-,\bar{g}_X(\bar{w}_+,.)))=T(w_+,0)=w_+ \ .
\end{align*}
The fact that $ \Lambda\circ \bar{\Lambda}={\rm Id}$ follows from symmetric arguments (by using (\ref{ident-psi-psiti}) instead of (\ref{ident-psi-psiti-0})).

\smallskip

\noindent
$(iv)$ Let us recall that according to \eqref{eq:def-cal-H} we have
$$g_{W}^+(w_+,t)={\cal H}(g_{_W}^-,g_X(w_+,.))_t=C_1({\cal R}_0 g_W^-)(t)+C_2\int_0^t (t-s)^{-\frac 12-H} [g_X(w_+)](s) \, ds \ .$$
The desired uniform bound on the $L^2$-norm of $g_W^+(w_+,.)$ then follows from two elementary facts:
(a)
It holds that $\int_0^1 |({\cal R}_0 g_W^-)(s)|^2 \, ds \leq 1$ thanks to the admissibility condition \eqref{hairerassumpcond}. 
(b)
The function $g_X(w_+)$ defined by \eqref{defi-g-x-w-plus} is bounded. This trivially stands from our assumption~\eqref{eq:hyp-bnd-inverse-sigma} on $\si^{-1}$ and from the definition of the cutoff function $\vp_{K}$. 
The same arguments can be used for the bound on the $L^2$-norm of $\bar{g}_W^+(w_+,.)$.

\smallskip

\noindent
$(v)$ Just as above, it is an immediate consequence of Hypothesis (\textbf{H3}) and the definition of $\vp_K$.

\end{proof}

\begin{remark}
We have written our results for a time origin $\tau=0$ for notational sake. However, notice that the generalizations of Lemma \ref{lem:inversi} and Theorem \ref{theo:construc:appli} to a shifted time origin $\tau$ are straightforward. Let us just describe the transformation $\ch$ in this context: consider $w_-,\wti_- \in \Theta_{\tau}\Omega_-$, $w_+ \in \Theta_{\tau}\Omega_+$ and $\gw^{-}\in\Theta_{\tau}\cb_{-}^{c}$, where $\Theta_{\tau}$ denotes the shift of a path by $\tau$. We still assume that relation \eqref{eq:relation-w-wtilde-minus} between $\wti_-$ and $w_{-}$ holds true. Then for $g_{X}^{+}\in\Theta_{\tau}\cb_{+}$ we have
\begin{equation*}
{\cal D}_X(w_- \sqcup w_+)_{|[\tau,\tau+1]}+\int_{\tau}^. g_X^+(s) \, ds
=
{\cal D}_X\lp \wti_- \sqcup \ck_{\tau}(g_{_W}^-,\om_{+},g_{X}^{+})\rp_{|[\tau,\tau+1]},
\end{equation*}
where $\ck_{\tau}$ is defined as follows:
\begin{equation*}
\ck_{\tau}(g_{_W}^-,\om_{+},g_{X}^{+}) 
:=
w_{+} + \int_{\tau}^{.}  \ch_{\tau}(g_{W}^{-},g_{X}^{+})_{s} \, ds\ ,
\end{equation*}
with
\begin{equation*}
{\cal H}_{\tau}(g_1,g_2)_t
:=C_1{\cal R}_0 g_1^\tau(t)+C_2\int_\tau^t (t-s)^{-\frac 12-H} g_2(s) \, ds \ ,
\end{equation*}
and where we recall that the notation ${\cal R}_{T}$ has been introduced in \eqref{eq:def-R-T}. These transforms are then used in the successive binding trials alluded to in Section \ref{subsec:general3stepschem}.
\end{remark}

\subsection{Achievement of Step 1}
As a conclusion of this section, we obtain the following result.

{\begin{prop}\label{prop:step1}Assume $\mathbf{(H1)}$ and $\mathbf{(H3)}$.  Then, for all $\alpha >0$ and $K>0$, there exist constants {$\delta=:\mathbf{\delta}(\alpha,K)>0$} and $C_K >0$ such that for each $k\geq 1$, $(W,\widetilde{W})$ can be built on $[\tau_{k-1},\tau_{k-1}+1]$ in such a way that the following properties hold:
\begin{itemize}
\item[(a)] One has $  \PE(Y_{\tau_{k-1}+1}=\widetilde{Y}_{\tau_{k-1}+1}| \ce_{k-1}\cap \AD_k(K,\alpha,\gamma))\ge \delta$.
\item[(b)] If $\omega\in \AD_k(K,\alpha,\gamma)$ and Step $1$ is successful, then $\sup_{t\in[0,1]} |g_X(t+\tau_{k-1})|\le C_K.$
\item[(c)]  $\int_{\tau_{k-1}}^{\tau_{k-1}+1}|\gw(s)|^2ds\le C_K\quad a.s.$
\end{itemize}
\end{prop}}

\begin{proof} 
At the price of a change of variable, we can assume that $\tau_{k-1}=0$. The construction of the coupling follows the lines of \cite{hairer} and \cite{fontbona-panloup}. For the sake of completeness, one however recalls the principle below.

\noindent
\textit{Step 1: Definition of $\cl(W,\widetilde{W})$:}
With the notations of Theorem \ref{theo:construc:appli}, let $\pi$ denote the current state and let $\Lambda$ denote the related coupling function.  Let $\PE_W$ denote the Wiener measure on $[0,1]$ and  $\Lambda^*\PE_W$ be the image measure of $\PE_W$ by the mapping $\Lambda$. By  Girsanov's Theorem,
 ${\Lambda}^*\PE_W(dw)= D_{{\Lambda}} (w) \PE_W(dw)$ where, with the notations of Theorem \ref{theo:construc:appli},
\bqn\label{eq:dphi}
D_{\Lambda}(w)=\exp\left(\int_0^1 \gw^+(w,s)  dw(s)-\frac{1}{2}\int_0^1 |\gw^+(w,s)|^2ds\right).
\eqn

\noindent First, if $\omega\notin \AD_k(K,\alpha,\gamma)$, one does not attempt Step $1$. In other words, in this case, $\gw^{+}=0$ on $[\tau_{k-1},\tau_{k-1}+1]$.\smallskip

\noindent Second, assume that $\omega\in \AD_k(K,\alpha,\gamma)$. For  positive measures $\mu_1$ and $\mu_2$ with densities $D_1$ and $D_2$ with respect to another measure $\mu$, denote by $\mu_1\wedge\mu_2$ the measure defined by $(\mu_1\wedge \mu_2) (dw)=D_1(w)\wedge D_2(w) \mu(dw)$.
According to Theorem  \ref{theo:construc:appli}$(ii)$, the way of gluing  $Y$ and $\widetilde{Y}$ at time $\tau_{k-1}+1$ implies  the construction of  a coupling $(W,\widetilde{W})$ such that $\widetilde{W}=\Lambda(W)$ on $[\tau_{k-1},\tau_{k-1}+1]$ with lower-bounded probability. However, let us recall that in this non-Markovian setting, we also need to control the distance between $W$ and $\widetilde{W}$ on the event where the coalescent coupling fails. This particular feature leads to a construction of the coupling which slightly differs from the classical maximal coupling, where the components are independent conditionally to the failure (see $e.g.$ \cite{lindvall}).
Namely,  with the help of the invertibility of $\Lambda$ proved in Theorem~\ref{theo:construc:appli}, one defines a non-negative measure ${\bf P}_1$ on $\oom_{+}^2$ by
$${\bf {P}}_1=\frac{1}{2}\left(\Lambda_{1}^* \PE_W\wedge \Lambda^*_{2}\PE_W\right)$$
where $\Lambda_{1}$ and $\Lambda_{2}$ are the functions  $a.s.$ defined on ${\cal C}([0,1],\ER^d)$ by
$$\Lambda_1(w)=(w,\Lambda(w))\quad \textnormal{and}\quad \Lambda_2(w)=(\Lambda^{-1}(w),w).$$
Let us recall here that, even though this is not suggested by the notation,  $ {\bf P}_1$  strongly depends on the current state $\pi$ (via $\Lambda$).  
Indeed, if one goes back to the definition \eqref{eq:def-Lambda} of $\Lambda$, it is readily checked that the function $g_{X}(w_{+})$ therein depends also on $w_{-}$.
In addition, for any bounded measurable function $F$ defined on $\oom_{+}^2$, we have:
\begin{eqnarray*}
\E_{\Lambda_{1}^{*}\PE}[F]
&=&
\int_{\oom_{+}} F(w,\laa(w)) \, \PE(dw)
=
\int_{\oom_{+}} F(w_{1},w_{2}) \, \1_{(w_{2}=\laa(w_{1}))} 
D_{{\Lambda}}(w_{2}) \, \PE(dw_{2}) \\
&=&
\int_{\oom_{+}} F(w_{1},w_{2}) \, \1_{(w_{1}=\laa^{-1}(w_{2}))} 
D_{{\Lambda}}(w_{2}) \, \PE(dw_{2}).
\end{eqnarray*}
Therefore we get:
$$\Lambda_{1}^* \PE_W(dw_1,dw_2)
=\mathbf{1}_{\{(\Lambda^{-1}(w),w)\}}(w_1,w_2)D_{{\Lambda}}(w_{2}) \PE_W (dw_{2}),$$
where $D_{{\Lambda}}$ is defined by \eqref{eq:dphi}. This implies that ${\bf P}_1$ satisfies
\bqn\label{eq:repp1}
{\bf P}_1(dw_1,dw_2)=\frac{1}{2}\mathbf{1}_{\{(\Lambda^{-1}(w),w)\}}(w_1,w_2)
(D_{\Lambda}(w_{2})\wedge 1) \PE_W (dw_{2}).
\eqn
Write  $S(w_1,w_2)=(w_2,w_1)$ and denote by $\widetilde{\bf P}_1$ the ``symmetrized'' non-negative measure induced by ${\bf P}_1$,  $\widetilde{\bf P}_1:={\bf P}_1+S^*{\bf P}_1$.  We then define the coupling   $({W}_t^{\tau_{k-1}},\widetilde{W}_t^{\tau_{k-1}})=(W_{t+\tau_{k-1}}-W_{\tau_{k-1}},\widetilde{W}_{t+\tau_{k-1}}-\widetilde{W}_{\tau_{k-1}})$ as follows:
\begin{equation}\label{eq:def-law-W-tilde-W}
{\cal L}(({W}_t^{\tau_{k-1}},\widetilde{W}_t^{\tau_{k-1}})_{t\in[0,1]})
=\widetilde{\bf P}_1+ \Delta^*(\PE_W-\Pi^*_1 \widetilde{\bf P}_1)
={\bf P}_1+{\bf P}_2,
\end{equation}
with $\Delta(w)=(w,w)$, $\Pi_1(w_1,w_2)=w_1$ and ${\bf P}_2=S^*{\bf P}_1+\Delta^*(\PE_W-\Pi^*_1 \widetilde{\bf P}_1)$. 
Using \eqref{eq:repp1}, we check that for nonnegative functions $f$, 
\begin{align*}
\Pi^*_1\widetilde{\bf P}_1(f)&\le\frac{1}{2}\int \left(f(\Lambda^{-1}(w))D_{\Lambda}(w)+ f(w)\right)\PE_{W}(dw)\leq \PE_{W}(f),
\end{align*}
hence  ${\bf P}_2$ is the sum of two non-negative measures.
Thanks to the symmetry property of $\widetilde{\bf P}_1$ and to the fact that $ \Pi_1 \circ\Delta$ is the identity, one can also check that the marginals of $ {\bf P}_1+ {\bf P}_2$ are both equal to $\PE_W$. 
In conclusion, the coupling \eqref{eq:def-law-W-tilde-W} between $W$ and $\widetilde{W}$ has been achieved in such a way that:\smallskip

\noindent\emph{(i)} Each marginal is the distribution of a Brownian motion.\smallskip

\noindent\emph{(ii)} Only three possibilities occur: $w_2=\Lambda(w_1)$ (under ${\bf P}_1$) or $w_2=\Lambda^{-1}(w_1)$
(under $S^*{\bf P}_1$) or $w_1=w_2$ (under $\Delta^*(\PE_W-\Pi^*_1 \widetilde{\bf P}_1)$) 
so that $\gw=0$ in that last case. In particular, whenever the coupling fails, the distance between the two Brownian motions is still controlled.

\noindent
\textit{Step 2: Proof of statements $(b)$ and $(c)$.} Statement $(b)$ is a direct consequence of the last statement of Theorem \ref{theo:construc:appli}. For $(c)$, the result is obvious  if $\omega\notin \AD_k(K,\alpha,\gamma)$. Otherwise, this is a consequence of the statement $(ii)$ above and from Theorem \ref{theo:construc:appli}$(iv)$.

\noindent
\textit{Step 3: Proof of statement $(a)$}: 
Recall that  $\tilde{P}_1$ denotes the distribution of  $({W}^{\tau_{k-1}},\widetilde{W}^{\tau_{k-1}})$ on $[\tau_{k-1},\tau_{k-1}+1]$ and that $\widetilde{W}_t^{\tau_{k-1}}=\Lambda({W}_t^{\tau_{k-1}})$ under the subprobability ${\bf P}_1$. Set  ${\cal R}:=\{(w, (y^{\pi,w}(1)=\tilde{y}^{\pi,\Lambda(w)}(1)\}$ where $(y^{\pi,w},y^{\pi,\Lambda(w)})$ stands for a coupled solution to the SDE on $[0,1]$ with initial condition $\pi$ and couple of Wiener innovations $(w,\Lambda({w}))$.
With the notations of Theorem \ref{theo:construc:appli}, ${\cal R}$ can be more precisely written as follows:
$${\cal R}:=\big\{ w_+\in \Omega_+: \, \Phi(a;\mathfrak{L}({\cal D}_X(w_- \sqcup w_+)))_1=\Phi(\tilde{a};\mathfrak{L}({\cal D}_X(\wti_- \sqcup \Lambda(w_+)))_1\big\} \ ,$$
where $\pi:=(a,\tilde{a},w_-,\wti_-)\in (\R^d)^2 \times \Omega_-^2$ stands for the past of the system up to time $\tau_{k-1}=0$.
As a consequence,
$$\PE(Y_{\tau_{k-1}+1}=\widetilde{Y}_{\tau_{k-1}+1}|\ce_{k-1}\cap \AD_k(K,\alpha,\gamma))\ge  {\bf P}_1({{\cal R}\times\Lambda({\cal R})})=\|\mathbf{1}_{{\cal R}\times\Lambda({\cal R})}{\bf P}_1\|_{TV}.$$
Now, by Theorem  \ref{theo:construc:appli}, we know that there exists a positive $\delta_K$ (depending only on $K$) such that $\PE_W({\cal R})\ge\delta_K$. Our claim (a) is thus reduced to lower bound $\|\mathbf{1}_{{\cal R}\times\Lambda({\cal R})}{\bf P}_1\|_{TV}$ in terms of $\PE_W({\cal R})$.

The aforementioned lower bound is obtained as follows: by \eqref{eq:repp1} and Lemma C.1. of \cite{mattingly02} (applied to $p=2$, $\mu_1=\Lambda^*\PE_W$, $\mu_2=\PE_W$ and $X={\cal R}$) we have
$$\|\mathbf{1}_{{\cal R}\times\Lambda({\cal R})}{\bf P}_1\|_{TV}\ge \frac{\left[\int_{\Lambda({\cal R})} D_{\Lambda}(w)\PE_W(dw)\right]^2}{ 4 \int_{\Lambda({\cal R})} D_{\Lambda}(w)^{3} \PE_W(dw)}.$$
Following the lines of the proof of \cite[Lemma 3.1]{fontbona-panloup} and using that $w\mapsto \int_0^1|g_{_W}^+(w,s)|^2ds$ is  bounded by a constant depending only on $K$ (see Theorem \ref{theo:construc:appli}$(iv)$), one deduces that  
$$\|\mathbf{1}_{{\cal R}\times\Lambda({\cal R})}{\bf P}_1\|_{TV}\ge C\left[  \PE_W({\cal R})  \right]^2\ge C\delta_K>0,$$
which is the desired lower bound.
\end{proof}

\section{About Step 2}\label{sec:step2}

{As explained in Section \ref{subsec:general3stepschem}, and following the ideas of \cite{fontbona-panloup,hairer}, Step 2 consists in a series of trials to keep $Y$ and $\widetilde{Y}$ as equal on successive intervals $I_{k,\ell}:=[s_{k,\ell},s_{k,\ell+1}]$ of length $c_2 2^\ell$. To be more specific, for every fixed $k\ge1$, we define $(s_{k,\ell})_{\ell\ge0}$ by
\bqn\label{eq:skuv}
 s_{k,0}=s_{k,1}=\tau_{k-1}+1\quad \textnormal{and for every $\ell\ge 1$} \quad s_{k,\ell+1}=s_{k,\ell}+c_2 2^{\ell}.
 \eqn
From a pathwise point of view, an obvious way to achieve our goal here, that is to keep the paths $Y$ and $\widetilde{Y}$ glued together is to set $g_X(t)=0$ after time $\tau_{k-1}+1$, which by the one-to-one connection of Lemma \ref{lemma:formdrift}, amounts to setting
\begin{equation}\label{def-g-s}
g_{_W}(t)=g_{_S}(t):=({\cal R}_0 g^{\tau_{k-1}+1}_{_W})(t)  \ .
\end{equation}
The aim then is to extend the previous Brownian coupling $(W,\widetilde{W})$ in such a way that, with some controlled probability, Condition (\ref{def-g-s}) is indeed satisfied on the successive intervals $I_{k,\ell}$. Using the notation ${\cal B}_{k,\ell}$ introduced in (\ref{eq:bkl-in-general-proof}) and with the proof of Theorem \ref{theo:principal} in mind (see (\ref{applic-step2-1}) and (\ref{applic-step2-2})), we are more precisely interested in the control of the related quantity $\PE({\cal B}_{k,\ell}|{\cal B}_{k,\ell-1})$, that is the probability of respecting (\ref{def-g-s}) on the interval $I_{k,\ell}=[s_{k,\ell},s_{k,\ell+1}]$ provided it holds up to time $s_{k,\ell}$. This specific issue has been studied in \cite[Section 3.2]{fontbona-panloup} for $H>\frac12$. It turns out that the result therein obtained, as well as its proof, can be transposed into our setting without any change, which leads us directly to the following assertion:}

 {\begin{prop}\label{lemme:step2.2} Let $g_{_S}$ be defined by (\ref{def-g-s}) for $t\geq \tau_{k-1}+1$. Then for every $\alpha \in (0,H)$ and every $K>0$, there exist constants $\mathbf{C}^{\mathbf{2}}_{\alpha,K}\ge 1$, $\mathbf{C}^{\mathbf{2,1}}_{\alpha,K}\ge 1$ and $\rho^1_{\alpha,K},\rho^2_{\alpha,K}\in (0,1)$, which do not depend on $k$ and such that the following properties hold:
\begin{itemize}
\item[(a)] On the event $\AD_k(K,\alpha,\gamma)$, one has
\bqne
\int_0^{+\infty} (1+t)^{2\alpha}|g_{_S}(\tau_{k-1}+1+t)|^2dt\le \mathbf{C}^{\mathbf{2}}_{\alpha,K} \ .\eqne
\item[(b)] One can extend the coupling $(W,\widetilde{W})$ along Condition (\ref{def-g-s}) in such a way that, calibrating Step 2 by the formula $c_2:=(\mathbf{C}^{\mathbf{2}}_{\alpha,K})^{\frac{1}{2\alpha}}$, one has
$$\rho^1_{\alpha,K}\le\PE({\cal B}_{k,1}|{\cal B}_{k,0})\le \rho^2_{\alpha,K} \ ,$$
and for all $k\geq 0$, $\ell\geq 2$,
\bqn\label{eq:lbbk1}
1-2^{-\alpha \ell}\le\PE({\cal B}_{k,\ell}|{\cal B}_{k,\ell-1})\le 1-2^{-\alpha \ell-1} \ .
\eqn
\item[(c)] On the event $F_{k,\ell}$ ($\ell \geq 1$) defined by \eqref{eq:akl} and under the same calibration $c_2:=(\mathbf{C}^{\mathbf{2}}_{\alpha,K})^{\frac{1}{2\alpha}}$, one has
$$\int_{s_{k,1}}^{s_{k,2}} |\gw(t)|^2dt\le \mathbf{C}^{\mathbf{2,1}}_{\alpha,K} \ ,$$
and if $\ell \geq 2$,
\bqne
\int_{s_{k,\ell}}^{s_{k,\ell+1}} |\gw(t)|^2dt\le (2(\ell+3))^2\quad ,\quad\int_{s_{k,p-1}}^{s_{k,p}}|\gw(t)|^2dt\le 2^{-2\alpha p} \ , \ p \in\{2,\ldots,\ell\}\ .
\eqne
\end{itemize}
\end{prop}}

\section{$(K,\alpha,\gamma)$-admissibility condition}\label{sec:Kadmis}
In this section, we assume that Steps 1 and 2 are carried out  as described previously, and the aim is to ensure that  the system is $(K,\alpha,\gamma)$-admissible with positive probability at every time $\tau_k$.  This is the purpose of the next proposition. {In the subsequent statements, we recall that for all $\alpha \in (0,H)$ and $K>0$, the notation $\mathbf{C}^{\mathbf{2}}_{\alpha,K}$ refers to the constant in $(1,+\infty)$ provided by Proposition \ref{lemme:step2.2}.}

\begin{proposition}\label{prop:minokalp} 
Let $H\in(1/3,1/2)$ and  assume  ${\bf (H1)}$, ${\bf (H2)}$, ${\bf (H3)}$ hold true.
For all $\varepsilon \in (0,1)$, $\alpha\in (0,H)$, $\beta>(1-2\alpha)^{-1}$ and $\varsigma >1$, there exist strictly positive constants 
\begin{equation}\label{csts-K-c-3}
K=:\mathbf{K}(\varepsilon, \alpha) \quad , \quad c_3=:\mathbf{c}_{\mathbf{3}}(\varepsilon, \alpha,\beta,\varsigma)
\end{equation}
 such that calibrating Step 2 and Step 3 along the formulas
\begin{equation}\label{best-calib}
c_2:=(\mathbf{C}^{\mathbf{2}}_{\alpha,K})^{\frac{1}{2\alpha}} \quad , \quad \Delta_3(k,\ell):=c_3 \varsigma^k 2^{\beta \ell}
\end{equation}
yields that $c_3\geq 2c_2$ and for every $k\geq 1$
\begin{equation}\label{lower-boun-step-3}
\PE\big( A_{k+1}(K,\alpha,\gamma)\, \big|\, \ce_k\big)\ge 1-\varepsilon \ .
\end{equation}
\end{proposition}

{
\begin{remark}
Let us insist on the fact that, in accordance with our notations in (\ref{csts-K-c-3}), the function $\mathbf{K}$ so defined depends on the two parameters $\varepsilon$ and $\alpha$ only, whereas $\mathbf{c}_{\mathbf{3}}$ depends both on $(\varepsilon,\alpha)$ and on $(\beta,\varsigma)$. This dependence issue is of paramount importance in the proof of Theorem \ref{theo:principal}, as we have seen it in Section \ref{sec:prooftheoprinc}.
\end{remark}}

{The proof of Proposition \ref{prop:minokalp} will actually be obtained as a consequence of the three following lemmas. We assume here that $H\in (1/3,1/2)$ is fixed, and that Hypotheses ${\bf (H1)}$, ${\bf (H2)}$, ${\bf (H3)}$ hold true.}
{
\begin{lemma}\label{lem:step3-0}
For all $\alpha\in (0,H)$, $\beta>(1-2\alpha)^{-1}$, $K>0$ and $\varsigma >1$, there exists a constant $\mathbf{c}_{\mathbf{3,1}}(\alpha,\beta,K,\varsigma)>0$ such that for every $c_3\geq \mathbf{c}_{\mathbf{3,1}}(\alpha,\beta,K,\varsigma)$, calibrating Step 2 and Step 3 along the formulas in (\ref{best-calib}) yields that for every $k\geq 1$
$$
\PE\left( \sup_{T\ge0} \int_0^{+\infty} (1+t)^{2\alpha}(({\cal R}_T |\gw^{\tau_k}|)(t))^2dt\le 1\, \Big|\, \ce_k\right)=1 \ .
$$
\end{lemma}}

{
\begin{lemma}\label{lem:step3-2}
For all $\alpha \in (0,H)$ and $\varsigma >1$, there exist constants $\mathbf{C}^{\mathbf{3,2}}_\alpha >0$ and $\mathbf{c}_{\mathbf{3,2}}(\alpha,\varsigma) >0$  such that for all $K>0$ and $c_3\geq \mathbf{c}_{\mathbf{3,2}}(\alpha,\varsigma)$, calibrating the scheme along the formulas in (\ref{best-calib}) yields that 
$$
\sup_{k\geq 0} \max \Big( \ES[\vvvert D^{{\tau_k}}(W) \vvvert_{1;\ga}  |\ce_k]\ ,\ \ES[\vvvert D^{{\tau_k}}(\widetilde{W}) \vvvert_{1;\ga} |\ce_k]\Big) \le\mathbf{C}^{\mathbf{3,2}}_\alpha\ .
$$
\end{lemma}
}

{\begin{lemma}\label{lem:step3-1}
There exists a constant $p\in (0,1)$ and for all $\alpha \in (0,H)$ and $\varsigma >1$, there exist constants $\mathbf{C}^{\mathbf{3,3}}_\alpha >0$ and $\mathbf{c}_{\mathbf{3,3}}(\alpha,\varsigma) >0$  such that for all $K>0$ and $c_3\geq \mathbf{c}_{\mathbf{3,3}}(\alpha,\varsigma)$, calibrating the scheme along the formulas in (\ref{best-calib}) yields that 
$$
\sup_{k\geq 0} \max \Big( \ES[|Y_{\tau_k}|^{p}|\ce_k] \ , \ \ES[|\widetilde{Y}_{\tau_k}|^{p}|\ce_k]\Big) \le\mathbf{C}^{\mathbf{3,3}}_\alpha\ .
$$
\end{lemma}}

{
\begin{remark}
Just as above, let us stress the fact that, as indicated by our notations, the constants $\mathbf{C}^{\mathbf{3,2}}_\alpha$ and $\mathbf{C}^{\mathbf{3,3}}_\alpha$ in Lemmas \ref{lem:step3-2} and \ref{lem:step3-1} only depend on $\alpha$, and not on $\varsigma$. This will be an essential point in the subsequent proof of Proposition \ref{prop:minokalp}.
\end{remark}}

{Before we turn to the proof of these three lemmas, let us see how their combination can lead to the desired proposition.}
{
\begin{proof}[Proof of Proposition \ref{prop:minokalp}]
Fix $\varepsilon \in (0,1)$, $\alpha\in (0,H)$, $\beta>(1-2\alpha)^{-1}$, $\varsigma >1$, and let $p \in (0,1)$, $\mathbf{C}^{\mathbf{3,2}}_\alpha$, $\mathbf{C}^{\mathbf{3,3}}_\alpha$, $\mathbf{c}_{\mathbf{3,2}}(\alpha,\varsigma)$ and $\mathbf{c}_{\mathbf{3,3}}(\alpha,\varsigma)$ be defined as in Lemmas \ref{lem:step3-2} and \ref{lem:step3-1}. Then for all $K >0$ and $c_3 \geq \max(\mathbf{c}_{\mathbf{3,2}}(\alpha,\varsigma),\mathbf{c}_{\mathbf{3,3}}(\alpha,\varsigma))$, calibrating the scheme as in (\ref{best-calib}) yields that
\begin{eqnarray*}
\lefteqn{\PE\big( |Y_{\tau_k}|+ |\widetilde{Y}_{\tau_k}|+  \vvvert D^{(\tau_k)}(W) \vvvert_{1;\ga}+ \vvvert D^{(\tau_k)}( \widetilde{W}) \vvvert_{1;\ga}\le K\, \big|\, \ce_k\big)}\\
&\geq& 1-\PE\big( |Y_{\tau_k}|^{p}+ |\widetilde{Y}_{\tau_k}|^{p}+  \vvvert D^{(\tau_k)}(W) \vvvert_{1;\ga}^{p}+ \vvvert D^{(\tau_k)}( \widetilde{W}) \vvvert_{1;\ga}^{p}> K^{p_\alpha}\, \big|\, \ce_k\big)\\
&\geq & 1-\frac{1}{K^{p}} \big\{ \ES[|Y_{\tau_k}|^{p}|\ce_k] +\ES[|\widetilde{Y}_{\tau_k}|^{p}|\ce_k]+\ES[\vvvert D^{(\tau_k)}(W) \vvvert_{1;\ga}^{p}|\ce_k]+\ES[\vvvert D^{(\tau_k)}( \widetilde{W}) \vvvert_{1;\ga}^{p}|\ce_k]\big\}\\
&\geq & 1-\frac{2}{K^{p}}\big\{\mathbf{C}^{\mathbf{3,3}}_\alpha+(\mathbf{C}^{\mathbf{3,2}}_\alpha)^{p} \big\}\ .
\end{eqnarray*}
Therefore, setting from now on
$$K=\mathbf{K}(\varepsilon, \alpha) :=\big( 2\varepsilon^{-1}\big\{\mathbf{C}^{\mathbf{3,3}}_\alpha+(\mathbf{C}^{\mathbf{3,2}}_\alpha)^{p} \big\}\big)^{1/p}  \ ,$$
we get that for every $c_3 \geq \max(\mathbf{c}_{\mathbf{3,2}}(\alpha,\varsigma),\mathbf{c}_{\mathbf{3,3}}(\alpha,\varsigma))$ and for the calibration in (\ref{best-calib}),
\begin{equation}\label{proof-step3-1}
\PE\big( |Y_{\tau_k}|+ |\widetilde{Y}_{\tau_k}|+  \vvvert D^{(\tau_k)}(W) \vvvert_{1;\ga}+ \vvvert D^{(\tau_k)}( \widetilde{W}) \vvvert_{1;\ga}\le K\, \big|\, \ce_k\big) \geq 1-\varepsilon \ .
\end{equation}
Then, appealing also to the notations of Lemma \ref{lem:step3-0}, we define
$$\mathbf{c}_{\mathbf{3}}(\varepsilon, \alpha,\beta,\varsigma):=\max\big(\mathbf{c}_{\mathbf{3,1}}(\alpha,\beta,K,\varsigma) ,\mathbf{c}_{\mathbf{3,2}}(\alpha,\varsigma),\mathbf{c}_{\mathbf{3,3}}(\alpha,\varsigma), 2(\mathbf{C}^{\mathbf{2}}_{\alpha,K})^{\frac{1}{2\alpha}}\big) \ .$$
In this way, setting $c_3:=\mathbf{c}_{\mathbf{3}}(\varepsilon, \alpha,\beta,\varsigma)$ and still calibrating the scheme as in (\ref{best-calib}), we deduce, on top of (\ref{proof-step3-1}), that $c_3\geq 2c_2$ and by Lemma \ref{lem:step3-0}
\begin{equation}\label{proof-step3-2}
\PE\left( \sup_{T\ge0} \int_0^{+\infty} (1+t)^{2\alpha}(({\cal R}_T |\gw^{\tau_k}|)(t))^2dt\le 1\, \Big|\, \ce_k\right)=1 \ .
\end{equation}
The bound (\ref{lower-boun-step-3}) immediately follows from (\ref{proof-step3-1}) and (\ref{proof-step3-2}).
\end{proof}}

\

{It remains us to prove Lemma \ref{lem:step3-0}, Lemma \ref{lem:step3-2} and Lemma \ref{lem:step3-1}. It turns out that Lemma \ref{lem:step3-0} can be shown along the very same arguments as in \cite[Proposition 4.6]{fontbona-panloup}, and therefore we will not return to this proof for the sake of conciseness. As for the strategy toward Lemma \ref{lem:step3-2}, resp. Lemma \ref{lem:step3-1}, it is the topic of the subsequent Section \ref{subsec:hpdeux}, resp. Section \ref{subsec:lem:step3-1}.}  

\


\subsection{Proof of Lemma \ref{lem:step3-2}}\label{subsec:hpdeux}
Let us first remark that $[D^{(\tau)}(W)]'$ is well-defined (see Lemma \ref{lem:justifd}) and satisfies:
$$[D^{(\tau)}(W)]'_t=\alpha_H \left(H-\frac 12\right)\int_{-\infty}^\tau (t+\tau-r)^{H-\frac{3}{2}} dW_r.$$
In this section, our computations will hinge on a related incremental process, defined as follows: for $u\leq v \leq \tau$, 
$${\cal D}_{u,v}^\tau(t):=\int_u^v(t+\tau-r)^{H-\frac{3}{2}} dW_r.$$
For $k\ge1$, we thus decompose $D^{(\tau_k)}$ in a series of terms depending on the sequence $(\tau_m)_{m=0}^k$:
\bqn\label{decomp:dtauk}
[D^{(\tau_{k})}(W)]'_t=\alpha_H\left(H-\frac 12\right)\left( {\cal D}_{-\infty,0}^{\tau_k}(t)+\sum_{m=1}^{{k}} {\cal D}_{\tau_{m-1},\tau_m}^{\tau_k}(t)\right).
\eqn
The idea of the sequel is to control each term of the right-hand side separately. We begin by a simple lemma:
\begin{lemma}\label{lemma:IPP1}
For every $0\leq u<v<\tau$ and every $t>0$, one has almost surely
$$|{\cal D}_{u,v}^\tau(t)|\le c_H\Big\{ (t+\tau-u)^{H-\frac 32}\left|W_v-W_u\right|+\int_u^v(t+\tau-r)^{H-\frac 52}  \left|W_v-W_r\right|\, dr\Big\}$$
{and
$$|{\cal D}_{-\infty,0}^\tau(t)|\le c_H\int_{-\infty}^0(t+\tau-r)^{H-\frac 52}  \left|W_r\right|\, dr \ ,$$
for some deterministic constant $c_H>0$.}
\end{lemma}
\begin{proof} {Both bounds follow from an integration-by-parts argument similar to the one we performed in \eqref{ipp-d-x-minus}. The second bound also involves the fact that $\lim_{u\rightarrow-\infty} u^{H-\frac{3}{2}} |W_u|=0$ almost surely.}
\end{proof}

We now state some controls related to the decomposition \eqref{decomp:dtauk}.

{\begin{lemma} \label{lemma:contDuv} 
Let $\alpha>0$and assume that for some (fixed) calibration of the scheme, there exists ${\eta}\in (0,1)$ such that for all $k\geq 1$, $\ell \geq 0$ and $K>0$ ,
\begin{equation}\label{eq:cdt-Ek-Fkl}
\PE({\cal E}_{k}|{\cal E}_{k-1})\ge {\eta}\ ,
\qquad \PE(F_{k,\ell}|{\cal E}_{k-1})\le  2^{-\alpha \ell}\quad\textnormal{ and }\quad\Delta \tau_k\ge a_k \ \ \textnormal{a.s.},
\end{equation}
where $F_{k,\ell}$ and $\ce_{k}$ are respectively defined in \eqref{eq:akl} and \eqref{eq:def-cal-Ek}, and $(a_k)_{k\ge1}$ is a deterministic sequence such that $a_k\ge 1$ for every $k\ge1$.  Then there exists a constant $C^2_{\eta,\alpha}>0$ and for every $p>0$ there exists a constant $C^1_{\eta,\alpha,p}>0$ such that for every $k\ge1$,
\begin{equation}\label{eq:cdtional-Dtau-k-1-gamma}
\ES\Big[\Big(\sup_{t\in (0,1]} t^{1-\gamma}|{\cal D}_{\tau_{m-1},\tau_m}^{\tau_k}(t)|\Big)\Big|\, {\cal E}_k\Big]\le \frac{C^1_{\eta,\alpha,p}}{a_k^{1/2-H}\eta^{(k-m)/p}} \quad , \quad   m\in\{1,\ldots,k-1\}\ ,
\end{equation}
\begin{equation}\label{eq:cdtional-Dtau-k-1-gamma-0}
\ES\Big[\Big(\sup_{t\in (0,1]} t^{1-\gamma}|{\cal D}_{-\infty,0}^{\tau_k}(t)|\Big)\Big|\, {\cal E}_k\Big]\le \frac{C^1_{\eta,\alpha,p}}{a_k^{1/2-H}\eta^{k/p}} \ ,
\end{equation}
and
\begin{equation}\label{eq:cdtional-Dtau-k-1-gamma-1}
\max \Big( \ES\Big[\Big(\sup_{t\in (0,1]} t^{1-\gamma}|{\cal D}_{\tau_{k-1},\tau_k}^{\tau_k}(t)|\Big)\Big|\, {\cal E}_k\Big],\ES\Big[\Big(\sup_{t\in (0,1]} t^{1-\gamma}|{\cal D}_{-\infty,0}^{0}(t)|\Big)\Big]\Big)\le C^2_{\eta,\alpha} \ .
\end{equation}
\end{lemma}}
{\begin{remark}\label{rem:lem888} Observe that the second assumption of \eqref{eq:cdt-Ek-Fkl} holds true by Proposition \ref{lemme:step2.2}$(b)$  as soon as $c_2=:=(\mathbf{C}^{\mathbf{2}}_{\alpha,K})^{\frac{1}{2\alpha}}$. As well, with the help of the upper-bound in \eqref{eq:lbbk1}, we have for instance:
$$\PE({\cal E}_k|{\cal E}_{k-1})\ge 1-\PE({\cal B}_{k,2}|{\cal B}_{k,1})\ge \eta=2^{-2\alpha-1}.$$
Remark that the previous inequality corresponds to an upper bound for the probability of success of the attempt. This (possibly surprising) technical condition will in fact provide us with a way to (roughly) control the effect of conditioning by the event ${\cal E}_{k}$ throughout the computations (see in particular \eqref{eq:controlYtaum}).
\end{remark}

\begin{proof}
{The proof is divided in three steps.}

\smallskip

\noindent
\textbf{Step 1.} The aim of this step is to ``make deterministic'' the duration of Attempt $m$ and to go back to a conditioning by ${\cal E}_{m-1}$.
  First, remark that owing to \eqref{eq:skuv}, $\Delta \tau_m$ is deterministic on the event ${F}_{m,\ell}$. We denote it by $\Delta(m,\ell)\geq 1$.
{Now fix $m\geq 1$ and on the event $\ce_{m}$, consider a generic process $(R_t)_{t> \tau_{m-1}}$.} One can readily check that:
$$\ES[R_{\tau_m}|{\cal E}_m]=\sum_{\ell\ge0}\ES[R_{\tau_{m-1}+\Delta(m,\ell)}|F_{m,\ell}]\PE(F_{m,\ell}|{\cal E}_m).$$
By the Cauchy-Schwarz inequality and the fact that $F_{m,\ell}\subset {\cal E}_{m-1}$, one can check that
\begin{align*}
\ES[R_{\tau_{m-1}+\Delta(m,\ell)}|F_{m,\ell}]&=\frac{\ES[R_{\tau_{m-1}+\Delta(m,\ell)}1_{F_{m,\ell}}|{\cal E}_{m-1}]}{\PE(F_{m,\ell}|{\cal E}_{m-1})}\\
& \le \ES[R_{\tau_{m-1}+\Delta(m,\ell)}^2|{\cal E}_{m-1}]^{\frac{1}{2}}\PE(F_{m,\ell}|{\cal E}_{m-1})^{-\frac 12}\ .
\end{align*}
{As} $F_{m,\ell}\subset {\cal E}_m\subset {\cal E}_{m-1}$, {we can obviously write}
$$
\PE(F_{m,\ell}|{\cal E}_{m})=\frac{\PE({ F}_{m,\ell}|{\cal E}_{m-1})}{\PE({\cal E}_{m}|{\cal E}_{m-1})}\ ,
$$
so that, {thanks to our assumption \eqref{eq:cdt-Ek-Fkl}, the following holds true: } 
$$\ES[R_{\tau_m}|{\cal E}_m]\le \eta^{-\frac{1}{2}}\sum_{\ell\ge0} \ES[R_{\tau_{m-1}+\Delta(m,\ell)}^2|{\cal E}_{m-1}]^{\frac{1}{2}}{\PE(F_{m,\ell}|{\cal E}_{m-1})}^{\frac{1}{2}}\ .$$
Invoking our assumption \eqref{eq:cdt-Ek-Fkl} again, we deduce that
$$\ES[R_{\tau_m}|{\cal E}_m]\le  \eta^{-\frac{1}{2}}\sum_{\ell\ge0} 2^{-\frac{\alpha \ell}{2}}\ES[R_{\tau_{m-1}+\Delta(m,\ell)}^2|{\cal E}_{m-1}]^{\frac{1}{2}}$$
which yields:
\begin{equation}\label{eq:controlYtaum1}
\ES[R_{\tau_m}|{\cal E}_m]\le c_{\eta,\alpha}\cdot \sup_{\ell\ge0} \ES[R_{\tau_{m-1}+\Delta(m,\ell)}^2|{\cal E}_{m-1}]^{\frac{1}{2}}\ .
\end{equation}
Finally, since ${\cal E}_k\subset{\cal E}_m$ for $k\geq m$, a similar Cauchy-Schwarz argument as above implies that for a given random variable $S$ and for every $p>0$,
\bqn\label{eq:controlXEK}
|\ES[S |{\cal E}_k]|\le \frac{\ES[|S|^p|{\cal E}_{m}]^{\frac1p} }{\PE({\cal E}_k|{\cal E}_{m})^{\frac{1}{p}}}\ .
\eqn
Using that $\PE({\cal E}_k|{\cal E}_{m})\ge \eta^{k-m}$, we can conclude this step with the following control:
\begin{equation}\label{eq:controlYtaum}
\ES[R_{\tau_m}|{\cal E}_k]\le c_{\eta,\alpha} {\eta^{(m-k)/p}}\sup_{\ell\ge0} \ES[|R_{\tau_{m-1}+\Delta(m,\ell)}|^{2p}|{\cal E}_{m-1}]^{\frac{1}{2p}}\ .
\end{equation}
where $c_{\eta,\alpha}$ depends on $\eta$ and $\alpha$ only.

\smallskip

\noindent 
{\textbf{Step 2. Case $1\leq m<k$.} Since $\Delta \tau_k\ge a_k$, it is readily checked that for all $m<k$, $t\in[0,1]$ and $r\in[\tau_{m-1},\tau_m]$, one has $t+\tau_k-r\ge a_k$. We then deduce from Lemma \ref{lemma:IPP1}  that 
\bqn\label{eq:12334}
\sup_{t\in[0,1]}t^{1-\gamma} | {\cal D}_{\tau_{m-1},\tau_m}^{\tau_k}(t)| \leq \sup_{t\in[0,1]} | {\cal D}_{\tau_{m-1},\tau_m}^{\tau_k}(t)| 
\le 
c_H a_k^{H-1/2} R_{\tau_m}
\eqn
where we have set, for every $t> \tau_{m-1}$,
$$R_{t}:=\left(t-\tau_{m-1}\right)^{-1}\left|W_{t}-W_{\tau_{m-1}}\right|+\int_{\tau_{m-1}}^{t}(t+1-r)^{-2}  \left|W_{t}-W_r\right|dr\ .$$
Since $\int_{\tau_{m-1}}^{t}(t+1-r)^{- 3/2}dr\le 2$ for every $t>\tau_{m-1}$, one can first check by Jensen's inequality that 
$$\left(\int_{\tau_{m-1}}^{t}(t+1-r)^{-2}  \left|W_{t}-W_r\right|dr\right)^p\le c_{p}\int_{\tau_{m-1}}^{t}(t+1-r)^{- 3/2} \left((t-r)^{-\frac 12}\left|W_{t}-W_r\right|\right)^p dr\ .$$
Then it follows from the scaling property of the Brownian motion that 
$$\ES[|R_{\tau_{m-1}+\Delta(m,\ell)}|^{2p}|{\cal E}_{m-1}]\le c_p \ ,$$
where $c_p$ depends on $p$ only. Going back to \eqref{eq:controlYtaum}, we get the desired bound (\ref{eq:cdtional-Dtau-k-1-gamma}).}

\smallskip

\noindent 
{\textbf{Step 3. Case $m=k \geq 1$.} Let us write here
\begin{equation}\label{decompo-case-k=m}
\sup_{t\in (0,1]} t^{1-\gamma}|{\cal D}_{\tau_{k-1},\tau_k}^{\tau_k}(t)| \leq \sup_{t\in (0,1]} |{\cal D}_{\tau_{k-1},\tau_k-1}^{\tau_k}(t)|+ \sup_{t\in (0,1]} t^{1-\gamma}|{\cal D}_{\tau_{k}-1,\tau_k}^{\tau_k}(t)| \ .
\end{equation} 
The first term in the right-hand side can then be treated along the very same arguments as above (using $p=1/2$ in (\ref{eq:controlYtaum})), which gives us directly
\begin{equation}\label{case-k=m-1}
\ES\Big[\Big(\sup_{t\in (0,1]}|{\cal D}_{\tau_{k-1},\tau_k-1}^{\tau_k}(t)|\Big)\Big|\, {\cal E}_k\Big]\le c^1_{\eta,\alpha} \ .
\end{equation}
On the other hand, using the bound of Lemma \ref{lemma:IPP1} again, we get that
$$\sup_{t\in (0,1]} t^{1-\ga} |{\cal D}_{\tau_k-1,\tau_k}^{\tau_k}(t)| \leq c_H R_{\tau_k} \ ,$$
with, for every $t>\tau_{k-1}$,
$$R_t:=|W_t-W_{t-1}|+\int_{t-1}^t |t-r|^{(H-\ga)-3/2} |W_t-W_r| \, dr \ .$$
It is then readily checked that for every $\ell \geq 0$, $\ES[|R_{\tau_{k-1}+\Delta(k,\ell)}||{\cal E}_{k-1}]\leq c_{H,\ga}$, and so we can apply (\ref{eq:controlYtaum}) again (with $p=1/2$) to assert that
\begin{equation}\label{case-k=m-2}
\ES\Big[\Big(\sup_{t\in (0,1]}t^{1-\ga}|{\cal D}_{\tau_{k}-1,\tau_k}^{\tau_k}(t)|\Big)\Big|\, {\cal E}_k\Big]\le c^2_{\eta,\alpha} \ .
\end{equation}
The combination of (\ref{decompo-case-k=m}), (\ref{case-k=m-1}) and (\ref{case-k=m-2}) provide the first part of (\ref{eq:cdtional-Dtau-k-1-gamma-1}).}

\smallskip

\noindent 
{\textbf{Step 4. Asymptotic cases.} On the one hand, we can use Lemma \ref{lemma:IPP1} to obtain that for every $k\geq 1$,
$$\sup_{t\in (0,1]} t^{1-\gamma}|{\cal D}_{-\infty,0}^{\tau_k}(t)| \leq \sup_{t\in (0,1]} |{\cal D}_{-\infty,0}^{\tau_k}(t)|\leq c_H a_k^{H-1/2} \int_{-\infty}^{0} |1-r|^{-2} |W_r| \, dr\, $$
and (\ref{eq:cdtional-Dtau-k-1-gamma-0}) then follows from the general bound (\ref{eq:controlXEK}) (with $m=0$).}

\smallskip

\noindent
{On the other hand, it is not hard to see that the situation where $k=0$ can be handled with the same strategy as in Step 3, namely writing
$$
\sup_{t\in (0,1]} t^{1-\gamma}|{\cal D}_{-\infty,0}^{0}(t)| \leq \sup_{t\in (0,1]} |{\cal D}_{-\infty,-1}^{0}(t)|+ \sup_{t\in (0,1]} t^{1-\gamma}|{\cal D}_{-1,0}^{0}(t)|  
$$
and then bounding the first, resp. second, term along the arguments of Step 2, resp. Step 3, with $p=1/2$. This easily leads us to the second part of (\ref{eq:cdtional-Dtau-k-1-gamma-1}), and accordingly the proof of the lemma is achieved.}
\end{proof}

{With Lemma \ref{lemma:contDuv} in hand, we can now turn to the proof of our technical result.}

\begin{proof}[Proof of Lemma \ref{lem:step3-2}] We only prove the result for $W$, the proof for $\widetilde{W}$ being completely similar. Fix $\alpha >0$ and  $\varsigma >1$. By Remark \ref{rem:lem888}, the conditions of Lemma \ref{lemma:contDuv} are satisfied by the sequence $a_k:=c_3 \varsigma^k$, provided $c_3\geq 1$. 
Now, by decomposition \eqref{decomp:dtauk}, it holds that for every $k\geq 0$,
\begin{equation}\label{decom-proof-lemme-8.4}
\ES[\vvvert D^{{\tau_k}}(W) \vvvert_{1;\ga}  |{\cal E}_k]\le c_H \Big\{ \ES\Big[\Big(\sup_{t\in (0,1]} t^{1-\gamma}|{\cal D}_{-\infty,0}^{\tau_k}(t)|\Big)\Big|\, {\cal E}_k\Big]+\sum_{m=1}^{k} \ES\Big[\Big(\sup_{t\in (0,1]} t^{1-\gamma}|{\cal D}_{\tau_{m-1},\tau_m}^{\tau_k}(t)|\Big)\Big|\, {\cal E}_k\Big]\Big\}\ .
\end{equation}
For $k=0$, we can use (\ref{eq:cdtional-Dtau-k-1-gamma-1}) to assert that
\begin{equation}\label{k=0}
\ES\Big[\Big(\sup_{t\in (0,1]} t^{1-\gamma}|{\cal D}_{-\infty,0}^{0}(t)|\Big)\Big] \leq C^2_{\eta,\alpha} \ .
\end{equation}
For $k\geq 1$, combining the three bounds (\ref{eq:cdtional-Dtau-k-1-gamma})-(\ref{eq:cdtional-Dtau-k-1-gamma-0})-(\ref{eq:cdtional-Dtau-k-1-gamma-1}) gives that for every $p>0$,
\begin{multline*}
\ES\Big[\Big(\sup_{t\in (0,1]} t^{1-\gamma}|{\cal D}_{-\infty,0}^{\tau_k}(t)|\Big)\Big|\, {\cal E}_k\Big]+\sum_{m=1}^{k} \ES\Big[\Big(\sup_{t\in (0,1]} t^{1-\gamma}|{\cal D}_{\tau_{m-1},\tau_m}^{\tau_k}(t)|\Big)\Big|\, {\cal E}_k\Big]\\
\le \ \frac{C^1_{\eta,\alpha,p}}{(c_3\varsigma^k)^{1/2-H}} \sum_{m=0}^{k-1} \frac{1}{\eta^{(k-m)/p}}+C^2_{\eta,\alpha}\ \le \ \frac{C^{1,1}_{\eta,\alpha,p}}{c_3^{1/2-H}} \bigg(\frac{\varsigma^{H-1/2}}{\eta^{1/p}}\bigg)^k+C^2_{\eta,\alpha}\ .
\end{multline*}
At this point, and since $\varsigma >1$, we can pick $p=p(\varsigma,\eta)>0$ such that
$$\frac{\varsigma^{H-1/2}}{\eta^{1/p}}\le 1 \ ,$$
which entails that
\begin{equation}\label{k-larger-1}
\ES\Big[\Big(\sup_{t\in (0,1]} t^{1-\gamma}|{\cal D}_{-\infty,0}^{\tau_k}(t)|\Big)\Big|\, {\cal E}_k\Big]+\sum_{m=1}^{k} \ES\Big[\Big(\sup_{t\in (0,1]} t^{1-\gamma}|{\cal D}_{\tau_{m-1},\tau_m}^{\tau_k}(t)|\Big)\Big|\, {\cal E}_k\Big]\le \frac{C^{1,2}_{\eta,\alpha,\varsigma}}{c_3^{1/2-H}}+C^2_{\eta,\alpha}\ .
\end{equation}
By injecting (\ref{k=0})-(\ref{k-larger-1}) into (\ref{decom-proof-lemme-8.4}) and setting $\mathbf{c}_{\mathbf{3,2}}(\alpha,\varsigma):=\max\big( 1,(C^{1,2}_{\eta,\alpha,\varsigma})^{1/(1/2-H)}\big)$, we can conclude that for every $c_3\geq \mathbf{c}_{\mathbf{3,2}}(\alpha,\varsigma)$ and every $k\geq 0$, 
$$\ES[\vvvert D^{{\tau_k}}(W) \vvvert_{1;\ga}  |{\cal E}_k]\le c_H  \{ 1+C^2_{\eta,\alpha} \} \ ,$$
which corresponds to the desired estimate.
 \end{proof}

\subsection{Proof of Lemma \ref{lem:step3-1}}\label{subsec:lem:step3-1}
The argument is based on a combination of the Lyapunov control established in Theorem \ref{theo:lyapou} and the properties of the noise shown in the previous section. Since the arguments are identical for $Y$ and $\wt{Y}$, we only prove the statement for $Y$.

\smallskip

\noindent
First, set $\rho:=e^{-C_2/2}$ where $C_2$ is defined in Hypothesis ${\bf (H2)}$. By Theorem \ref{theo:lyapou}, we know that there exists a constant $c$ (depending only on $\gamma$) such that for every $t\in\mathbb{R}_+$,
\begin{equation}
|Y_{t+1}|^2 \leq  \rho | Y_t|^2+c \big\{1+\lVert \mathbf{X}^{(t)}\rVert_\ga^\la\big\} 
 \ , \quad 
\end{equation}
where we have set $\la:=\frac{8}{3\ga-1}$, $X^{(t)}_s:=X_{t+s}-X_t$ ($s\ge0$) and $\lVert \mathbf{X}^{(t)}\rVert_\ga:=\lVert \mathbf{X}^{(t)}\rVert_{\ga,[0,1]}$. Accordingly, for every $p\in (0,2]$, we get that
\begin{equation}\label{eq:defrho}
| Y_{t+1}|^p \leq \rho^{p/2} | Y_t|^p+c^{p/2} \big\{1+\lVert\mathbf{X}^{(t)}\rVert_\ga^{\la p/2}\big\} \ .
\end{equation}
A straightforward induction procedure then yields 
$$| Y_{\tau_k}|^p\le (\rho^{p/2})^{\Delta \tau_k} |Y_{\tau_{k-1}}|^p+c^{p/2}\sum_{\ell=0}^{\Delta \tau_k-1}(\rho^{p/2})^{\Delta \tau_k  -\ell}(1+{\lVert\mathbf{X}^{(\tau_{k-1}+\ell)}\rVert_\ga^{\la p/2}}) \ .$$
At this point, note that by Remark \ref{rem:lem888}, we can rely on the existence of a parameter $\eta>0$ (depending only on $\alpha$) such that for every $k\ge0$  and $K>0$, $\PE({\cal E}_k|{\cal E}_{k-1})\ge\eta$, and in particular $\ES[ |Y_{\tau_{k-1}}|^p|{\cal E}_k]\le \eta^{-1}\ES[ |Y_{\tau_{k-1}}|^p|{{\cal E}_{k-1}}]$. Therefore, for all $p\in (0,2]$ and $c_3\ge \frac{\log ({\eta}/2)}{\log \rho^{p/2}}$, we have, due to $\Delta \tau_k \geq c_3$, 
$$(\rho^{p/2})^{\Delta \tau_k}\leq (\rho^{p/2})^{c_3}\le \frac{{\eta}}{2}\ ,$$
and so
\begin{eqnarray*}
\ES[| Y_{\tau_k}|^p|{\cal E}_k]&\le& \frac{{\eta}}{2} \ES[|Y_{\tau_{k-1}}|^{p}|{\cal E}_k]+c^{p/2}\sum_{\ell=0}^{+\infty}\rho^{\ell}+c^{p/2}\ES\left[\sum_{\ell=0}^{\Delta \tau_k}(\rho^{p/2})^{\Delta \tau_k  -\ell}\lVert\mathbf{X}^{(\tau_{k-1}+\ell)}\rVert_\ga^{\la p/2}\Big|{\cal E}_k\right]\\
&\le& \frac{1}{2} \ES[ |Y_{\tau_{k-1}}|^p|{\cal E}_{k-1}]+\frac{c^{p/2}}{1-\rho}+c^{p/2}\ES\left[\sum_{\ell=0}^{\Delta \tau_k}(\rho^{p/2})^{\Delta \tau_k  -\ell}\lVert\mathbf{X}^{(\tau_{k-1}+\ell)}\rVert_\ga^{\la p/2}\Big|{\cal E}_k\right] \ ,
\end{eqnarray*}
which, by induction, entails that
$$\sup_{k\ge0}\ES[| Y_{\tau_k}|^p|{\cal E}_k]\le \ES[| Y_{\tau_0}|^p|{\cal E}_0]+2{C}_{\rho,p}=\ES[| Y_{0}|^p]+2{C}_{\rho,p}\ ,$$
where 
\begin{equation}\label{eq:crhorho}
C_{\rho,p}:=\frac{c^{p/2}}{1-\rho}+c^{p/2}\ES\left[\sum_{\ell=0}^{\Delta \tau_k}(\rho^{p/2})^{\Delta \tau_k  -\ell}\lVert\mathbf{X}^{(\tau_{k-1}+\ell)}\rVert_\ga^{\la p/2}\Big|{\cal E}_k\right] \ .
\end{equation}
Let us recall here that we have assumed the existence of a parameter $r>0$ such that $\ES[| Y_{0}|^r]<+\infty$ (for $\wt{Y}_0$, one even knows that $\ES[| \wt{Y}_{0}|^p]<+\infty$ for every $p>0$, since the invariant measure has moments of any order). The conclusion now comes from the result of Proposition \ref{prop:H2gamma}  below.

\

 \begin{prop}\label{prop:H2gamma} For all $\alpha \in (0,H)$ and $\varsigma >1$, there exists a constant $\bar{\mathbf{c}}(\alpha,\varsigma) >0$ such that for all $K>0$ and $c_3\geq \bar{\mathbf{c}}(\alpha,\varsigma)$, calibrating the scheme along the formulas in (\ref{best-calib}) entails that for all $q\in (0,\frac14]$ and $\rho\in (0,1)$, 
$$ \sup_{k\ge1} \ES\left[\sum_{\ell=0}^{\Delta \tau_k}\rho^{\Delta \tau_k  -\ell}\lVert\mathbf{X}^{(\tau_{k-1}+\ell)}\rVert_\ga^{q}\Big|{\cal E}_k\right]<+\infty \ .$$
\end{prop}

\begin{proof} Let us again recall the existence of a parameter $\eta>0$ (depending only on $\alpha$) such that for every $k\ge0$ and $K>0$, $\PE({\cal E}_k|{\cal E}_{k-1})\ge\eta$. Then, by using the general bound \eqref{eq:controlYtaum1}, one obtains that 
\begin{align*}
\ES\left[\sum_{\ell=0}^{\Delta \tau_k}\rho^{\Delta \tau_k -\ell}\lVert\mathbf{X}^{(\tau_{k-1}+\ell)}\rVert_\ga^q\Big|{\cal E}_k\right]
&\le c_{\eta,\al}\cdot \sup_{m\ge1}\sum_{\ell=0}^{\Delta(k,m)}\rho^{\Delta(k,m)  -\ell}\ES\left[\lVert\mathbf{X}^{(\tau_{k-1}+\ell)}\rVert_\ga^{2q}\big|{\cal E}_{k-1}\right]^{\frac 12}\\
& \le \frac{c_{\eta,\al}}{1-\rho}\cdot\sup_{\ell,k}\ES\left[\lVert\mathbf{X}^{(\tau_{k-1}+\ell)}\rVert_\ga^{1/2}\big|{\cal E}_{k-1}\right]^{2q}\ .
\end{align*}
Secondly, by Corollary \ref{prop:controlnormXDH}, we know that
\begin{equation}
\|\mathbf{X}^{(\tau_{k-1}+\ell)}\|_{\ga} \le c(1+\|\mathbf{Z}^{(\tau_{k-1}+\ell)}\|_{\ga}^2+\vvvert D^{(\tau_{k-1}+\ell)}\vvvert_{1;\ga}^2) \ ,
\end{equation}
for some constant $c$ that depends only on $\ga$, and so 
$$\sup_{\ell,k}\ES\left[\lVert\mathbf{X}^{(\tau_{k-1}+\ell)}\rVert_\ga^{1/2}|{\cal E}_{k-1}\right]\le 
c^{1/2} \big( 1+\sup_{\ell,k}\ES[\|\mathbf{Z}^{(\tau_{k-1}+\ell)}\|_{\ga}|{\cal E}_{k-1}]+\sup_{\ell,k}\ES[\vvvert D^{(\tau_{k-1}+\ell)}\vvvert_{1;\ga}|{\cal E}_{k-1}]\big) \ .$$
Owing to the stationarity and the independence of the Brownian increments, it is clear that 
$$\ES[\|\mathbf{Z}^{(\tau_{k-1}+\ell)}\|_{\ga}|{\cal E}_{k-1}]=\ES[\|\mathbf{Z}\|_{\ga}] \ ,$$
where the latter expectation is known to be finite (see \cite[Theorem 15.33]{FV-bk}).

\smallskip 

Then, similarly to \eqref{decomp:dtauk}, one has the decomposition 
$$(D^{(\tau_{k-1}+\ell)})'(t)=\alpha_H\left(H-\frac 12\right)\left({\cal D}_{-\infty,\tau_{k-1}}^{\tau_{k-1}+\ell}(t)+{\cal D}_{\tau_{k-1},\tau_{k-1}+\ell}^{\tau_{k-1}+\ell}(t)\right) \ ,$$
On the one hand, the fact that, for any $c_3$ large enough (depending on $\alpha$ and $\varsigma$), the quantity
$$\ES\left[\sup_{t\in[0,1]}t^{1-\gamma}\big|{\cal D}_{-\infty,\tau_{k-1}}^{\tau_{k-1}+\ell}(t)\big| \Big|{\cal E}_{k-1}\right]$$
is uniformly bounded in $\ell$ and $k$ can be shown with similar arguments as those in the proof of Lemma \ref{lem:step3-2}. To be more specific, the idea is to start from a similar decomposition as the one in \eqref{decomp:dtauk} and to control each term with the help of Lemma \ref{lemma:IPP1}. Since the right-hand term in the latter lemma decreases with $\tau$, the dependency in $\ell$ can be managed as follows: for all $\ell \geq 0$ and $u<v$, 
$$\left|{\cal D}_{u,v}^{\tau_{k-1}+\ell}(t)\right|\le c_H\left((t+\tau_{k-1}-u)^{H-\frac 32}\left|W_v-W_u\right|+\int_u^v(t+\tau_{k-1}-r)^{H-\frac 52}  \left|W_v-W_r\right|dr\right)\ ,$$
and from here we can mimick the arguments of the proof of Lemma \ref{lemma:contDuv}.\smallskip

\noindent As far as the process ${\cal D}_{\tau_{k-1},\tau_{k-1}+\ell}^{\tau_{k-1}+\ell}$ is concerned, we can use the independency and stationary properties of the Brownian motion, together with the bound of Lemma \ref{lemma:IPP1}, to assert that for every $\ell\ge 1$,
$$\ES\left[\sup_{t\in[0,1]}t^{1-\gamma}\big|{\cal D}_{-\infty,\tau_{k-1}}^{\tau_{k-1}+\ell}(t)\big|\Big|{\cal E}_{k-1}\right]\le 
c_H\ES\left[\sup_{t\in[0,1]}t^{1-\gamma}\left((t+\ell)^{H-\frac 32}|W_\ell|+ \int_0^{\ell} (t+\ell-r)^{H-\frac 52} |W_r|dr\right)\right]\ .$$
For the first term, it is enough to observe that $t^{1-\gamma}(t+\ell)^{H-\frac 32}|W_\ell|\le \ell^{-\frac 12}|W_\ell|$. For the second term, we can write
\begin{align*}
t^{1-\gamma}\int_0^{\ell} (t+\ell-r)^{H-\frac 52} |W_r-W_\ell|dr&\le \int_{0}^{\ell-1} (\ell-r)^{H-2} (\ell-r)^{-\frac 12}|W_r-W_\ell| dr\\
&+\int_{\ell-1}^{\ell}(\ell-r)^{H-1-\gamma}(\ell-r)^{-\frac 12}|W_r-W_\ell| dr\ ,
\end{align*}
where we have used the fact that for every $r\in[\ell-1,\ell]$ and every $t\in(0,1]$, 
$$t^{1-\gamma} (t+\ell-r)^{H-2}\le (t+\ell-r)^{H-1-\gamma}\le (\ell-r)^{H-1-\gamma}\ .$$
The uniform boundedness of $\ES\left[\sup_{t\in[0,1]}t^{1-\gamma}\big|{\cal D}_{-\infty,\tau_{k-1}}^{\tau_{k-1}+\ell}(t)\big|\Big|{\cal E}_{k-1}\right]$ follows immediately, and this achieves the proof of our assertion.
\end{proof}

\bigskip

{ {\bf Acknowledgements: }We are grateful to Martin Hairer for fruitful advices, and to Peter K. Friz for bringing our attention to the references behind Lemma \ref{lem:justifd}$(iii)$.  }

\appendix

\section{Singular paths and canonical lift}

Let us recall that the space $\mathcal{E}_\ga^2([0,1];\R^d)$, as well as the notation $\vvvert f\vvvert_{1;\ga}$, have been introduced in Section \ref{sec:decompo-fbm-ter}. Besides, let us denote by $\cac^1([0,1];\R^d)$ the space of differentiable $\R^d$-valued paths on $[0,1]$ with continuous derivative.

\begin{proposition}\label{prop:lift-sum}
Let $z\in \cac_1^\ga([0,1];\R^d)$ be a path that can be canonically lifted into a rough path $\mathfrak{L}(z)$, in the sense of Definition \ref{defi-canonic-rp}, and let $g\in \mathcal{E}^2_\ga([0,1];\R^d)$, resp. $g\in \cac^1([0,1];\R^d)$. Then $z+g$ can be canonically lifted into a rough path $\mathfrak{L}(z+g)$ and it holds that
\begin{equation}\label{boun-diff-levy-areas}
\cn[\mathfrak{L}(z+g)^{\mathbf{2}}-\mathfrak{L}(z)^{\mathbf{2}};\cac_{2,\ga}^{2\ga,1+\ga}([0,1];\R^{d,d})]\leq c_\ga \big\{ 1+\vvvert g\vvvert_{1;\ga}^2+\cn[z;\cac_1^\ga([0,1];\R^d)]^2\big\} \ ,
\end{equation}
resp. 
\begin{equation}\label{boun-diff-levy-areas-2}
\cn[\mathfrak{L}(z+g)^{\mathbf{2}}-\mathfrak{L}(z)^{\mathbf{2}};\cac_{2}^{1+\ga}([0,1];\R^{d,d})]\leq c_\ga \big\{  1+\cn[g;\cac^1([0,1];\R^d)]^2+\cn[z;\cac_1^\ga([0,1];\R^d)]^2\big\} \ ,
\end{equation}
for some constant $c_\ga$ that depends only on $\ga$.
\end{proposition}

\

The two following results, which are extensively used in our analysis, are immediate consequences of (\ref{boun-diff-levy-areas}) and (\ref{boun-diff-levy-areas-2}).
\begin{corollary}\label{coro:identif-sol}
Let $z\in \cac_1^\ga([0,1];\R^d)$ be a path that can be canonically lifted into a rough path $\mathfrak{L}(z)$ and let $g\in \mathcal{E}^2_\ga([0,1];\R^d)$, resp. $g\in \cac^1([0,1];\R^d)$. Then, in the setting of Definition \ref{def:general-sol} (with $\beta:=\ga$, resp. $\beta=1$), a path $y:[0,1]\to V$ is a solution of
$$
dy_t =B(y_t) \, dh_t+\varSigma(y_t) \, d\mathfrak{L}(z+g)_t \quad , \quad y_{0}=v_0 \ ,
$$
if and only if $y$ is a solution of
$$
dy_t =\big[ B(y_t) \, dh_t+\varSigma(y_t) \, dg_t\big]+\varSigma(y_t) \, d\mathfrak{L}(z)_t \quad , \quad y_{0}=v_0 \ .
$$
\end{corollary}

\

\begin{corollary}\label{prop:controlnormXDH}
Let $z\in \cac_1^\ga([0,1];\R^d)$ be a path that can be canonically lifted into a rough path $\mathfrak{L}(z)$ and let $g\in \mathcal{E}^2_\ga([0,1];\R^d)$. Then it holds that
\begin{equation}
\lVert \mathfrak{L}(z+g) \rVert_{\ga;[0,1]} \leq c_\ga \big\{  1+\lVert \mathfrak{L}(z) \rVert_{\ga;[0,1]}^2+\vvvert g\vvvert_{1;\ga}^2\big\} \ ,
\end{equation}
for some constant $c_\ga$ that depends only on $\ga$.
\end{corollary}

\

We will only prove Proposition \ref{prop:lift-sum} in the situation where $g\in \mathcal{E}^2_\ga([0,1];\R^d)$, but the proof when $g\in \cac^1([0,1];\R^d)$ could be derived from the very same arguments.
\begin{lemma}
Let $g\in \mathcal{E}^2_\ga([0,1];\R^d)$ and denote by $g^n$ the linear interpolation of $g$ along the dyadic partition $\mathcal{P}_n$ of $[0,1]$. Then it holds that
\begin{equation}\label{unif-bou-g-n}
\sup_n \sup_{t\in (0,1]\backslash \mathcal{P}_n} t^{1-\ga} |(g^n)'_t| \lesssim \vvvert g\vvvert_{1;\ga}
\end{equation}
and for every $0<\ga'<\ga$,
\begin{equation}\label{unif-bou-g-n-g}
\sup_{t\in (0,1]\backslash \mathcal{P}_n} t^{1-\ga'} |(g^n-g)'_t| \lesssim \vvvert g\vvvert_{2;\ga} 2^{-n(\ga-\ga')} \ .
\end{equation}
\end{lemma}

\begin{proof}
Pick $t\in (t_i^n,t_{i+1}^n)$, for some $i=0,\ldots,2^n$. One has
$$t^{1-\ga} |(g^n)'_t| =\frac{t^{1-\ga}}{t_{i+1}^n-t_i^n}|g_{t_{i+1}^n}-g_{t_i^n}| \leq \vvvert g\vvvert_{1;\ga}\,  t^{1-\ga}\int_0^1 \frac{dr}{(t_i^n+r(t_{i+1}^n-t_i^n))^{1-\ga}} \ .$$
If $i=0$, then $t\leq t_{i+1}^n-t_i^n$ and so $t^{1-\ga} |(g^n)'_t|\leq \vvvert g\vvvert_{1;\ga} \int_0^1 \frac{dr}{r^{1-\ga}}$. If $i\geq 1$, then $\frac{t}{2} \leq \frac{i+1}{2^{n+1}} \leq \frac{i}{2^n}=t_i^n$, and so $t^{1-\ga} |(g^n)'_t|\leq \vvvert g\vvvert_{1;\ga}(t/t_i^n)^{1-\ga} \leq \vvvert g\vvvert_{1;\ga}2^{1-\ga}$, which completes the proof of (\ref{unif-bou-g-n}). 

\smallskip

\noindent
For (\ref{unif-bou-g-n-g}), note first that if $i=0$, then $t\leq 2^{-n}$ and so by (\ref{unif-bou-g-n}) we get in this case
$$t^{1-\ga'}|(g^n-g)'_t|\leq 2^{-n(\ga-\ga')} \big\{ t^{1-\ga} |(g^n)'_t|+t^{1-\ga}|g'_t|\big\} \lesssim \vvvert g\vvvert_{1;\ga} 2^{-n(\ga-\ga')} \ .$$
If $i\geq 1$, then  
\begin{eqnarray*}
t^{1-\ga'}|(g^n-g)'_t|& =& \frac{t^{1-\ga'}}{t_{i+1}^n-t_i^n}\bigg|\int_{t_i^n}^{t_{i+1}^n}\{g'_r-g'_t\} \, dr\bigg|\\
&\lesssim & \vvvert g\vvvert_{2;\ga}\frac{t^{1-\ga'}}{t_{i+1}^n-t_i^n}\int_{t_i^n}^{t_{i+1}^n} \frac{|t-u|}{(t_i^n)^{2-\ga}} \, du\\
&\lesssim & \vvvert g\vvvert_{2;\ga}\bigg( \frac{t}{t_i^n}\bigg)^{1-\ga'} \bigg( \frac{t_{i+1}^n-t_i^n}{t_i^n}\bigg)^{1-(\ga-\ga')} 2^{-n(\ga-\ga')} \ .
\end{eqnarray*}
As above, we can conclude by using the fact that in this case, one has $\max\big( \frac{t}{2},t_{i+1}^n-t_i^n \big) \leq t_i^n$.
\end{proof}

\

\begin{proof}[Proof of Proposition \ref{prop:lift-sum}]
Denote by $z^n$, resp. $g^n$, the linear interpolation of $z$, resp. $g$, along the dyadic partition $\mathcal{P}_n$. By (\ref{unif-bou-g-n-g}), the convergence of $g^n$ to $g$ (and accordingly the convergence of $z^n+g^n$ to $z+g$) in $\cac_1^{\ga'}([0,1];\R^d)$ is immediate, since 
$$|\delta (g^n-g)_{st}|\leq \int_s^t |(g^n-g)'_u| \, du \lesssim 2^{-n(\ga-\ga')} (t-s)\int_0^1 \frac{dr}{(s+r(t-s))^{1-\ga'}}\lesssim 2^{-n(\ga-\ga')} (t-s)^{\ga'} \ .$$
Then, by setting $x:=z+g$ and using the notation (\ref{approxi-levy-area}), we have the following readily-checked decomposition
\begin{equation}\label{diff-levy-areas}
\mathbf{x}^{\mathbf{2},n}_{st}-\mathbf{z}^{\mathbf{2},n}_{st}=\int_s^t (\delta z^n)_{su} \otimes dg^n_u+\bigg( \int_s^t (\delta z^n)_{ut} \otimes dg^n_u\bigg)^\ast+\int_s^t (\delta g^n)_{su} \otimes dg^n_u \ .
\end{equation}
Now consider the integral $\int_s^t (\delta z)_{su} \otimes dg_u$, which, due to the regularity of $g$, can be interpreted in the classical Lebesgue sense, and use (\ref{unif-bou-g-n})-(\ref{unif-bou-g-n-g}) to assert that
\begin{eqnarray*}
\lefteqn{\bigg| \int_s^t (\delta z^n)_{su} \otimes dg^n_u-\int_s^t (\delta z)_{su} \otimes dg_u \bigg|}\\
 &\leq &\int_s^t |\delta(z^n-z)_{su}| \otimes |dg^n_u|+\int_s^t |(\delta z)_{su}| \otimes |d(g^n-g)_u|\\
&\lesssim & \cn[z^n-z;\cac_1^{\ga'}([0,1];\R^d)] \int_s^t \frac{|u-s|^{\ga'}}{u^{1-\ga}} \, du+\cn[z;\cac_1^\ga([0,1];\R^d)] 2^{-n(\ga-\ga')} \int_s^t \frac{|u-s|^{\ga}}{u^{1-\ga'}} \, du\\
&\lesssim & |t-s|^{\ga+\ga'}  \big\{\cn[z^n-z;\cac_1^{\ga'}([0,1];\R^d)] +2^{-n(\ga-\ga')} \big\}\int_0^1 \frac{dr}{r^{1-(\ga+\ga')}} \ .
\end{eqnarray*}
We can treat the two other summands in (\ref{diff-levy-areas}) along the same lines, which leads us to the desired conclusion, namely $\cn[\mathbf{x}^{\mathbf{2},n}-\mathbf{z}^{\mathbf{2},n};\cac_2^{2\ga'}([0,1];\R^d)] \to 0$ as $n\to \infty$. We even get the explicit description
$$\mathfrak{L}(z+g)^{\mathbf{2}}_{st}-\mathfrak{L}(z)^{\mathbf{2}}_{st}=\int_s^t (\delta z)_{su} \otimes dg_u+\bigg( \int_s^t (\delta z)_{ut} \otimes dg_u\bigg)^\ast+\int_s^t (\delta g)_{su} \otimes dg_u \ .$$
With this decomposition in hand, it is now easy to exhibit the bound (\ref{boun-diff-levy-areas}): for instance, for every $0<s<t$,
\begin{eqnarray*}
\lefteqn{\bigg| \int_s^t (\delta z)_{su} \otimes dg_u \bigg|}\\
 &\leq &\vvvert g\vvvert_{1;\ga} \cn[z;\cac_1^\ga([0,1];\R^d)]\,  \int_s^t \frac{|u-s|^\ga}{u^{1-\ga}} \, du \\
&\leq & \vvvert g\vvvert_{1;\ga} \cn[z;\cac_1^\ga([0,1];\R^d)]\, |t-s|^{1+\ga}  \int_0^1 \frac{r^\ga}{(s+r(t-s))^{1-\ga}} \, dr\\
&\leq & \vvvert g\vvvert_{1;\ga} \cn[z;\cac_1^\ga([0,1];\R^d)]\, \min\bigg(|t-s|^{2\ga}\int_0^1\frac{dr}{r^{1-2\ga}} , s^{\ga-1} |t-s|^{1+\ga} \int_0^1 r^\ga \, dr\bigg) \ .
\end{eqnarray*}
\end{proof}

\section{Proof of Lemma \ref{lem:discr-sewing}}\label{proofoflem:discr-sewing}

The argument relies on the algorithm introduced in \cite[Section 6]{Deya} and which aims at ``removing the points one by one'' between $t_p$ and $t_{q+1}$ in a tricky way. First, just as in \cite[Section 3.1]{Deya}, and given any (not necessarily uniform) subpartition $\Pi$ of $\mathcal{P}_n$, we define the path $G^{\Pi}$ as follows: for every $s\leq t\in \mathcal{P}_n$,
$$G^{\Pi}_{st}:=\begin{cases}
0 & \text{if} \ (s,t) \cap \Pi=\emptyset\\
(\der G)_{sut} & \text{if} \ (s,t) \cap \Pi =u\\
G_{st}-G_{s\tilde{t}_1}-\sum_{k=1}^{\ell-1}G_{\tilde{t}_{k}\tilde{t}_{k+1}}-G_{\tilde{t}_\ell t} & \text{if} \ (s,t) \cap \Pi=\{\tilde{t}_1,...,\tilde{t}_\ell \}
\end{cases} \ .$$
With this notation, if $s=t_p$ and $t=t_{q+1}$, one has in particular
\begin{equation}\label{decompo-m-n}
G_{st}=G^{\llbracket s,t\rrbracket}_{st}+\sum_{i=p}^q G_{t_it_{i+1}} \ .
\end{equation}
As far as the sum is concerned, we have on the one hand, since $\mu_1\geq 1$,
$$
s^{1-\la} \big\| \sum_{i=p}^q G_{t_it_{i+1}} \big\| \leq \cm_{\la}^{\al,\mu_1}\big[G;\llbracket s,t\rrbracket\big] \cdot \sum_{i=p}^q |t_{i+1}-t_i|^{\mu_1}\leq \cm_{\la}^{\al,\mu_1}\big[G;\llbracket s,t\rrbracket\big]\cdot |t-s|^{\mu_1} \ ,
$$
and on the other hand
$$
\big\| \sum_{i=p}^q  G_{t_it_{i+1}} \big\|\leq\cm_{\la}^{\al,\mu_1}\big[G;\llbracket s,t\rrbracket\big]\cdot \bigg\{ \lln t_{p+1}-s\rrn^\al+ \sum_{i=p+1}^q  t_i^{\la-1} |t_{i+1}-t_i|^{\mu_1}\bigg\} \ ,
$$
with
$$
\sum_{i=p+1}^q  t_i^{\la-1} |t_{i+1}-t_i|^{\mu_1} = \frac{1}{2^{n(\la+\mu_1-1)}} \sum_{i=p+1}^q \frac{1}{i^{1-\la}}\lesssim \frac{1}{2^{n(\la+\mu_1-1)}} |q+1-p|^\la \lesssim |t-s|^\la \ .
$$

\smallskip

\noindent
Going back to (\ref{decompo-m-n}), it remains us to bound $\|G^{\llbracket s,t\rrbracket}_{st}\|$. For the sake of clarity, let us temporarily change the notation by setting, for $s,t$ fixed as above,
\begin{equation}\label{notation-times}
t_k:=t-\frac{k}{2^n} \ , \ k=0,\ldots,N \ , \quad \text{where} \ N:=2^n(t-s)\ \  (\, =q+1-p\, ) \  .
\end{equation}
We make this (unnatural) choice to ``reverse'' the time, that is to consider a decreasing function $k\mapsto t_k$, in a such a way that the below notations will be consistent with those of \cite[Section 6]{Deya} (and especially those of \cite[Proposition 6.2]{Deya}). Consider indeed the algorithm described in \cite[Section 6]{Deya} to remove one by one the points between $0$ and $N$, and accordingly the points of $\mathcal{P}_n$ between $s$ and $t$ (just use the transformation (\ref{notation-times}) to connect one with the other). Denote by $(\Pi^m)_{m=0,\ldots,N-1}$ the decreasing sequence of partitions of $\llbracket s,t\rrbracket$ that is associated with this algorithm. With the notations of \cite[Section 6]{Deya}, it is readily checked that
$$G^{\Pi^m}_{st}-G^{\Pi^{m+1}}_{st}=(\der G)_{t_{k_m^+}t_{k_m}t_{k_m^-}} \quad , \quad G^{\Pi^0}_{st}=G^{\llbracket s,t\rrbracket}_{st} \quad , \quad G^{\Pi^{N-1}}_{st}=0 \ ,$$
and so
\begin{equation}\label{decompo-associated}
G^{\llbracket s,t\rrbracket}_{st} =\sum_{m=1}^{N-1} (\der G)_{t_{k_m^+}t_{k_m}t_{k_m^-}} \ .
\end{equation}
Now, still with the notations of \cite[Section 6]{Deya} in mind, write
$$\sum_{m=0}^N (\der G)_{t_{k_m^+}t_{k_m}t_{k_m^-}}=\sum_{r=1}^{M-1} \bigg\{ (\der G)_{st_{k_{A_{r-1}+1}}t_{k^-_{A_{r-1}+1}}} +\sum_{m=A_{r-1}+2}^{A_r} (\der G)_{t_{k_m^+}t_{k_m}t_{k_m^-}} \bigg\} \, $$
and so
\begin{eqnarray}
\lefteqn{\big\|\sum_{m=0}^N (\der G)_{t_{k_m^+}t_{k_m}t_{k_m^-}}\big\|}\nonumber\\
&\leq & \cn[\der G;\cac_{3;\la}^{\al,\mu_2}(\llbracket s,t\rrbracket)] \cdot \sum_{r=1}^{M-1} \bigg\{ |t_{k^-_{A_{r-1}+1}}-s|^\al+\sum_{m=A_{r-1}+2}^{A_r}t_{k_m^+}^{\la-1} |t_{k_m^-}-t_{k_m^+}|^{\mu_2} \  \bigg\} \ .\label{toward-prop}
\end{eqnarray}
Observe at this point that
$$|t_{k^-_{A_{r-1}+1}}-s|^\al=|t-s|^\al\cdot \bigg| 1-\frac{k^-_{A_{r-1}+1}}{N}\bigg|^\al$$
and
$$t_{k_m^+}^{\la-1} |t_{k_m^-}-t_{k_m^+}|^{\mu_2} \leq |t-s|^{\la+\mu_2-1}\cdot \frac{1}{N^{\mu_2}} \bigg|1-\frac{k^+_m}{N}\bigg|^{\la-1} |k_m^+-k_m^-|^{\mu_2} \ .$$
Going back to (\ref{toward-prop}), we get that
\begin{equation}\label{toward-prop-2}
\big\|\sum_{m=0}^N (\der G)_{t_{k_m^+}t_{k_m}t_{k_m^-}}\big\| \leq \lln t-s \rrn^\al \cn[\der G;\cac_{3;\la}^{\al,\mu_2}(\llbracket s,t\rrbracket)] \cdot Q^N_{\al,\la,\mu_2} \ ,
\end{equation}
where we have set
$$Q^N_{\al,\la,\mu_2} :=\sum_{r=1}^{M-1}\bigg\{ \bigg| 1-\frac{k^-_{A_{r-1}+1}}{N}\bigg|^\al+\frac{1}{N^{\mu_2}}\sum_{m=A_{r-1}+2}^{A_r}\bigg|1-\frac{k^+_m}{N}\bigg|^{\la-1} |k_m^+-k_m^-|^{\mu_2} \bigg\} \ . $$
Therefore, we are exactly in a position to apply \cite[Proposition 6.2]{Deya} and assert that $\sup_{N \geq 1} Q^N_{\al,\la,\mu_2} < \infty$. The combination of (\ref{decompo-associated}) and (\ref{toward-prop-2}) then gives us the desired estimate, namely
$$\| G^{\llbracket s,t\rrbracket}_{st}\| \lesssim \lln t-s \rrn^\al \cn[\der G;\cac_{3;\la}^{\al,\mu_2}(\llbracket s,t\rrbracket)]  \ .$$

\smallskip

\noindent
The estimation of $s^{1-\la} \| G^{\llbracket s,t\rrbracket}_{st}\|$ is easier. Indeed, with decomposition (\ref{decompo-associated}) in mind, we simply use the fact that the above algorithm also satisfies
$$|k_m^+-k_m^-| \leq \frac{2N}{(N-m+1)} \quad \text{for every} \ m=1,\ldots,N-1 \ ,$$
and consequently
\begin{eqnarray*}
s^{1-\la} \| G^{\llbracket s,t\rrbracket}_{st}\| &\leq &\sum_{m=1}^{N-1} t_{k_m^+}^{1-\la} \|(\der G)_{t_{k_m^+}t_{k_m}t_{k_m^-}}\|\\
& \leq & \cn[\der G;\cac_{3;\la}^{\al,\mu_2}(\llbracket s,t\rrbracket)]\cdot \sum_{m=1}^{N-1} |t_{k_m^-}-t_{k_m^+}|^{\mu_2}\ \lesssim \ \lln t-s\rrn^{\mu_2} \cn[\der G;\cac_{3;\la}^{\al,\mu_2}(\llbracket s,t\rrbracket)] \ .
\end{eqnarray*}

\bibliographystyle{plain}
\bibliography{biblio-rough-ergodicity}
\end{document}